\documentclass[11pt, reqno]{amsart}

\usepackage{amsmath,amssymb,amsthm,amsfonts,mathrsfs,bigints,paralist,graphics,graphicx,epsfig,epstopdf,cite,upgreek,xcolor,boites,caption,soul}
\usepackage{mathtools}
\mathtoolsset{showonlyrefs,showmanualtags}

\usepackage{hyperref}

\hypersetup{hidelinks}
\usepackage{placeins}
\usepackage[pagewise]{lineno}

\usepackage{geometry}
\usepackage{bbm}

\theoremstyle{plain}
\newtheorem{theorem}{Theorem}[section]
\newtheorem{corollary}[theorem]{Corollary}
\newtheorem{lemma}[theorem]{Lemma}
\newtheorem{proposition}[theorem]{Proposition}

\theoremstyle{definition}

\newtheorem{definition}[theorem]{Definition}

\theoremstyle{remark}
\newtheorem{remark}[theorem]{Remark}

\newcommand{\ep}{\varepsilon}

\newcommand*\di{\mathop{}\!\mathrm{d}}
\newcommand{\bb}[1]{\mathbb{#1}}
\newcommand{\al}[1]{\mathcal{#1}}
\newcommand{\RomanNumeralCaps}[1]
    {\MakeUppercase{\romannumeral #1}}

\allowdisplaybreaks

\begin{document}
\title[Wasserstein geometry of nonnegative measures on finite Markov chains II]{Wasserstein geometry of nonnegative measures on finite Markov chains II: Geodesic and duality formulae}

\author[Q. Mao, X. Wang, and X. Xue]{Qifan Mao, Xinyu Wang, and Xiaoping Xue$^{*}$}

\address{School of Mathematics, Harbin Institute of Technology, Harbin  150001, People's Republic of China}
\email{qifanmao@stu.hit.edu.cn}
\email{wangxinyumath@hit.edu.cn}
\email{xiaopingxue@hit.edu.cn}
\begin{abstract}
In this paper, we investigate the geodesic structure and the associated Kantorovich-type duality for a Benamou–Brenier-type transportation metric defined on the space of nonnegative measures over a finite reversible Markov chain. 
The metric is introduced through a dynamic formulation that combines transport and source costs along solutions of a nonconservative continuity equation, where mass variation is constrained to occur along a fixed strictly positive reference direction. 
We show that geodesics associated with this metric exhibit a \textit{non-locality} property: almost every time, they are supported on the whole state space, independently of the choice of endpoints. 
Moreover, along optimal curves, the source term displays a characteristic temporal profile, with mass creation occurring at early times and subsequent decay as the curve approaches the target measure. 
As an application of this property, we compare our metric with the shift–transport distance and prove that the latter is always bounded above by our metric. 
Finally, we establish a Kantorovich-type duality formula in terms of Hamilton–Jacobi subsolutions, which provides a characterization of the metric and highlights the role of the momentum associated with geodesic curves.
\end{abstract}
\thanks{$^{*}$Corresponding author. }
% 49Q22(2020–now)Optimal transportation [See also 90B06]
% 60J27(1980–now)Continuous-time Markov processes on discrete state spaces
% 35K05 — Heat equation
\maketitle

\tableofcontents
\setcounter{equation}{0}

\section{Introduction}
Moving mass on graphs and networks is ubiquitous in diffusion on discrete structures, 
    transportation and logistics, and population dynamics on contact networks \cite{lovasz1993random,chung1997spectral,farahani2013review,pastor2015epidemic}. 
In many such models, however, total mass is not conserved. 
Nodes may create or absorb mass, external inflows and outflows may act at selected locations, 
    and reaction or birth--death mechanisms may change the overall amount of mass, 
    as is standard in epidemic and population models and in reaction--diffusion systems on graphs \cite{pastor2015epidemic,zuniga2020reaction,weber2006multicomponent}. 
This raises a basic geometric question: how to compare and interpolate nonnegative measures on a graph when both transport and mass variation are intrinsic to the dynamics.

Classical optimal transport provides two guiding viewpoints. 
The Benamou--Brenier formulation turns the quadratic Wasserstein distance into a convex action minimization along time-dependent flows under a continuity equation \cite{benamou2000computational}, 
    while Otto’s calculus interprets diffusion equations as gradient flows of entropy in Wasserstein space \cite{otto2001geometry,ambrosio2005gradient}. 
On graphs, discrete analogues have been developed on the probability simplex, most notably the discrete transport metric and its geometric consequences such as entropy convexity and discrete curvature notions \cite{maas2011gradient,erbar2012ricci,mielke2013geodesic,erbar2019geometry}.
Recent development also extends the corresponding optimal transport theory to graphon and infinite graphs \cite{erbar2014gradient,warren2024gradient,CarrilloWang2025}.
These conservative formulations are tailored to probability measures and therefore do not directly capture the nonconservative mechanisms above.

In parallel, a systematic theory of unbalanced transport on nonnegative measures has emerged in the continuum, 
    including generalized Wasserstein distances, optimal entropy--transport, and the Hellinger--Kantorovich geometry, all featuring dynamic Benamou--Brenier descriptions that couple transport with reaction or source terms \cite{piccoli2014generalized,piccoli2016properties,liero2018optimal,chizat2018interpolating,kondratyev2016new}. 
For graph-based unequal-mass data, recent work provides scalable computational tools \cite{le2023scalable,keriven2023entropic}, 
    yet an intrinsic, graph-adapted geometric framework that simultaneously reflects discrete calculus and accommodates nonconservative effects remains to be developed. 
This motivates a Benamou--Brenier type transportation metric on finite graphs for general nonnegative measures, 
    designed to encode both transport along edges and mass variation driven by sources.

For this purpose, we introduced a Benamou--Brenier type metric on the space of nonnegative measure in our previous work \cite{mao2026wasserstein}. More precisely, 
    on a finite reversible Markov chain $\al X$,
    given $\mu_0,\mu_1$ in the set of all nonnegative measures $\al M$, we introduce their Benamou--Brenier type distance by
    \begin{equation*}
    \begin{gathered}
    \al W_{p}^{a,b}(\mu_0,\mu_1)^{2}
    := \inf\left\{
    a^{2}\!\int_{0}^{1} h_t^{2}\di t
    \;+\;
    b^{2}\!\int_{0}^{1} \|\nabla \psi_t\|_{\mu_t}^{2}\di t
    \right\},
    \end{gathered}
    \end{equation*}
    where the infimum is taken over curves $t\mapsto \mu_t$ connecting $\mu_0$ to $\mu_1$, 
    potentials $\psi_t$, and scalar source rates $h_t$ such that
    \begin{equation*}
    \left\{
    \begin{aligned}
    & \dot{\mu}_t + \nabla\!\cdot\big(\,\hat{\mu}_t * \nabla \psi_t\,\big) \;=\; h_t\,p, \\
    & \mu_{t=0}=\mu_0,\qquad \mu_{t=1}=\mu_1 .
    \end{aligned}
    \right.
    \end{equation*}
Here $\|\nabla \psi_t\|_{\mu_t}$ denotes the graph--weighted flux norm associated with $\mu_t$, 
    $\hat{\mu}_t$ is the mobility corresponding to $\mu_t$, 
    and $\ast$ is the Hadamard product, namely 
    $(\hat{\mu}_t \ast \nabla \psi_t)(x,y)=\hat{\mu}_t(x,y)\,\nabla \psi_t(x,y)$.
    The parameters $a>0$ and $b>0$ balance the source and transport contributions.
And we prescribe the \emph{direction} of mass variation by a fixed strictly positive probability vector $p$.
In \cite{mao2026wasserstein}, 
    we proved that on the strictly positive cone, 
    the induced Riemannian structure
    aligns the length distance with $\al W^{a,b}_p$,  
    and we identified the generalized heat equation as the entropy gradient flow, 
    with global well-posedness and exponential convergence to equilibrium. However, the geodesic and duality formulae of our Benamou--Brenier type distance need to be further explored.
\newline

For the conservative discrete dynamical transport metric on the probability simplex that originally motivated our construction, 
    Erbar, Maas, and Wirth investigated in \cite{erbar2019geometry} a basic locality problem for geodesics. 
Given a subset $\al Y\subset \al X$ and two probability measures $\rho_0,\rho_1$ supported in $\al Y$, 
    they asked whether one can find a constant-speed geodesic connecting $\rho_0$ to $\rho_1$ that remains supported in $\al Y$ for all times. 
They termed this property \emph{weak locality}, 
and proved that it holds whenever the induced subgraph $\al Y$ satisfies a suitable structural condition, 
    the \emph{retraction property}. 
    In the present nonconservative setting, where mass variation is encoded by a source term constrained to a fixed fully supported direction $p$, it is natural to ask whether an analogous weak locality property persists, despite the coupling between transport and source:
\begin{itemize}
\item[(Q1)] Does $\al W^{a,b}_p$ satisfy such a weak locality principle?
\end{itemize}
A second, complementary question concerns a dual characterization of $\al W^{a,b}_p$. In the conservative setting, Kantorovich-type duality links the distance to Hamilton--Jacobi inequalities and provides a variational certificate for optimality. In the present nonconservative framework, one may ask:
\begin{itemize}
\item[(Q2)] Can $\al W^{a,b}_p$ be characterized by a dual formula involving Hamilton--Jacobi subsolutions that govern the momentum of geodesic curves?
\end{itemize}

\vspace{0.2cm}

The purpose of this paper is to answer the above questions (Q1) and (Q2). 
Regarding (Q1), we provide a negative answer in Section \ref{sec_geodesic}.
We first show that $\al W^{a,b}_p$ defines a geodesic metric on $\mathcal M$:
    the action infimum is always attained, 
    and any pair of endpoints is joined by a constant-speed geodesic $t\mapsto\mu_t$.
We then exploit the monotonicity of the logarithmic mean appearing in the edge mobility $\hat\mu$
    and the strict positivity of the reference direction $p$
    to uncover a structure of optimal trajectories.
First, the source amplitude $t\mapsto h_t$ is monotone nonincreasing,
    so that mass is typically created at early times and then decreases as the curve approaches the target measure
    (Lemma~\ref{lem_frontload}).
Second, $(h_t)_{t\in[0,1]}$ cannot vanish on a time set of positive measure
    (Lemma~\ref{lem:no_zero_platform}).
Combining these properties, 
    we prove in Theorem~\ref{thm_a.e._not_in_boundary} that geodesics are,
    for almost every time, supported on the whole chain:
    if $(\mu_t)_{t\in[0,1]}$ is a constant-speed $\al W_p^{a,b}$-geodesic, then
    \begin{equation*}
    \operatorname{supp}(\mu_t)=\al X,
    \qquad\text{for a.e.\ }t\in(0,1).
    \end{equation*}
Since this non-locality phenomenon occurs for arbitrary endpoints, 
    it rules out any weak locality principle
    in the sense of \cite{erbar2019geometry}, thus answering (Q1) in the negative.
As an application of the strict-positivity result, 
    we compare $\al W^{a,b}_p$ with the discrete shift--transport distance
    in the spirit of Rossi--Piccoli \cite{piccoli2014generalized,piccoli2016properties},
    and obtain an explicit upper bound together with a strictness criterion
    (Theorem~\ref{thm_comparision_RP}).

We further conjecture that if both endpoints are fully supported on $\al X$, then the entire geodesic remains fully supported for all $t\in(0,1)$. Conditional on this conjecture, we establish uniqueness of geodesics in Proposition~\ref{prop_conditioned_uniqueness} and derive a weak formulation of the geodesic equations in Theorem~\ref{thm:GE-weak}.
We also include a numerical experiment illustrating the shape of geodesics in Remark~\ref{rmk_numerical}, which provides supporting evidence for the conjectured strict positivity.
\newline

Regarding (Q2), in Section \ref{sec_duality} we establish a Kantorovich-type duality for $\al W^{a,b}_p$
    in terms of Hamilton--Jacobi subsolutions.
We introduce the class $\mathsf{HJ}^{1,b}_{\al X}$ of potentials $\varphi\in \mathrm H^{1}((0,1);\mathbb R^{X})$
    satisfying the discrete Hamilton--Jacobi inequality
    \begin{equation*}
    \langle \dot\varphi_t,\mu\rangle_{\varpi}
    +\frac{1}{2b^{2}}\|\nabla\varphi_t\|_{\mu}^{2}
    \le 0,
    \qquad
    \text{for a.e.\ }t\in(0,1)\text{ and for all }\mu\in\mathcal M,
    \end{equation*}
    which serves as a graph-adapted analogue of the usual subsolution condition in the conservative theory.
We then prove in Theorem~\ref{thm_duality} the dual representation
    \begin{equation*}
    \frac12\,\al W^{a,b}_p(\mu_0,\mu_1)^{2}
    =
    \sup\left\{
    \langle \varphi_1,\mu_1\rangle_{\varpi}
    -
    \langle \varphi_0,\mu_0\rangle_{\varpi}
    -
    \frac{1}{2a^{2}}\int_{0}^{1}\langle \varphi_t,p\rangle_{\varpi}^{2}\di t
    \;\middle|\;
    \varphi\in \mathsf{HJ}^{1,b}_{\al X}
    \right\},
    \end{equation*}
    and show that the supremum restricted to $\varphi\in \mathrm C^{1}([0,1];\mathbb R^{X})$
    still satisfying the same Hamilton--Jacobi inequality.
The additional penalty term involving $\langle \varphi_t,p\rangle_{\varpi}$
    encodes the constrained source direction and is the only modification of the classical quadratic duality.
Our proof follows a convex-analytic route:
    we recast the primal action problem as a constrained convex minimization in flux--source variables,
    apply a minimax argument together with Fenchel--Rockafellar duality,
    and identify the dual constraint precisely as the Hamilton--Jacobi subsolution condition.
\newline 

The paper is organized as follows.
Section~\ref{sec_preliminaries} introduces the reversible Markov chain framework and recalls the main definitions and basic facts from \cite{mao2026wasserstein}.
In Section~\ref{sec_geodesic}, we show that $(\al M,\al W^{a,b}_p)$ is a geodesic metric space, and we analyze the resulting nonlocal features of geodesics.
Section~\ref{sec_duality} establishes the Hamilton--Jacobi-based Kantorovich duality formula for $\al W^{a,b}_p$.
Finally, Section~\ref{sec_conc} summarizes the main findings and discusses several open questions and directions for future work.

\section{Preliminaries}
\label{sec_preliminaries}

In this section, we first introduce the Markov chain and several notations, including discrete divergence, inner product, logarithmic mean, and so on. Then, we recall our Benamou--Brenier-type metric introduce in \cite{mao2026wasserstein}.

\subsection{Basic notations}\label{sec:2.1}
Throughout the paper we work with a finite reversible Markov chain and follow the notational conventions of \cite{maas2011gradient,erbar2012ricci}.
Let $\al X$ be a finite set with $|\al X|=N$, and let $K:\al X\times \al X\to \bb R_+$ be an irreducible Markov kernel.
Standard Markov chain theory ensures the existence of a unique invariant distribution $\varpi:\al X\to \bb R_+$, namely
\begin{equation*}
\varpi(x)=\sum_{y\in \al X}\varpi(y)K(y,x),
\ \forall~x\in \al X,
\qquad
\text{and}\qquad
\sum_{y\in \al X}\varpi(y)=1.
\end{equation*}
We assume that $\varpi$ is reversible for $K$, in the sense that
\begin{equation}
\label{eq_markov_reversible}
K(x,y)\varpi(x)=K(y,x)\varpi(y),
\quad \forall~x,y\in \al X.
\end{equation}
We denote by
\begin{equation*}
\al M=\al M(\al X):=\{\mu:\al X\to \bb R \mid \mu(x)\ge 0 \text{ for all }x\in \al X\}
\end{equation*}
the cone of nonnegative densities, and we write
\begin{equation*}
\al M_+=\al M_+(\al X):=\{\mu:\al X\to \bb R \mid \mu(x)>0 \text{ for all }x\in \al X\}
\end{equation*}
for its strictly positive part.

Since any real-valued map on $\al X$ is simply an array of length $N$,
we use $\bb R^{\al X}$ and $\bb R^{\al X\times \al X}$ for real-valued functions on vertices and on ordered pairs,
and freely identify them with $\bb R^{N}$ and $\bb R^{N\times N}$.
We write $[\cdot,\cdot]$ for the Euclidean inner product and $|\cdot|_2$ for the Euclidean norm.

For $\psi\in \bb R^{\al X}$ we define its $\mathrm L^1$-norm with respect to $\varpi$ by
\begin{equation*}
\label{L_1norm}
\|\psi\|_{\varpi,1}:=[|\psi|,\varpi]=\sum_{x\in \al X}\varpi(x)\,|\psi(x)|.
\end{equation*}
For $\mu\in \al M$, the quantity $\|\mu\|_{\varpi,1}$ is interpreted as the total mass of $\mu$.
The discrete gradient of $\psi$ is
\begin{equation*}
\nabla\psi(x,y):=\psi(y)-\psi(x),
\quad \forall~x,y\in \al X.
\end{equation*}
Given $\Psi:\al X\times \al X\to \bb R$, its discrete divergence is
\begin{equation*}
(\nabla\cdot \Psi)(x):=\frac12\sum_{y\in \al X}\big(\Psi(x,y)-\Psi(y,x)\big)K(x,y),
\quad \forall~x\in \al X.
\end{equation*}

We equip $\bb R^{\al X}$ and $\bb R^{\al X\times \al X}$ with the $\varpi$-weighted bilinear forms
\begin{equation*}
\begin{gathered}
\langle \varphi,\psi\rangle_{\varpi}:=\sum_{x\in \al X}\varphi(x)\psi(x)\varpi(x),\\
\langle \Phi,\Psi\rangle_{\varpi}:=\frac12\sum_{x,y\in \al X}\Phi(x,y)\Psi(x,y)K(x,y)\varpi(x),
\end{gathered}
\end{equation*}
for $\varphi,\psi\in \bb R^{\al X}$ and $\Phi,\Psi\in \bb R^{\al X\times \al X}$.
With these conventions one has the integration by parts identity
\begin{equation}
\label{eq_int_by_parts}
\langle \nabla\psi,\Psi\rangle_{\varpi}
= -\langle \psi,\nabla\cdot \Psi\rangle_{\varpi}.
\end{equation}

The logarithmic mean $\theta$ is defined for $u,v\in \bb R_+$ by
\begin{equation}
\label{eq_logarithmic_mean}
\theta(u,v):=\int_0^1 u^\xi v^{1-\xi}\di \xi
=
\left\{
\begin{aligned}
& v, && \text{if } u=v, \\
& \frac{u-v}{\log u -\log v}, && \text{if } u\neq v \text{ and } u,v>0, \\
& 0, && \text{if } u=0 \text{ or } v=0.
\end{aligned}
\right.
\end{equation}
It is nonnegative and symmetric, and it is strictly increasing in each argument on $(0,\infty)\times(0,\infty)$.
Its reciprocal admits the representation
\begin{equation}
\label{eq_logarithmic_mean_reciprocal}
\theta(u,v)^{-1}
=
\int_0^1 \frac{1}{\xi u + (1-\xi)v}\di \xi
=
\left\{
\begin{aligned}
& \frac{1}{v}, && \text{if } u=v, \\
& \frac{\log u -\log v}{u-v}, && \text{if } u\neq v \text{ and } u,v>0, \\
& \infty, && \text{if } u=0 \text{ or } v=0.
\end{aligned}
\right.
\end{equation}
Although \cite{maas2011gradient,erbar2012ricci} allow more general choices of $\theta$,
the logarithmic mean is particularly suited to entropy-related arguments.
We therefore fix \eqref{eq_logarithmic_mean} throughout, while keeping the notation $\theta$.

For $\mu\in \al M$ we define the associated mobility $\hat\mu:\al X\times \al X\to \bb R$ by
\begin{equation*}
\label{eq_hatmu}
\hat\mu(x,y):=\theta\big(\mu(x),\mu(y)\big),
\quad \forall~x,y\in \al X.
\end{equation*}
Using $\hat\mu$, we introduce the $\mu$-dependent bilinear form on $\bb R^{\al X\times \al X}$,
\begin{equation}
\label{inner_mu}
\langle \Phi,\Psi\rangle_{\mu}
:=\langle \Phi,\hat\mu*\Psi\rangle_{\varpi}
=
\frac12\sum_{x,y\in \al X}\Phi(x,y)\Psi(x,y)\hat\mu(x,y)K(x,y)\varpi(x),
\end{equation}
where $(\hat\mu*\Psi)(x,y)=\hat\mu(x,y)\Psi(x,y)$ denotes the Hadamard product.
The induced seminorm is
\begin{equation*}
\label{eq_mu_2norm}
\|\Phi\|_{\mu}:=\sqrt{\langle \Phi,\Phi\rangle_{\mu}}.
\end{equation*}

We impose an equivalence relation on $\bb R^{\al X\times \al X}$ by declaring $\Phi\sim \Psi$
if and only if $\Phi(x,y)=\Psi(x,y)$ for every pair $(x,y)$ such that $\hat\mu(x,y)K(x,y)>0$.
We denote by $\mathscr G_\mu$ the corresponding space of equivalence classes.
In the degenerate case where $\hat\mu(x,y)K(x,y)=0$ for all $(x,y)$,
the quotient $\mathscr G_\mu$ is a singleton.
When no confusion arises, we use the same symbol for an element of $\bb R^{\al X\times \al X}$
and for its class in $\mathscr G_\mu$.
Since \eqref{inner_mu} depends only on values on pairs with $\hat\mu(x,y)K(x,y)>0$,
it descends to an inner product on $\mathscr G_\mu$.
In the finite-state setting, $\mathscr G_\mu$ is therefore a Hilbert space.

Let $\mathrm L^2(\al X,\varpi)$ denote $\bb R^{\al X}$ endowed with $\langle\cdot,\cdot\rangle_{\varpi}$.
The gradient map $\nabla:\mathrm L^2(\al X,\varpi)\to \bb R^{\al X\times \al X}$
descends, via the $\mu$-dependent quotient defining $\mathscr G_\mu$,
to a well-defined linear operator $\nabla:\mathrm L^2(\al X,\varpi)\to \mathscr G_\mu$,
and we henceforth regard $\nabla\psi$ as an element of $\mathscr G_\mu$.
The negative adjoint of $\nabla$ is the $\mu$-divergence operator
$\nabla_\mu\cdot:\mathscr G_\mu\to \mathrm L^2(\al X,\varpi)$, defined by
\begin{equation*}
(\nabla_\mu\cdot \Psi)(x)
:=(\nabla\cdot(\hat\mu*\Psi))(x)
=\frac12\sum_{y\in \al X}\big(\Psi(x,y)-\Psi(y,x)\big)\hat\mu(x,y)K(x,y),
\quad \forall~x\in \al X.
\end{equation*}

\subsection{Benamou--Brenier-type metric}\label{sec:2.2}
We recall from \cite{mao2026wasserstein} the Benamou--Brenier-type distance on $\al M$,
defined as the infimum of a dynamic action over admissible curves.
Throughout we fix parameters $a,b>0$, and we also fix a reference density $p\in \al M_+$
with $\|p\|_{\varpi,1}=1$, so that $p$ is a probability density.

\begin{definition}[Benamou--Brenier-type metric on $\al M$]
\label{def_Wpab}
For $\mu_0,\mu_1\in \al M$, we set
\begin{equation}
\label{eq_def_w}
\al W^{a,b}_p(\mu_0,\mu_1)^2
:=
\inf\left\{
E^{a,b}_{\rm Quad}\big((\mu_t,\psi_t,h_t)_{t\in[0,1]}\big)
\ \middle|\
(\mu_t,\psi_t,h_t)_{t\in[0,1]}\in \mathrm{CE}_p(\mu_0,\mu_1;[0,1])
\right\}.
\end{equation}
Here $\mathrm{CE}_p(\mu_0,\mu_1;[0,1])$ denotes the collection of all triples $(\mu_t,\psi_t,h_t)_{t\in[0,1]}$ such that
\begin{itemize}
\item[$\circ$] the triple is sufficiently regular, with $\mu_t\in \mathrm{A\!C}([0,1];\al M)$,
$\psi_t:[0,1]\to \bb R^{\al X}$ measurable, and $h_t:[0,1]\to \bb R$ measurable;
\item[$\circ$] it satisfies the continuity equation, namely for a.e.\ $t\in[0,1]$ and all $x\in \al X$,
\begin{equation}
\label{eq_ce}
\dot\mu_t(x)
+
\sum_{y\in \al X}\big(\psi_t(y)-\psi_t(x)\big)\hat\mu_t(x,y)K(x,y)
=
h_t\,p(x),
\end{equation}
or, equivalently,
\begin{equation}
\label{eq_ce_discrete_calculus}
\dot\mu_t + \nabla_{\!\mu_t}\!\cdot\nabla\psi_t = h_t\,p,
\qquad
\text{or}
\qquad
\dot\mu_t + \nabla\!\cdot(\hat\mu_t*\nabla\psi_t)=h_t\,p;
\end{equation}
\item[$\circ$] it meets the endpoint constraints
\begin{equation*}
\mu_t|_{t=0}=\mu_0,
\qquad
\mu_t|_{t=1}=\mu_1.
\end{equation*}
\end{itemize}
The action functional $E^{a,b}_{\rm Quad}$ is defined by
\begin{equation}
\label{eq_def_E2_origin}
E^{a,b}_{\rm Quad}\big((\mu_t,\psi_t,h_t)_{t\in[0,1]}\big)
:=
a^2\int_0^1 h_t^2\,\di t
+
b^2\int_0^1
\frac12\sum_{x,y\in \al X}\big(\psi_t(y)-\psi_t(x)\big)^2
\hat\mu_t(x,y)K(x,y)\varpi(x)\,\di t.
\end{equation}
The label ``Quad'' refers to the quadratic time integrand.
In the discrete-calculus notation, \eqref{eq_def_E2_origin} can be rewritten as
\begin{equation}
\label{eq_def_E2_discrete_calculus}
\begin{gathered}
E^{a,b}_{\rm Quad}\big((\mu_t,\psi_t,h_t)_{t\in[0,1]}\big)
=
a^2\int_0^1 h_t^2\,\di t
+
b^2\int_0^1 \|\nabla\psi_t\|_{\mu_t}^2\,\di t
=
a^2\int_0^1 h_t^2\,\di t
+
b^2\int_0^1 \langle \nabla\psi_t,\nabla\psi_t\rangle_{\mu_t}\,\di t.
\end{gathered}
\end{equation}
\end{definition}

We verified the axioms of metric for $\al W_p^{a,b}$ on $\al M$ in \cite[Theorem 3.10]{mao2026wasserstein}.
\begin{theorem}
$(\al M,\al W^{a,b}_p)$ is a metric space.
\end{theorem}

We also had the following result for the instantaneous solvability of the continuity equation \eqref{eq_ce}-\eqref{eq_ce_discrete_calculus} as 
\cite[Lemma 3.4]{mao2026wasserstein} and \cite[Lemma 3.5]{mao2026wasserstein}.
\begin{lemma}
\label{lem_ce_solver_combined}
Fix $\mu \in \al M_+$.
The linear equation in the unknown $\psi$
\begin{equation*}
\label{eq_cenos_proto}
\nu + \nabla_\mu \cdot \nabla \psi=0
\end{equation*}
is solvable if and only if $\nu \in \{
    \psi \in \bb R ^{\al X} \mid  \sum_{x\in \al X} \psi(x) \varpi (x)=0
    \}$.
Moreover, for any arbitrary $\rho\in \mathbb{R}^{\al X}$, 
    the linear equation in the unknowns $(\nabla\psi_\rho,h_\rho)$
    \begin{equation*}
    \label{eq_tangent_identification_0}
    \rho+\nabla_{\!\mu}\!\cdot\nabla\psi_\rho=h_\rho\,p
    \end{equation*}
    has a unique solution.
We write the solution mapping as
    \begin{equation*}
    D_\mu(\rho)
    :=(\nabla\psi_\rho,h_\rho).
    \end{equation*}
\end{lemma}

\section{Geodesic geometry}
\label{sec_geodesic}
In this section, we investigate the geodesic geometry induced by our Benamou--Brenier-type metric \cite{mao2026wasserstein}.
Section \ref{sec_geoexists} develops a flux formulation of the continuity equation and of the action,
which is the main tool for proving the existence of minimizers,
and in particular yields constant-speed geodesics.
In Section \ref{sec_nonlocality}, we show that the prescribed source direction $p$ together with the monotonicity of the logarithmic mean $\theta$
forces optimal trajectories to attain positive levels across all components for most times,
leading to a pronounced non-local behavior of geodesics.
Section \ref{sec_compare} applies this non-locality to compare $\al W^{a,b}_p$ with the shift--transport metric of \cite{piccoli2014generalized},
and derives an explicit bound together with a strictness criterion.
Assuming in addition the existence of geodesics that remain fully supported on the chain for all times in $[0,1]$,
Section \ref{sec_uniqueness} proves a conditioned uniqueness statement.
Finally, under the same all-time full-support assumption, Section \ref{sec_geoeq} derives a weak form of the geodesic equation.

\subsection{$(\al M,\al W_p^{a,b})$ is a geodesic space}
\label{sec_geoexists}
In this subsection, we follow the approach of \cite{erbar2012ricci} for the existence of an optimal trajectory whose action realizes $\al W_p^{a,b}$, which turns out to be a constant-speed geodesic.
It all relies on what we call the \emph{flux-reformulation of $\al W_p^{a,b}$}, 
    which turns the admissible constraints linear and the action convex.

We shall reformulate ${\al W}_p^{a,b}$ 
    by replacing $\hat\mu\ast\nabla\psi_t$ with $V_t:[0,1]\rightarrow\mathbb R^{\al X \times \al X}$. 
Consider the modified admissible trajectory set $\mathcal{CE}_p(\mu_0,\mu_1;[0,1])$ with the new continuity equation:
    \begin{equation}
    \label{new_ce}
    \left\{
    \begin{aligned}
    & \mu_t:[0,1] \rightarrow \al M ~\text{is absolutely continuous},
    \\
    & V_t:[0,1]\rightarrow\mathbb R^{\al X \times \al X} ~\text{is measurable},
    \\
    & h_t:[0,1]\rightarrow \bb R ~\text{is measurable},
    \\
    & \left\{
    \begin{aligned}
    &\dot\mu_t+\nabla \cdot  V_t=h_tp, \qquad \text{for a.e. }t\in[0,1],\\
    &\mu_t |_{t=0} = \mu_0,\quad \mu_t |_{t=1}=\mu_1.
    \end{aligned}
    \right.
    \end{aligned}
    \right.
    \end{equation}
We also define
    \begin{equation}
    \label{def_A'}
        \begin{gathered}
            A'(\mu,V):= \frac 1 2 \sum_{x,y\in \al X}\alpha(V(x,y),\mu(x),\mu(y)) K(x,y)\varpi(x),
        \end{gathered}
    \end{equation}
    where $\alpha:\bb R \times \bb R_+^2\rightarrow \bb R \cup\{+\infty\}$ is defined by
    \begin{equation}
    \label{def_alpha}
    \alpha(v,s,t):=
    \left\{
    \begin{aligned}
    & 0, & & \text{if }\theta(s,t)=0 \text{ and }v=0, \\
    &\frac{v^2}{\theta(s,t)}, & &\text{if }\theta(s,t)\ne 0, \\
    &\infty, && \text{if } \theta(s,t)=0 \text{ and }v\ne0. 
    \end{aligned}
    \right.
    \end{equation}
With the usual convention that $v^2/0 = 0$ for $v = 0$ and $v^2/0 = \infty$ for $v \neq 0$,
it is convenient to directly write
    \begin{equation*}
    \alpha(v,s,t) = \frac{v^2}{\theta(s,t)}.
    \end{equation*}
We have the flux-actions:
    \begin{equation*}
    \begin{gathered}
    \mathfrak E^{a,b}_{\rm Quad}((\mu_t,V_t,h_t)_{t\in[0,1]}):=
    a^2
    \int_0^1 | h_t |^2 \text d t
    +
    b^2 \int_0^1 A'(\mu_t,V_t )\text d t,
    \\
    \mathfrak E^{a,b}_{\rm LinSq}((\mu_t,V_t,h_t)_{t\in[0,1]}):=
    \left(\int_0^1 \bigg[ a^2 | h_t |^2 
    +
    b^2 A'(\mu_t,V_t )\bigg]^{\frac{1}{2}}\text d t
    \right)^2.
    \end{gathered}
    \end{equation*}
As ``Quad'' refers to the quadratic time integrand,
    the subscript ``LinSq'' indicates integrating the linear amplitude first and then squaring.

The two flux action functionals yield the same infimum value.
This is a standard fact for action minimization of this type; see, for instance, \cite[Theorem 5.4]{dolbeault2009new}.
A similar proof for $E^{a,b}_{\rm Quad}$ was also given in our previous work \cite[Lemma 3.6]{mao2026wasserstein}.
We therefore omit the proof here.
\begin{lemma}
\label{lem_frakELinSq_potential_rewrite}
For any $\mu_0,\mu_1\in \al M$,
    \begin{equation*}
    \begin{aligned}
    &\inf \left\{ \mathfrak E^{a,b}_{\rm Quad}((\mu_t,V_t,h_t)_{t\in[0,1]} )\mid (\mu_t,V_t,h_t)_{t\in[0,1]}\in\mathcal{CE}_p (\mu_0,\mu_1;[0,1])\right\}
    \\
    &\qquad \qquad \qquad \qquad =
    \inf \left\{ \mathfrak E^{a,b}_{\rm LinSq}((\mu_t,V_t,h_t)_{t\in[0,1]}) \mid (\mu_t,V_t,h_t)_{t\in[0,1]}\in\mathcal{CE}_p (\mu_0,\mu_1;[0,1])\right\}.
    \end{aligned}
    \end{equation*}
\end{lemma}

The next result provides a characterization of the convexity and affine cases of $\alpha$. Proof deferred to Appendix \ref{sec_app_e}.
\begin{lemma} 
\label{lem_alpha_convex} On $\mathbb{R}\times \mathbb{R}_+^2$, the extended-valued map $\alpha$ is convex and lower semi-continuous. 
Moreover, on $\mathbb{R}\times(0,\infty)^2$, for two distinct points 
    \begin{equation*} Z_1=(v_1,s_1,t_1),\qquad Z_2=(v_2,s_2,t_2),\end{equation*} 
    the statements (\RomanNumeralCaps{1}) and (\RomanNumeralCaps{2}) are equivalent:
    \begin{enumerate}
    \item[(\RomanNumeralCaps{1})] $\alpha$ is affine on the segment $[Z_1,Z_2]$;
    \item[(\RomanNumeralCaps{2})] $v_1=v_2=0$ or $\exists\kappa>0, \kappa\neq 1$ with $Z_2=\kappa Z_1$.
    \end{enumerate}
\end{lemma}

The following result provides a reformulation of $\al W^{a,b}_p$ in terms of $\mathcal{CE}_p$ and $A'$.
A related statement was obtained by Erbar and Maas in \cite[Lemma 2.9]{erbar2012ricci}, and our argument is a slightly adapted version of theirs.
The proof is deferred to Appendix \ref{sec_app_W_rewrite}.
\begin{lemma}
\label{lem_W_rewrite}
For $\mu_0,\mu_1\in \al M$, we have 
    \begin{equation}
    \label{eq_Wpab_with_V}
    ({\al W}_p^{a,b}(\mu_0,\mu_1))^2=
    \inf_{}
    \left\{
    \mathfrak E^{a,b}_{\rm Quad}((\mu_t,V_t,h_t)_{t\in[0,1]})
    \mid
    (\mu_t,V_t,h_t)_{t\in[0,1]}\in\mathcal{CE}_p(\mu_0,\mu_1;[0,1])
    \right\}.
    \end{equation}
\end{lemma}

We now give the existence of a minimizer for the reformulated optimization problem \eqref{eq_Wpab_with_V}.
The argument follows the strategy of \cite[Theorem 3.2]{erbar2012ricci} and is presented in Appendix \ref{sec_app_minimizer}.
\begin{theorem}
\label{thm_minimizer}
For $\mu_0,\mu_1\in \al M$,
    there exists a triplet $(\mu_t,V_t,h_t)_{t\in[0,1]}\in\mathcal{CE}_p(\mu_0,\mu_1;[0,1])$  
    such that ${\al W}_p^{a,b}(\mu_0,\mu_1)^2=\mathfrak E^{a,b}_{\rm Quad}((\mu_t,V_t,h_t)_{t\in[0,1]})$.
\end{theorem}

A direct consequence of Theorem \ref{thm_minimizer} by Jensen inequality is that:
\begin{corollary}
\label{cor_geodesic}
$(\al M,\al W_p^{a,b})$ is a geodesic space. 
In particular, for any $\mu_0$ and $\mu_1$,
    any minimizer in $\mathcal{CE}_p (\mu_0,\mu_1;[0,1]) $ realizing $\al W_p^{a,b}(\mu_0,\mu_1)^2$ is a constant-speed geodesic.
\end{corollary}
\begin{proof}
Fix a pair of $\mu_0,\mu_1$.
Let $(\mu_t,V_t,h_t)_{t\in[0,1]}$ be the minimizing curve in Theorem \ref{thm_minimizer},
    then we have
    \begin{equation*}
    \al W_p^{a,b}(\mu_0,\mu_1)^2 =  \mathfrak E^{a,b}_{\rm Quad}((\mu_t,V_t,h_t)_{t\in[0,1]}) \ge  \mathfrak E^{a,b}_{\rm LinSq}((\mu_t,V_t,h_t)_{t\in[0,1]}) \ge \al W_p^{a,b}(\mu_0,\mu_1)^2,
    \end{equation*}
    where we used Jensen's inequality to compare $ \mathfrak E^{a,b}_{\rm Quad}$ and $\mathfrak E^{a,b}_{\rm LinSq}$
    and Lemma \ref{lem_frakELinSq_potential_rewrite} for the last ``$\ge$''.
Thus the equality must hold, giving $\mathfrak E^{a,b}_{\rm Quad}((\mu_t,V_t,h_t)_{t\in[0,1]}) = \mathfrak E^{a,b}_{\rm LinSq}((\mu_t,V_t,h_t)_{t\in[0,1]})$, which is
    \begin{equation*}
    \int_0^1 a^2 | h_t |^2 
    +
    b^2 A'(\mu_t,V_t ) \text d t
    =
    \left(\int_0^1 \bigg[ a^2 | h_t |^2 
    +
    b^2 A'(\mu_t,V_t )\bigg]^{\frac{1}{2}}\text d t
    \right)^2.
    \end{equation*}
For the above equality to hold,
    the integrand $a^2 | h_t |^2 + b^2 A'(\mu_t,V_t )$
    has to be constant almost always.
In particular, for almost every $t\in[0,1]$,
    \begin{equation*}
    a^2 | h_t |^2 + b^2 A'(\mu_t,V_t ) =  (\al W_p^{a,b}(\mu_0,\mu_1))^2.
    \end{equation*}

Take an interval $[s_0,s_1]\subset[0,1]$.
The restriction $(\mu_t,V_t,h_t)_{t\in[s_0,s_1]}$ is an admissible trajectory in $\al{CE}_p(\mu_{s_0},\mu_{s_1};[s_0, s_1])$.
By optimality on $[0,1]$, $(\mu_t,V_t,h_t)_{t\in[s_0,s_1]}$ also 
    realizes $(\al W_p^{a,b}(\mu_{s_0},\mu_{s_1}))^2$ on $[s_0,s_1]$,
    otherwise one could find a better trajectory between $\mu_{s_0},\mu_{s_1}$ to replace $(\mu_t,V_t,h_t)_{t\in[s_0,s_1]}$,
    contradicting the optimality of $(\mu_t,V_t,h_t)_{t\in[0,1]}$.
Thus we have
    \begin{equation*}
    \al W_p^{a,b}(\mu_{s_0},\mu_{s_1})
    =
    \int_{s_0}^{s_1} \bigg[ a^2 | h_t |^2 
    +
    b^2 A'(\mu_t,V_t )\bigg]^{\frac{1}{2}}\text d t
    =(s_1-s_0) \al W_p^{a,b}(\mu_{0},\mu_{1}).
    \end{equation*}
\end{proof}

\subsection{Non-locality of geodesics}
\label{sec_nonlocality}
At this part, we show the non-locality of optimal trajectories:
    given any endpoints $\mu_0,\mu_1$,
    the geodesic $(\mu_t)_t$ is almost always supported on the whole chain.

The argument proceeds in four steps.
First, Lemma~\ref{lem_frontload} shows it is optimal to {front–load} mass creation: increasing mass earlier never increases the action.
Second, Lemma~\ref{lem:no_zero_platform} shows that an optimal curve cannot contain a non-negligible time set with vanishing source rate $h_t$. Otherwise one can construct a strictly better competitor.
Third, Theorem~\ref{thm_a.e._not_in_boundary} shows that whenever the curve has vanishing components, the source vanishes. 
Taken in the sense of time measure, these statements imply that optimal curves stay in the interior $\al M_{+}$ for almost all times.
Finally, as a by-product, this non-locality allows us to pass from flux–minimizers to potential-minimizers, recovering $\psi$ from $V$ along the optimal curve.

Let $f\in \mathrm L^{1}([0,1];\bb R)$ and let $\lambda$ denote the Lebesgue measure on $[0,1]$.
Define the {value–distribution measure} of $f$ by
    \begin{equation*}
    \lambda_{f}:=f_{\#}\lambda\quad\text{on }\mathbb{R},
    \end{equation*}
    its right-cumulative distribution function by
    \begin{equation*}
    F_{f}(\alpha):=\lambda_{f}((\alpha,\infty))=\lambda\{t\in[0,1]:\,f(t)> \alpha\},
    \end{equation*}
    and the right–continuous \emph{decreasing rearrangement} of $f$ by
    \begin{equation*}
    f^{\ast}(u):=\inf\bigl\{\alpha\in\mathbb{R}:\, F_f(\alpha)\le u\bigr\},\qquad u\in(0,1).
    \end{equation*}
Then $f$ and $f^{\ast}$ are \emph{equimeasurable}, i.e., $f_{\#}\lambda=(f^{\ast})_{\#}\lambda$, 
    and for any measurable $S\subset[0,1]$ with $\lambda(S)=t$
    \begin{equation}
    \label{ineq_majorization}
    \int_S f(s)\di s\le \int_0^t f^*(s)\di s.
    \end{equation}
For a concise account, see \cite[§1.4.1]{grafakos2008classical}.

\begin{lemma}
\label{lem_frontload}
Let $(\mu_t,V_t,h_t)_{t\in[0,1]}\in\mathcal{CE}_p(\mu_0,\mu_1;[0,1])$ be an admissible trajectory 
    and let $h^*_t$ be the decreasing rearrangement of $h_t$.
Let $\nu_t$ be the unique absolutely continuous solution of
    \begin{equation}
    \label{eq_nu_def}
    \dot\nu_t + \nabla\!\cdot V_t \;=\; h^*_t\,p,\qquad \nu_{0}=\mu_{0}.
    \end{equation}
Then $(\nu_t,V_t, h^*_t)_{t\in[0,1]}\in \mathcal{CE}_p(\mu_0,\mu_1;[0, 1])$ and the action does not increase:
    \begin{equation}
    \label{eq_nu_action_decrease}
    \mathfrak{E}^{a,b}_{\rm Quad}((\mu_t,V_t,h_t)_{t\in[0,1]})\;\ge\;\mathfrak{E}^{a,b}_{\rm Quad}((\nu_t,V_t, h^*_t)_{t\in[0,1]}).
    \end{equation}
\end{lemma}
\begin{proof}
We split the proofs into four steps.
\vspace{0.2cm}

\noindent$\bullet$ {\textbf{Step 1. Endpoint matching.}}
By definition, $\nu_t$ solves~\eqref{eq_nu_def} with the {same} $V_t$ and the rearranged source $h^\ast_t$.  
By integration, we obtain the identities
    \begin{equation*}
    \mu_t-\mu_0=\int_0^t\bigl(h_s\,p-\nabla\!\cdot V_s\bigr)\di s,
    \qquad
    \nu_t-\mu_0=\int_0^t\bigl(h^\ast_s\, p-\nabla\!\cdot V_s\bigr)\di s.
    \end{equation*}
Thus
    \begin{equation}
    \label{eq_variation of constants}
    \nu_t-\mu_t=\Bigl(\int_0^t (h^\ast_s-h_s)\di s\Bigr)\,p.
    \end{equation}
Equimeasurability of $h_t$ and $h^\ast_t$ yields $\int_0^1 h_t\di t=\int_0^1 h^\ast_t\di t$, hence
    \begin{equation*}
    \nu_1-\mu_1=\Bigl(\int_0^1 (h^\ast_t-h_t)\di t\Bigr)\,p=0,
    \end{equation*}
    i.e. $\nu_1=\mu_1$. 
Thus $(\nu_t,V_t, h^*_t)_{t\in[0,1]}$ is admissible with the same endpoints.
\vspace{0.2cm}

\noindent$\bullet$ {\textbf{Step 2. Source action is preserved.}}
Equimeasurability also preserves $\mathrm L^2$ norms, therefore
    \begin{equation*}
    \int_0^1 |h^\ast_t|^2\di t=\int_0^1 |h_t|^2\di t.
    \end{equation*}
\vspace{0.2cm}

\noindent$\bullet$ {\textbf{Step 3. Pointwise transport comparison and integration.}}
Combining \eqref{ineq_majorization} and \eqref{eq_variation of constants}, we have, componentwise,
    \begin{equation*}
    \nu_t(x)\ge \mu_t(x)\quad\text{for a.e.\ }t\in(0,1).
    \end{equation*}
Hence \eqref{eq_logarithmic_mean} gives for a.e. $t$ and all $(x,y)$,
    \begin{equation*}
    \theta\bigl(\nu_t(x),\nu_t(y)\bigr)\ \ge\ \theta\bigl(\mu_t(x),\mu_t(y)\bigr),
    \end{equation*}
    and consequently
    \begin{equation*}
    \alpha\bigl(V_t(x,y),\nu_t(x),\nu_t(y)\bigr)
    =\frac{V_t(x,y)^2}{\theta(\nu_t(x),\nu_t(y))}
    \ \le\
    \frac{V_t(x,y)^2}{\theta(\mu_t(x),\mu_t(y))}
    =\alpha\bigl(V_t(x,y),\mu_t(x),\mu_t(y)\bigr).
    \end{equation*}
Here we adopt the convention that $v^2/0=+\infty$ for $v\neq 0$ and $v^2/0=0$ for $v=0$.
Summing with the nonnegative weights $K(x,y)\varpi(x)$ gives
    \begin{equation*}
    A'(\nu_t,V_t)\ \le\ A'(\mu_t,V_t)\quad\text{for a.e.\ }t\in(0,1).
    \end{equation*}
Integrating in time yields
    \begin{equation*}
    \int_0^1 A'(\nu_t,V_t)\di t\ \le\ \int_0^1 A'(\mu_t,V_t)\di t.
    \end{equation*}

\medskip
\noindent$\bullet$ {\textbf{Step 4. Conclusion.}}
Combining the transport comparison with the preservation of the source part,
    \begin{align*}
    \mathfrak{E}^{a,b}_{\rm Quad}((\nu_t,V_t, h^*_t)_{t\in[0,1]})
    &=
    a^{2}\!\int_0^1 |h^\ast_t|^2\di t
    +
    b^{2}\!\int_0^1 A'(\nu_t,V_t)\di t\\\
    & \le\
    a^{2}\!\int_0^1 |h_t|^2\di t
    +
    b^{2}\!\int_0^1 A'(\mu_t,V_t)\di t=
    \mathfrak{E}^{a,b}_{\rm Quad}((\mu_t,V_t,h_t)_{t\in[0,1]}),
    \end{align*}
    which is exactly~\eqref{eq_nu_action_decrease}.
\end{proof}

Consequently, the optimal source can not vanish for a positive measure of time.
\begin{lemma}
\label{lem:no_zero_platform}
For $\mu_0\neq\mu_1\in \al M$, let $(\mu_t,V_t,h_t)_{t\in[0,1]}\in\mathcal{CE}_p(\mu_0,\mu_1;[0,1])$ be an action-minimizing trajectory.
Then $h_t\neq0$ for a.e. $t\in[0,1]$.
\end{lemma}
\begin{proof}
Let $\lambda$ be the Lebesgue measure on $[0,1]$
    and assume by contradiction that $\lambda([h_t=0])>0$.
Here $[h_t=0]=\{t\in[0,1]\mid h_t=0\}$.
Then, we can construct some feasible trajectories $(\tilde\nu^\varepsilon_t,V_t,\tilde h^\varepsilon_t)_{t\in[0,1]}$ such that    
    \begin{equation*}
    \mathfrak{E}^{a,b}_{\rm Quad}((\tilde\nu^\varepsilon_t,V_t,\tilde h^\varepsilon_t)_{t\in[0,1]})
    <
    \mathfrak{E}^{a,b}_{\rm Quad}((\mu_t,V_t,h_t)_{t\in[0,1]}).
    \end{equation*}
This is contradictory to the optimality of $(\mu_t,V_t,h_t)_{t\in[0,1]}$.  
We split the proof into three steps.
\vspace{0.2cm}

\noindent$\bullet$ {\textbf{Step 1. Rearrange $h_t$ and keep $V_t$.}}
Let $h^\ast_t$ be the decreasing rearrangement of $h_t$, and let $\nu_t$ be the unique absolutely continuous solution of
    \begin{equation}
    \label{eq:nu_weak_CE}
    \dot\nu_t+\nabla\!\cdot V_t=h^\ast_t\,p,\qquad \nu_0=\mu_0.
    \end{equation}
By Lemma \ref{lem_frontload}, 
    we have $(\nu_t,V_t,h^\ast_t)_{t\in[0,1]}\in\mathcal{CE}_p(\mu_0,\mu_1;[0,1])$, and
    \begin{equation*}
    \mathfrak{E}^{a,b}_{\rm Quad}((\nu_t,V_t,h^\ast_t)_{t\in[0,1]})\le \mathfrak{E}^{a,b}_{\rm Quad}((\mu_t,V_t,h_t)_{t\in[0,1]}).
    \end{equation*}
Optimality of $(\mu_t,V_t,h_t)_{t\in[0,1]}$ forces equality, hence $(\nu_t,V_t,h^\ast_t)_{t\in[0,1]}$ is also optimal. 
Since $h_t$ and $h^\ast_t$ are equimeasurable, 
    $\lambda([h^\ast_t=0])=\lambda([h_t=0])>0$.
And from that $h^\ast_t$ is decreasing,
    there exists a nonempty open interval, call it $(t_0,t_1)$, such that
    \begin{equation*}
    h^\ast_t=0\quad\text{for a.e.\ }t\in(t_0,t_1).
    \end{equation*}
\vspace{0.1cm}

\noindent$\bullet$ {\textbf{Step 2. A localized admissible perturbation on the $0$-platform of $h^\ast_t$.}}
For $\varepsilon\ge 0$, define $\tilde h^\varepsilon_t$ by
    \begin{equation}
    \label{eq_no0pf_construct_h}
    \tilde{h}^\varepsilon_t :=
    \begin{cases}
    h^\ast_t, & t\notin(t_0,t_1),\\
    \varepsilon, & t\in\left(t_0,\frac{t_0+t_1}{2}\right),\\
    -\varepsilon, & t\in\left[\frac{t_0+t_1}{2},t_1\right),
    \end{cases}
    \end{equation}
    and define $\tilde\nu^\varepsilon_t$ by
    \begin{equation}
    \label{eq_no0pf_construct_nu}
    \tilde{\nu}^\varepsilon_t :=
    \begin{cases}
    \nu_t, & t\notin(t_0,t_1),\\
    \nu_t+(t-t_0)\,\varepsilon\,p, & t\in\left(t_0,\frac{t_0+t_1}{2}\right),\\
    \nu_t+(t_1-t)\,\varepsilon\,p, & t\in\left[\frac{t_0+t_1}{2},t_1\right).
    \end{cases}
    \end{equation}
A direct computation shows that
    \begin{equation*}
    \dot{\tilde{\nu}}^\varepsilon_t+\nabla\!\cdot V_t=\tilde h^\varepsilon_t\,p,\qquad
    \tilde\nu^\varepsilon_0=\nu_0=\mu_0,\qquad \tilde\nu^\varepsilon_1=\nu_1=\mu_1,
    \end{equation*}
hence $(\tilde\nu^\varepsilon_t,V_t,\tilde h^\varepsilon_t)_{t\in[0,1]}\in\mathcal{CE}_p(\mu_0,\mu_1;[0,1])$ 
    and $(\tilde\nu^0_t,V_t,\tilde h^0_t)_{t\in[0,1]}=(\nu_t,V_t,h^\ast_t)_{t\in[0,1]}$.
\vspace{0.3cm}

\noindent$\bullet$ {\textbf{Step 3. First variation of the action at $\varepsilon=0$.}}
By construction \eqref{eq_no0pf_construct_h}, on $(t_0,t_1)$ we have $\tilde h^0_t=0$, so
    \begin{equation}
    \label{eq_no0pf_variation@h}
    \frac{\di^+}{\di \varepsilon}\bigg|_{\varepsilon=0}\,a^2\!\int_{t_0}^{t_1}\!\big(\tilde h^\varepsilon_t\big)^2\di t
    =
    a^2\!\int_{t_0}^{t_1}\!2\,\tilde h^0_t\,\frac{\di^+}{\di \varepsilon}\bigg|_{\varepsilon=0}\tilde h^\varepsilon_t\di t
    =0,
    \end{equation}
    where $\frac{\di^+}{\di \varepsilon}|_{\ep=0}$ is the right derivative of $\ep$ at $0$.

For the transport part, Corollary \ref{cor_geodesic} states that 
    the optimal trajectory $(\nu_t,V_t,h^\ast_t)_{t\in[0,1]}$ has constant action.
In particular, since $h_t^\ast$ vanishes on $(t_0,t_1)$ we have
    \begin{equation*}
    0<\int_{t_0}^{t_1} A'(\nu_t,V_t)\,\mathrm dt<\infty.
    \end{equation*}
By the definition of $\alpha$ in \eqref{def_alpha}, for any $x,y$ with $K(x,y)>0$, the set
    \begin{equation*}
    \{t\in(t_0,t_1)\mid \alpha(V_t(x,y),\nu_t(x),\nu_t(y))=\infty\}=
    \{t\in(t_0,t_1)\mid V_t(x,y)\neq 0,\ \nu_t(x)\nu_t(y)=0\}
    \end{equation*}
    is negligible.
For each pair $(x,y)$ with $K(x,y)>0$, define the measurable sets
    \begin{gather*}
    Y_{x,y}^\circ:
    =\{t\in(t_0,t_1)\mid \alpha(V_t(x,y),\nu_t(x),\nu_t(y))=0\}
    =\{t\in(t_0,t_1)\mid V_t(x,y)=0\},
    \\
    Y_{x,y}^\bullet:=
    \{t\in(t_0,t_1)\mid \alpha(V_t(x,y),\nu_t(x),\nu_t(y))\in(0,\infty)\}
    =
    \{t\in(t_0,t_1)\mid V_t(x,y)\neq 0,\ \nu_t(x)\neq 0,\ \nu_t(y)\neq 0\}.
    \end{gather*}
Then $Y_{x,y}^\circ$ and $Y_{x,y}^\bullet$ are disjoint and 
    \begin{equation*}
    \lambda((t_0,t_1)\backslash (Y_{x,y}^\circ\cup Y_{x,y}^\bullet))=0.
    \end{equation*}

There exist $x^\bullet,y^\bullet$ with $K(x^\bullet,y^\bullet)>0$ such that 
    $\lambda\!\left(Y_{x^\bullet,y^\bullet}^\bullet\right)>0$.
Otherwise $\lambda\!\left(Y_{x,y}^\bullet\right)=0$ for all $x,y$ with $K(x,y)>0$, which gives
    \begin{align*}
    0<\int_{t_0}^{t_1} A'(\nu_t,V_t)\,\mathrm dt
    &=\frac{1}{2}\sum_{K(x,y)>0} K(x,y)\varpi(x)\int_{t_0}^{t_1} \alpha\!\left(V_t(x,y),\nu_t(x),\nu_t(y)\right)\,\mathrm dt \\
    &=\frac{1}{2}\sum_{K(x,y)>0} K(x,y)\varpi(x)\int_{Y_{x,y}^\bullet} \alpha\!\left(V_t(x,y),\nu_t(x),\nu_t(y)\right)\,\mathrm dt
    =0,
    \end{align*}
    a contradiction.

By \eqref{eq_no0pf_construct_nu}, for each $(x,y)$ with $K(x,y)>0$ and for $t\in Y_{x,y}^\bullet$,
    \begin{align*}
    \frac{\mathrm d^{+}}{\mathrm d\varepsilon}\bigg|_{\varepsilon=0}\,\frac{1}{\theta\!\left(\tilde{\nu}^\varepsilon_t(x),\tilde{\nu}^\varepsilon_t(y)\right)}
    = &
    -\frac{\partial_1\theta\!\left(\nu_t(x),\nu_t(y)\right)p(x)+\partial_2\theta\!\left(\nu_t(x),\nu_t(y)\right)p(y)}{\theta\!\left(\nu_t(x),\nu_t(y)\right)^2}
    \\
    &\times
    \left\{
    \begin{aligned}
    &t-t_0, && \text{if } t\in Y_{x,y}^\bullet\cap (t_0,\tfrac{t_0+t_1}{2}), \\
    &t_1-t, && \text{if } t\in Y_{x,y}^\bullet\cap [\tfrac{t_0+t_1}{2},t_1).\\
    \end{aligned}
    \right.
    \end{align*}
Using $p\in\mathcal M_+$ and $\theta$ is strictly increasing on $(0,\infty)\times(0,\infty)$,
    we have for $t\in Y_{x,y}^\bullet$,
    \begin{equation*}
     \frac{\mathrm d^{+}}{\mathrm d\varepsilon}\bigg|_{\varepsilon=0}\,\frac{1}{\theta\!\left(\tilde{\nu}^\varepsilon_t(x),\tilde{\nu}^\varepsilon_t(y)\right)}<0.
    \end{equation*}
Hence on $Y_{x,y}^\bullet$,
    \begin{equation*}
    \frac{\mathrm d^{+}}{\mathrm d\varepsilon}\bigg|_{\varepsilon=0}\,\alpha\!\left(V_t(x,y),\tilde{\nu}^\varepsilon_t(x),\tilde{\nu}^\varepsilon_t(y)\right)
    =
    \frac{\mathrm d^{+}}{\mathrm d\varepsilon}\bigg|_{\varepsilon=0}\,\frac{V_t(x,y)^2}{\theta\!\left(\tilde{\nu}^\varepsilon_t(x),\tilde{\nu}^\varepsilon_t(y)\right)}
    <0.
    \end{equation*}
On $Y_{x,y}^\circ$ we have for every $\varepsilon\ge 0$,
    \begin{equation*}
    \alpha\!\left(V_t(x,y),\tilde{\nu}^\varepsilon_t(x),\tilde{\nu}^\varepsilon_t(y)\right)
    =\alpha\!\left(0,\tilde{\nu}^\varepsilon_t(x),\tilde{\nu}^\varepsilon_t(y)\right)
    =0,
    \end{equation*}
    and therefore on $Y_{x,y}^\circ$,
    \begin{equation*}
    \frac{\mathrm d^{+}}{\mathrm d\varepsilon}\bigg|_{\varepsilon=0}\,\alpha\!\left(V_t(x,y),\tilde{\nu}^\varepsilon_t(x),\tilde{\nu}^\varepsilon_t(y)\right)=0.
    \end{equation*}
By the dominated convergence theorem,
    \begin{equation}
    \label{eq_no0pf_variation@nu}
    \begin{aligned}
    \frac{\mathrm d^{+}}{\mathrm d\varepsilon}\bigg|_{\varepsilon=0}\int_{t_0}^{t_1} A'(\tilde{\nu}^\varepsilon_t,V_t)\,\mathrm dt
    &=
    \frac{\mathrm d^{+}}{\mathrm d\varepsilon}\bigg|_{\varepsilon=0}\int_{t_0}^{t_1} \frac{1}{2}\sum_{K(x,y)>0}
    \alpha\!\left(V_t(x,y),\tilde{\nu}^\varepsilon_t(x),\tilde{\nu}^\varepsilon_t(y)\right) K(x,y)\varpi(x)\,\mathrm dt \\
    &=
    \frac{1}{2}\sum_{K(x,y)>0} K(x,y)\varpi(x)\int_{t_0}^{t_1}
    \frac{\mathrm d^{+}}{\mathrm d\varepsilon}\bigg|_{\varepsilon=0}\,\alpha\!\left(V_t(x,y),\tilde{\nu}^\varepsilon_t(x),\tilde{\nu}^\varepsilon_t(y)\right)\,\mathrm dt \\
    &=
    \frac{1}{2}\sum_{K(x,y)>0} K(x,y)\varpi(x)\int_{Y_{x,y}^\bullet}
    \frac{\mathrm d^{+}}{\mathrm d\varepsilon}\bigg|_{\varepsilon=0}\,\alpha\!\left(V_t(x,y),\tilde{\nu}^\varepsilon_t(x),\tilde{\nu}^\varepsilon_t(y)\right)\,\mathrm dt \\
    &\le
    \frac{1}{2}\,K(x^\bullet,y^\bullet)\varpi(x^\bullet)\int_{Y_{x^\bullet,y^\bullet}^\bullet}
    \frac{\mathrm d^{+}}{\mathrm d\varepsilon}\bigg|_{\varepsilon=0}\,\alpha\!\left(V_t(x^\bullet,y^\bullet),\tilde{\nu}^\varepsilon_t(x^\bullet),\tilde{\nu}^\varepsilon_t(y^\bullet)\right)\,\mathrm dt \\
    &<0.
    \end{aligned}
    \end{equation}

Combining \eqref{eq_no0pf_variation@h} and \eqref{eq_no0pf_variation@nu} yields
    \begin{equation*}
    \frac{\di^+}{\di \varepsilon}\bigg|_{\varepsilon=0}\,\mathfrak{E}^{a,b}_{\rm Quad}
    ((\tilde\nu^\varepsilon_t,V_t,\tilde h^\varepsilon_t)_{t\in[0,1]})
    =
    0\;+\;\frac{\di^+}{\di \varepsilon}\bigg|_{\varepsilon=0}\,b^2\!\int_{t_0}^{t_1}\!A'(\tilde\nu^\varepsilon_t,V_t)\di t
    <0.
    \end{equation*}
Therefore, for all sufficiently small $\varepsilon>0$,
    \begin{equation*}
    \mathfrak{E}^{a,b}_{\rm Quad}((\tilde\nu^\varepsilon_t,V_t,\tilde h^\varepsilon_t)_{t\in[0,1]})
    <
    \mathfrak{E}^{a,b}_{\rm Quad}((\nu_t,V_t,h^\ast_t)_{t\in[0,1]})
    =
    \mathfrak{E}^{a,b}_{\rm Quad}((\mu_t,V_t,h_t)_{t\in[0,1]}),
    \end{equation*}
    contradicting the optimality of $(\mu_t,V_t,h_t)_{t\in[0,1]}$. 
Hence $\lambda([h=0])=0$.
\end{proof}

The next result exploits the fact that for a trajectory to stay in the boundary, the source must vanish simultaneously.
\begin{theorem}[Non-locality]
\label{thm_a.e._not_in_boundary}
For $\mu_0\neq\mu_1$ in $\al M$, 
    let $(\mu_t,V_t,h_t)_{t\in[0,1]}\in \mathcal{CE}_p(\mu_0,\mu_1;[0,1])$ satisfy
    \begin{equation*}
    \mathfrak{E}^{a,b}_{\rm Quad}((\mu_t,V_t,h_t)_{t\in[0,1]})=\al W_p^{a,b}(\mu_0,\mu_1)^2,
    \end{equation*}
    then we have $\mu_t\in \al M_+$ for a.e. $t$.
\end{theorem}
\begin{proof}
Let $Z_x:=\{t\in[0,1] \mid \mu_t(x)=0\}$ for $x\in \al X$ be the time set such that $\mu_t$ touches the boundary at component $x$.

Assume without loss of generality that $V_t(x,y)=0$ whenever $K(x,y)=0$.
The definition of $A'$ and the reversibility \ref{eq_markov_reversible} implies that, 
    \begin{align*}
    \int_0^1 A'(\mu_t,V_t) \di t
    \ge & \
    \frac{1}{4}\int_{Z_x}  \sum_{\substack{y\in \al X,\\K(x,y)\neq 0}}
    [\alpha(V_t(x,y),\mu_t(x),\mu_t(y))+\alpha(V_t(y,x),\mu_t(y),\mu_t(x))]
    K(x,y)\varpi(x) \di t
    \\
    = & \
    \frac{1}{4}\int_{Z_x}  \sum_{\substack{y\in \al X,\\K(x,y)\neq 0}}
    [\alpha(V_t(x,y),0,\mu_t(y))+\alpha(V_t(y,x),\mu_t(y),0)]
    K(x,y)\varpi(x) \di t.
    \end{align*}
Thus for $\int_0^1 A'(\mu_t,V_t) \di t$ to be finite, 
    we have $V_t(x,y)=V_t(y,x)=0$ for almost every $t\in Z_x$, if $K(x,y)>0$.
Recalling the continuity function \eqref{new_ce}, we have for almost every $t\in Z_x$,
    \begin{equation}
    \label{eq_a.e._not_in_boundary_1}
   \dot \mu_t(x)=\frac{1}{2}\sum_{y\in \al X}(V_t(y,x)-V_t(x,y))K(x,y) + h_t \, p(x)=h_t p(x).
    \end{equation}
    
Define
    \begin{equation*}
    Z_x^\circ:=\{\,t\in Z_x:\ t \text{ is a Lebesgue density point of } Z_x \text{ and a Lebesgue point of } \dot\mu_t(x)\,\}.
    \end{equation*}
By the Lebesgue density and differentiation theorems, 
    $Z_x\setminus Z_x^\circ$ is negligible.
Fix $t\in Z_x^\circ$.
The definition of Lebesgue density points gives that
    \begin{equation*}
    \lim_{r\downarrow0} \frac{\lambda( Z_x^\circ \cap (t-r,t+r))}{2r} = 1.
    \end{equation*}
Then for sufficiently small $r$,
    \begin{equation*}
    \lambda\big(Z_x^\circ\cap(t-r,t)\big)>\tfrac{1}{2}r
    \quad\text{and}\quad
    \lambda\big(Z_x^\circ\cap(t,t+r)\big)>\tfrac{1}{2}r,
    \end{equation*}
    hence both intersections are nonempty.
Choose $a_r\in Z_x^\circ\cap(t-r,t)$ and $b_r\in Z_x^\circ\cap(t,t+r)$, and set $I_r:=(a_r,b_r)$.
Then
    \begin{equation*}
    a_r,b_r\in Z_x^\circ,\qquad t\in I_r,\qquad |b_r-a_r|\le 2r\ \longrightarrow\ 0\quad\text{as }r\downarrow 0,
    \end{equation*}
    which is the desired family of shrinking intervals with endpoints in $Z_x^\circ$.
Utilizing that $t$ is a Lebesgue point of $\dot\mu_t(x)$, and the fact that $\mu_{b_r}(x)=\mu_{a_r}(x)=0$,
    \begin{equation*}
    \dot\mu_t(x)=\lim_r\frac{1}{b_r-a_r} \int_{a_r}^{b_r}\dot\mu_s(x) \di s
    =\lim_r \frac{\mu_{b_r}(x)-\mu_{a_r}(x)}{b_r-a_r}=0.
    \end{equation*}
Thus for almost every $t\in Z_x$, \eqref{eq_a.e._not_in_boundary_1} gives $h_t=0$.
In other words,
    \begin{equation*}
    \lambda(Z_x)\le \lambda([h_t=0]).
    \end{equation*}
However by Lemma \ref{lem:no_zero_platform}, $\lambda([h_t=0])=0$, 
    otherwise there exists a strictly better trajectory.
Therefore, for any $x\in \al X$, $Z_x$ is a null set.  
\end{proof}

Finally, we can retrieve an optimal $\psi_t$ from $V_t$.
\begin{corollary}
\label{cor_minimizer}
For $\mu_0,\mu_1\in \al M$,
    there exists $(\mu_t,\psi_t,h_t)_{t\in[0,1]}\in\mathrm{CE}_p(\mu_0,\mu_1;[0,1])$ 
    such that 
    \begin{equation*}
        {\al W}_p^{a,b}(\mu_0,\mu_1)^2=E^{a,b}_{\rm Quad}((\mu_t,\psi_t,h_t)_{t\in[0,1]}).
    \end{equation*}
\end{corollary}
\begin{proof}
Let $(\mu_t,V_t,h_t)_{t\in[0,1]}\in \mathcal{CE}_p(\mu_0,\mu_1;[0,1])$ be the  minimizer guaranteed by Theorem \ref{thm_minimizer}.
By Theorem \ref{thm_a.e._not_in_boundary}, for any $x,y\in \al X$, $\hat\mu_t(x,y)$ is a.e. positive.
Thus we may define for a.e. $t\in[0,1]$,
    \begin{equation*}
    \Psi_t(x,y):=\frac{V_t(x,y)}{\hat\mu_t(x,y)}.
    \end{equation*}
Let $\mathscr P_{\mu_t}$ be the orthogonal projection in $\mathscr{G}_{\mu_t}$ onto the range of $\nabla$.
Similar to Step 2.3 of Appendix \ref{sec_app_W_rewrite},
there exists a measurable function $\psi_t:[0,1]\rightarrow \bb R^{\al X}$ such that the following arguments hold almost everywhere:
    \begin{equation*}
        \mathscr P_{\mu_t} \Psi_t = \nabla \psi_t,
        \quad
        \nabla \cdot  V_t = \nabla \cdot ({\hat \mu_t}  \nabla\psi_t)
        \quad
        \text{and}
        \quad
        [A_{\mu_t}\psi_t,\psi_t]\le A'(\mu_t,V_t).
    \end{equation*}
Thus we have that $(\mu_t,\psi_t,h_t)_{t\in[0,1]}\in\mathrm{CE}_p(\mu_0,\mu_1;[0,1])$
    and ${\al W}_p^{a,b}(\mu_0,\mu_1)^2=E^{a,b}_{\rm Quad}((\mu_t,\psi_t,h_t)_{t\in[0,1]})$.
\end{proof}

\subsection{Comparison with the shift--transport metric}
\label{sec_compare}
In this part we use the non-locality results in the above Section \ref{sec_nonlocality} to compare $\al W_p^{a,b}$ and the shift--transport metric.

In \cite{piccoli2014generalized}, the authors introduced a generalized Wasserstein distance 
    between nonnegative measures on Euclidean space. 
This was achieved by first shifting the two measures to have the same mass 
    and then considering the classical optimal transport problem. 
In their subsequent work \cite{piccoli2016properties}, 
    they established the Benamou--Brenier characterization for this generalized Wasserstein distance. 
Our approach, however, proceeds in the reverse direction: 
    we begin with the Benamou--Brenier characterization. 
This naturally raises the question: 
    \emph{Does our definition admit a reformulation in terms of shifting and transport?}
This subsection provides a negative answer to this question.

\begin{definition}
\label{def_Dpab}
For $\mu_0,\mu_1\in \al M$, the admissible set of shift--transport quadruplet is denoted 
    ${\rm ST}_p (\mu_0,\mu_1;[0,1])$ with elements $(H_0,H_1,(\nu_t)_{t\in[0,1]},(U_t)_{t\in[0,1]})$ satisfying
    \begin{equation}
    \left\{
    \begin{aligned}
    & H_0, H_1\in \bb{R},
    \\
    &
    \mu_0+H_0  p,\mu_1+H_1  p\in \al M,
    \\
    & \| \mu_0\|_{\varpi,1} + H_0  = \| \mu_1 \|_{\varpi,1} +H_1,
    \\
    & \nu:[0,1] \rightarrow \al M ~\text{is absolutely continuous}~ C^1,
    \\
    & U:[0,1]\rightarrow\mathbb R^{\al X \times \al X} ~\text{is measurable},
    \\
    & \left\{
    \begin{aligned}
    &\dot\nu_t+\nabla \cdot U_t=0, \\
    &\nu_t |_{t=0} = \mu_0+H_0  p,\quad \nu_t |_{t=1}=\mu_1 +H_1  p.
    \end{aligned}
    \right.
    \end{aligned}
    \right.
    \end{equation}
The shift--transport metric between $\mu_0,\mu_1$ is given as follows
    \begin{equation}
    \begin{gathered}
    \begin{aligned}
    D^{a,b}_p(\mu_0,\mu_1):=\inf
    \bigg\{ &
    \big(\mathcal{E}^{a,b}_{\rm Quad}(H_0,H_1,(\nu_t)_{t\in[0,1]},(U_t)_{t\in[0,1]})\big)^{\frac{1}{2}}
    \\
    & \bigg|~
    (H_0,H_1,(\nu_t)_{t\in[0,1]},(U_t)_{t\in[0,1]})\in {\rm ST}_p (\mu_0,\mu_1;[0,1])
    \bigg\},
    \end{aligned}
    \\
    \mathcal{E}^{a,b}_{\rm Quad}(H_0,H_1,(\nu_t)_{t\in[0,1]},(U_t)_{t\in[0,1]}):=
    a^2 (|H_0|+|H_1|)^2
    +b^2 \int_0^1 A'(\nu_t,U_t)\di t.
    \end{gathered}
    \end{equation}
\end{definition}

\begin{lemma}
\label{lem_D_rewrite}
$D^{a,b}_p$ defines a metric on $\al M$ and it admits the following reformulation
    \begin{equation}
    \label{eq_D_rewrite_E1}
    \begin{gathered}
    \begin{aligned}
    D^{a,b}_p(\mu_0,\mu_1)=\inf
    \bigg\{ &
    \big(\mathcal{E}^{a,b}_{\rm LinSq}(H_0,H_1,(\nu_t)_{t\in[0,1]},(U_t)_{t\in[0,1]})\big)^\frac{1}{2} 
    \\
    &\qquad\qquad\quad\bigg|~
    (H_0,H_1,(\nu_t)_{t\in[0,1]},(U_t)_{t\in[0,1]})\in {\rm ST}_p (\mu_0,\mu_1;[0,1])
    \bigg\},
    \end{aligned}    
    \\
    \mathcal{E}^{a,b}_{\rm LinSq}(H_0,H_1,(\nu_t)_{t\in[0,1]},(U_t)_{t\in[0,1]}):=
    a^2 (|H_0|+|H_1|)^2
    +b^2\left( \int_0^1 \left( A'(\nu_t,U_t)\right)^{\frac{1}{2}}  \di t \right)^2,
    \end{gathered}
    \end{equation}
    and
    \begin{equation}
    \label{eq_D_rewrite_F}
    \begin{gathered}
    D^{a,b}_p(\mu_0,\mu_1)=\inf
    \left\{
    \big(F^{a,b}((\mu_t,V_t,h_t)_{t\in[0,1]})\big)^{\frac{1}{2}} ~\big|~
    (\mu_t,V_t,h_t)_{t\in[0,1]}\in {\al C \al E}_p (\mu_0,\mu_1;[0,1])
    \right\},
    \\
    F^{a,b}((\mu_t,V_t,h_t)_{t\in[0,1]}):=
    a^2 \left( \int_0^1 |h_t|\di t \right) ^2
    +b^2 \left(\int_0^1  (A'(\mu_t,V_t))^{\frac{1}{2}}  \di t\right) ^2.
    \end{gathered}
    \end{equation}
\end{lemma}

\begin{proof}
The rewrite \eqref{eq_D_rewrite_E1} is of the same technical nature as Lemma \ref{lem_frakELinSq_potential_rewrite}.

To prove the equivalence between \eqref{eq_D_rewrite_E1} and \eqref{eq_D_rewrite_F}, 
    we first take an arbitrary quadruplet 
    \begin{equation*}
    (H_0,H_1,(\nu_t)_{t\in[0,1]},(U_t)_{t\in[0,1]})\in {\rm ST}_p (\mu_0,\mu_1;[0,1]).
    \end{equation*}
We define $(\mu_t,V_t,h_t)_{t\in[0,1]}$ as
    \begin{equation*}
    (\mu_t,V_t,h_t):=
    \left\{
    \begin{aligned}
    & (\mu_0+3t H_0 p,&&0,&& \ \ 3H_0), & & \text{if }t\in[0,1/3),\vspace{6pt}  \\
   & (\nu_{3t-1},&&U_{3t-1},&&\quad\ 0), & & \text{if }t\in[1/3,2/3), \vspace{6pt}\\
    & (\mu_1+3(1-t) H_1 p,&&0,&&-3H_1), & & \text{if }t\in[2/3,1].
    \end{aligned}
    \right.
    \end{equation*}
Then we have $(\mu_t,V_t,h_t)_{t\in[0,1]} \in {\al C \al E}_p (\mu_0,\mu_1;[0,1]) $
    and 
    \begin{equation*}
    \mathcal E^{a,b}_{\rm LinSq}(H_0,H_1,(\nu_t)_{t\in[0,1]},(U_t)_{t\in[0,1]}) = F^{a,b}((\mu_t,V_t,h_t)_{t\in[0,1]}).
    \end{equation*}
Thus
    \begin{equation*}
    D^{a,b}_p(\mu_0,\mu_1) \ge \inf
    \left\{
    \big(F^{a,b}((\mu_t,V_t,h_t)_{t\in[0,1]})\big)^{\frac{1}{2}} ~\big|~
    (\mu_t,V_t,h_t)_{t\in[0,1]}\in {\al C \al E}_p (\mu_0,\mu_1;[0,1])
    \right\}.
    \end{equation*}

For the other direction of the inequality of \eqref{eq_D_rewrite_F},
    we take $(\mu_t,V_t,h_t)_{t\in[0,1]} \in {\al C \al E}_p (\mu_0,\mu_1;[0,1]) $,
    and we define $H_0,H_1,U$ as
    \begin{equation*}
    H_0:= \int_0^1 h^+_t \di t, \quad H_1:=\int_0^1 h^-_t \di t,\quad U:=V,
    \end{equation*}
    and let $\nu_t:[0,1]\rightarrow \al M$ be the solution to this initial value problem
    \begin{equation*}
    \left\{
    \begin{aligned}
    &\dot \nu_t + \nabla \cdot U_t =0, \\
    &\nu_0= \mu_0+H_0p. \\
    \end{aligned}
    \right.
    \end{equation*}
Then we have
    \begin{equation*}
    \begin{aligned}
    \nu_1 
    =& \mu_0+H_0p -\int_0^1 \nabla \cdot U_t \di t 
    = \mu_0 + \int_0^1 (h^+_tp - \nabla \cdot V_t) \di t \\
    =&  \mu_0 + \int_0^1 (h_tp - \nabla \cdot V_t) \di t + \int_0^1  h_t^-p \di t \\
    =&\mu_1 + \int_0^1  h_t^-p \di t 
    = \mu_1 + H_1 p.
    \end{aligned}
    \end{equation*}
Now let us compare $ \al E^{a,b}_{\rm LinSq}(H_0,H_1,(\nu_t)_{t\in[0,1]},(U_t)_{t\in[0,1]}) $ and $ F^{a,b}((\mu_t,V_t,h_t)_{t\in[0,1]}) $.
For the source part, we have
    \begin{equation*}
    |H_0|+|H_1|= \int_0^1 h^+_t+h^-_t \di t=\int_0^1 |h_t| \di t.
    \end{equation*}
As for the transportation part, we present the explicit formula for $\mu_t$ and $\nu_t$ at position $x$
    \begin{equation}
    \begin{aligned}
    \mu_t(x) =& \mu_0(x) +p(x)\int_0^th_s\di s - \int_0^t\nabla \cdot V_s(x) \di s, \\
    \nu_t(x)=&\mu_0(x)+H_0p(x)-\int_0^t\nabla \cdot U_s(x) \di s \\
    =& \mu_0(x) +p(x)\int_0^1h_s^+\di s - \int_0^t\nabla \cdot V_s(x) \di s,
    \end{aligned}
    \end{equation}
    which leads to $\mu_t(x) \le \nu_t(x)$.
Then the monotonicity of $\theta$ in $A'$ \eqref{def_A'} gives
    \begin{equation*}
    A'(\mu_t,V_t) \ge A'(\nu_t,U_t).
    \end{equation*}
Furthermore,
    \begin{equation*}
    E^{a,b}_{\rm LinSq}(H_0,H_1,(\nu_t)_{t\in[0,1]},(U_t)_{t\in[0,1]}) \le  F^{a,b}((\mu_t,V_t,h_t)_{t\in[0,1]}).
    \end{equation*}
And finally,
    \begin{equation*}
    D^{a,b}_p(\mu_0,\mu_1) \le \inf
    \left\{
    \big(F^{a,b}((\mu_t,V_t,h_t)_{t\in[0,1]})\big)^{\frac{1}{2}} ~\big|~
    (\mu_t,V_t,h_t)_{t\in[0,1]}\in {\al C \al E}_p (\mu_0,\mu_1;[0,1])
    \right\}.
    \end{equation*}
We have validated the reformulation \eqref{eq_D_rewrite_F}.

Thanks to the similarity between $F^{a,b}$ and $E^{a,b}_{\rm Quad}$,
    the arguments that are used to establish $\al W_p^{a,b}$ as a metric in our previous work \cite[Theorem 3.10]{mao2026wasserstein} can be applied to characterization \eqref{eq_D_rewrite_F} with obvious adaptations.
Thus the proof is omitted here.
\end{proof}

\begin{theorem}
\label{thm_comparision_RP}
$D^{a,b}_p(\mu_0,\mu_1) \le \al W_p^{a,b}(\mu_0,\mu_1)$ and the inequality holds strictly if $\mu_0-\mu_1 \notin \mathrm{span}\{p\}$.
\end{theorem}

\begin{proof}

For readability, we collect the rewrite formulas for $D^{a,b}_p(\mu_0,\mu_1)$ and $\al W^{a,b}_p(\mu_0,\mu_1)$ given in Lemma \ref{lem_D_rewrite} and Lemma \ref{lem_W_rewrite}.
    \begin{align}
    \al W^{a,b}_p(\mu_0,\mu_1)
    &=
    \inf \left\{
    \big(\mathfrak{E}^{a,b}_{\rm Quad}((\mu_t,V_t,h_t)_{t\in[0,1]})\big)^{\frac{1}{2}}
    ~\big|~
    (\mu_t,V_t,h_t)_{t\in[0,1]}\in {\al C \al E}_p (\mu_0,\mu_1;[0,1])
    \right\},
    \label{eq_W_rewrite_B2_+}
    \\
    D^{a,b}_p(\mu_0,\mu_1)
    &=
    \inf \left\{
    \big(F^{a,b}((\mu_t,V_t,h_t)_{t\in[0,1]})\big)^{\frac{1}{2}}
    ~\big|~
    (\mu_t,V_t,h_t)_{t\in[0,1]}\in {\al C \al E}_p (\mu_0,\mu_1;[0,1])
    \right\}.
    \label{eq_D_rewrite_F_+}
    \end{align}
Here
    \begin{equation*}
    \begin{aligned}
    \mathfrak E^{a,b}_{\rm Quad}((\mu_t,V_t,h_t)_{t\in[0,1]})
    &:=
    a^2 \int_0^1 |h_t|^2 \, \di t
    +
    b^2 \int_0^1 A'(\mu_t,V_t)\, \di t ,
    \\
    F^{a,b}((\mu_t,V_t,h_t)_{t\in[0,1]})
    &:=
    a^2 \left( \int_0^1 |h_t| \, \di t \right)^2
    +
    b^2 \left( \int_0^1 \big(A'(\mu_t,V_t)\big)^{\frac{1}{2}} \, \di t \right)^2 .
    \end{aligned}
    \end{equation*}
By Jensen's inequality we obtain
    \begin{equation}
    \mathfrak E^{a,b}_{\rm Quad}((\mu_t,V_t,h_t)_{t\in[0,1]})
    \ge
    F^{a,b}((\mu_t,V_t,h_t)_{t\in[0,1]}) ,
    \label{ineq_B2_F}
    \end{equation}
    which implies $D^{a,b}_p(\mu_0,\mu_1) \le \al W^{a,b}_p(\mu_0,\mu_1)$.

We now prove the implication
\begin{equation}
\label{assump_D=W}
D^{a,b}_p(\mu_0,\mu_1)=\al W_p^{a,b}(\mu_0,\mu_1)
\;\Rightarrow\;
\mu_1-\mu_0\in \mathrm{span}\{p\}.
\end{equation}
Let $(\mu_t^*,V_t^*,h_t^*)_{t\in[0,1]}\in \al{CE}_p(\mu_0,\mu_1;[0,1])$ be a minimizer for \eqref{eq_W_rewrite_B2_+} given by Theorem \ref{thm_minimizer}.
Using \eqref{assump_D=W} together with \eqref{eq_D_rewrite_F_+} and \eqref{ineq_B2_F}, we get the equality of the two energies at the minimizer,
\begin{equation}
\label{eq:E=F}
\mathfrak E^{a,b}_{\rm Quad}\!\big((\mu_t^*,V_t^*,h_t^*)_{t\in[0,1]}\big)
=
F^{a,b}\!\big((\mu_t^*,V_t^*,h_t^*)_{t\in[0,1]}\big).
\end{equation}
By the equality case of Jensen’s inequality applied to $t\mapsto |h_t^*|$ and $t\mapsto A'(\mu_t^*,V_t^*)^{1/2}$, both $|h_t^*|$ and $A'(\mu_t^*,V_t^*)$ are constant for a.e.\ $t$.
{{According to Lemma \ref{lem:no_zero_platform}, there exists $h>0$ such that $|h_t^*|=h$ for a.e. $t$.}}
By Lemma \ref{lem_frontload}, we may assume $h_t^*$ is right–continuous and monotone decreasing, hence there exists some $t_0\in[0,1]$ with
\begin{equation*}
h_t^*=
\begin{cases}
h,& t\in[0,t_0),\\
-h,& t\in[t_0,1],
\end{cases}
\qquad
([0,0)=\emptyset,\ [1,1]=\{1\}).
\end{equation*}

Define $H_0^*:=t_0h$ and $H_1^*:=(1-t_0)h$, and let $\nu_t^*:[0,1]\to \al M$ solve
\begin{equation*}
\begin{aligned}
&\dot\nu_t^*+\nabla\cdot V_t^*=0,\\
&\nu_0^*=\mu_0+H_0^*p .
\end{aligned}
\end{equation*}
Then $\nu_1^*=\mu_1+H_1^*p$, and
\begin{equation}
\label{eq_source_vanish}
|H_0^*|+|H_1^*|=h=\int_0^1 |h_t^*|\,\di t .
\end{equation}
Moreover, for every $x\in\al X$ and a.e.\ $t\in[0,1]$,
\begin{equation*}
\begin{aligned}
\nu_t^*(x)-\mu_t^*(x)
&=\mu_0(x)+H_0^*p(x)-\int_0^t\nabla\cdot V_s^*(x)\,\di s
-\mu_0(x)-p(x)\!\int_0^t h_s^*\,\di s+\int_0^t\nabla\cdot V_s^*(x)\,\di s\\
&=p(x)\!\left(H_0^*-\int_0^t h_s^*\,\di s\right)
=
p(x)\times
\begin{cases}
h\,(t_0-t),& t\in[0,t_0),\\
h\,(t-t_0),& t\in[t_0,1],
\end{cases}
\end{aligned}
\end{equation*}
hence $\nu_t^*(x)>\mu_t^*(x)$ for a.e.\ $t$.

Since $\mu_t^*,\nu_t^*\in\al M_+$ a.e., we compute for a.e.\ $t$,
\begin{equation}
\begin{aligned}
\label{eq_DeltaA'}
A'(\mu_t^*,V_t^*)-A'(\nu_t^*,V_t^*)
&=\frac12\!\sum_{K(x,y)>0}\!
\Big[\alpha\big(V_t^*(x,y),\mu_t^*(x),\mu_t^*(y)\big)
-\alpha\big(V_t^*(x,y),\nu_t^*(x),\nu_t^*(y)\big)\Big]
K(x,y)\varpi(x)\\
&=\frac12\!\sum_{K(x,y)>0}\!
\Big[\theta\big(\mu_t^*(x),\mu_t^*(y)\big)^{-1}
-\theta\big(\nu_t^*(x),\nu_t^*(y)\big)^{-1}\Big]
V_t^*(x,y)^2\,K(x,y)\varpi(x).
\end{aligned}
\end{equation}
By the monotonicity of $\theta$ and $\nu_t^*>\mu_t^*$, every bracket is strictly positive, whence the whole sum is nonnegative and vanishes iff all the $V_t^*(x,y)$ with $K(x,y)>0$ vanish.
Since \eqref{eq:E=F} holds and the source parts coincide \eqref{eq_source_vanish}, 
    equality must also hold in the transportation parts, thus \eqref{eq_DeltaA'} vanishes.
By reversibility, $K(y,x)=K(x,y)\varpi(x)/\varpi(y)>0$ if $K(x,y)>0$, 
    so both orientations $(x,y)$ and $(y,x)$ occur among the summands.
Therefore, for a.e.\ $t$ and all oriented edges with $K(x,y)>0$,
    \begin{equation*}
    V_t^*(x,y)=V_t^*(y,x)=0 .
    \end{equation*}
Therefore, for every $x\in\al X$,
    \begin{equation*}
    \begin{aligned}
    \mu_1(x)-\mu_0(x)
    = & \
    \int_0^1 \dot \mu^*_s(x) \di s
    = 
    \int_0^1 (-\nabla\cdot V_s^*(x) +h_s^*p(x)) \di s
    \\
    = & \
    - \frac{1}{2}\int_0^1\sum_{y\in\al X}[V_s^*(x,y)-V_s^*(y,x) ]K(x,y)\di s +
    \int_0^1 h_s^*p(x) \di s
    \\
    = & \
    \int_0^1 h_s^* \di s \ p(x),
    \end{aligned}
    \end{equation*}
which shows $\mu_1-\mu_0\in \mathrm{span}\{p\}$.

\end{proof}

\begin{remark}
\label{rmk_ME}
Let $\mu_0,\mu_1$ be probability measures with $\mu_0\ne\mu_1$.
The Maas--Erbar metric restricts admissible curves to the probability simplex, i.e.\ to triples $(\mu_t,V_t,h_t)_{t\in[0,1]}$ with $h_t\equiv 0$, and it admits an optimal curve \cite[Theorem 3.2]{erbar2012ricci}.
Such a curve is also admissible for $\al W^{a,1}_p$, hence
\begin{equation*}
\al W^{a,1}_p(\mu_0,\mu_1)\ \le\ \mathrm{ME}(\mu_0,\mu_1).
\end{equation*}
To see that the inequality is strict, fix an ME--minimizer $(\mu^{\rm ME}_t,V^{\rm ME}_t,0)_{t\in[0,1]}$ connecting $(\mu_0,\mu_1)$.
By Lemma \ref{lem:no_zero_platform}, any $\al W^{a,1}_p$--minimizer for $\mu_0\ne\mu_1$ must satisfy $|h_t|>0$ for a.e.\ $t$; in particular, the ME--curve with $h_t\equiv 0$ cannot be optimal for $\al W^{a,1}_p$.
Therefore
\begin{equation*}
\al W^{a,1}_p(\mu_0,\mu_1)\ <\ \mathrm{ME}(\mu_0,\mu_1),
\end{equation*}
with equality only in the trivial case $\mu_0=\mu_1$.
\end{remark}

\subsection{Conditioned uniqueness of geodesics}
\label{sec_uniqueness}
For any measurable $\Psi\in\mathbb{R}^{\al X\times \al X}$,
    let its asymmetric transformation be $\check{\Psi}(x,y):=\frac{1}{2}\Psi(x,y)-\frac{1}{2}\Psi(y,x)$ for $x,y\in \al X$.
The next result states that the asymmetric transformation of $V_t$ in \eqref{new_ce} never increases the action functional. 
Proof in Appendix \ref{sec_app_f}.
\begin{lemma}
\label{lem_asymmetric}
If $(\mu_t,V_t,h_t)_{t\in[0,1]}\in\mathcal{CE}_p(\mu_0,\mu_1;[0,1])$,
    then $(\mu_t,\check V_t,h_t)_{t\in[0,1]}$ is also in $\mathcal{CE}_p(\mu_0,\mu_1;[0,1])$,
    and the action functional satisfies $\mathfrak{E}^{a,b}_{\rm Quad}((\mu_t,\check V_t,h_t)_{t\in[0,1]})\le \mathfrak{E}^{a,b}_{\rm Quad}((\mu_t,V_t,h_t)_{t\in[0,1]})$.
\end{lemma}

Next, we show that if a minimizer exists whose trajectory lies entirely in $\mathcal M_+$, hence the endpoints must also belong to $\mathcal M_+$, then that minimizer is unique.
\begin{proposition}
\label{prop_conditioned_uniqueness}
Let $\mu_0,\mu_1\in \al M_+$.
If for $i=1,2$, there exist $(\mu^i_t,V^i_t,h^i_t)_{t\in[0,1]}\in\mathcal{CE}_p(\mu_0,\mu_1;[0,1])$
    such that 
    \begin{enumerate}
    \item they minimize the functional $\mathfrak E^{a,b}_{\rm Quad}:$
        \begin{equation*}
        \mathfrak E^{a,b}_{\rm Quad}((\mu^1_t,V^1_t,h^1_t)_{t\in[0,1]})
        =
        \mathfrak E^{a,b}_{\rm Quad}((\mu^2_t,V^2_t,h^2_t)_{t\in[0,1]})
        =
        \al W_p^{a,b}(\mu_0,\mu_1)^2;
        \end{equation*}
    \item the fluxes are antisymmetric, $V_t^i(x,y)=-V_t^i(y,x)$ for $i=1,2$  and all $x,y\in \al X$ and a.e.\ $t$;
    \item the fluxes vanish when two points are not connected: if $K(x,y)=0$, then $V_t^i(x,y)=0$ for $i=1,2$;
    \item one trajectory, say $\mu^1$, satisfies
        \begin{equation*}
        \mu^1_t(x)>0\qquad\text{for every }x\in \al X\text{ and all }t\in[0,1].
        \end{equation*}
    \end{enumerate}
Then the two minimizers coincide:
    \begin{gather*}
    \mu^1_t=\mu^2_t, \quad \forall ~t\in[0,1],\\ 
    V^1_t=V^2_t,\quad h^1_t=h^2_t \quad \text{a.e. } t\in[0,1] 
    \end{gather*}
\end{proposition}
\begin{proof}
We split the proof into five steps.
\vspace{0.2cm}

\noindent$\bullet$ \textbf{Step 1. Source trajectories coincide.}
Set, for $\tau\in[0,1]$,
    \begin{equation*}
    \begin{gathered}
    (\mu^\tau_t,V^\tau_t,h^\tau_t):=\tau(\mu^1_t,V^1_t,h^1_t)+(1-\tau)(\mu^2_t,V^2_t,h^2_t), \\
    \Upsilon(\tau):=\int_0^1 |h^\tau_t|^2\,\di t,\qquad
    \Xi(\tau):=\int_0^1 A'(\mu^\tau_t,V^\tau_t)\,\di t.
    \end{gathered}
    \end{equation*}
By linearity of the constraints, 
    $(\mu^\tau_t,V^\tau_t,h^\tau_t)_{t\in[0,1]}\in\al{CE}_p(\mu_0,\mu_1;[0,1])$ for each $\tau$.
Since $h\mapsto|h|^2$ is convex and $(\mu,V)\mapsto A'(\mu,V)$ is convex, both $\Upsilon$ and $\Xi$ are convex, and
    \begin{equation*}
    \mathfrak E^{a,b}_{\rm Quad}((\mu^\tau_t,V^\tau_t,h^\tau_t)_{t\in[0,1]})=a^2\Upsilon(\tau)+b^2\Xi(\tau)
    \end{equation*}
    is convex.
    
Because $(\mu^1_t,V^1_t,h^1_t)_{t\in[0,1]}$ and $(\mu^2_t,V^2_t,h^2_t)_{t\in[0,1]}$ are minimizers, 
    the convex function
    \begin{equation*}
    \tau\mapsto \mathfrak E^{a,b}_{\rm Quad}((\mu^\tau_t,V^\tau_t,h^\tau_t)_{t\in[0,1]})
    \end{equation*}
    takes the same minimal value at $\tau=0$ and $\tau=1$. 
Hence it is constant on $[0,1]$. 
With $a,b>0$, this forces {both} $\Upsilon$ and $\Xi$ to be affine on $[0,1]$.
Otherwise $a^2\Upsilon+b^2\Xi$ would be strictly convex at some chord.

But $\Upsilon$ is in fact {strictly} convex unless $h^1_t=h^2_t$ a.e., since
    \begin{equation*}
    \Upsilon(\tau)
    =\int_0^1\big|\tau h^1_t+(1-\tau)h^2_t\big|^2\,\di t
    =\tau\Upsilon(1)+(1-\tau)\Upsilon(0)+\tau(\tau-1)\int_0^1 |h^1_t-h^2_t|^2\,\di t.
    \end{equation*}
Affineness of $\Upsilon$ therefore yields $\int_0^1 |h^1_t-h^2_t|^2\,\di t=0$, i.e.,
    \begin{equation*}
    h^1_t=h^2_t=:h_t\qquad\text{a.e. $t$\ on }(0,1),
    \end{equation*}
    and consequently $\Upsilon(\tau)\equiv \Upsilon(0)=\Upsilon(1)$ and $\Xi(\tau)\equiv \Xi(0)=\Xi(1)$ are \emph{constant} on $[0,1]$.
\vspace{0.2cm}

\noindent$\bullet$ \textbf{Step 2. Edgewise Jensen equality at each time.}
Fix $\tau\in[0,1]$. 
For a.e.\ $t\in(0,1)$, by convexity of $\alpha$,
    \begin{equation*}
    \begin{aligned}
    A'(\mu^\tau_t,V^\tau_t)
    =\frac12\sum_{x,y}\alpha(Z^\tau_{xy}(t))\,K(x,y)\,\varpi(x)
    \le & \
    \frac12\sum_{x,y}\Big[\tau\,\alpha(Z^1_{xy}(t))+(1-\tau)\,\alpha(Z^2_{xy}(t))\Big]K(x,y)\,\varpi(x)
    \\
    =
    \tau A'(\mu^1_t,V^1_t)+(1-\tau)A'(\mu^2_t,V^2_t),
    \end{aligned}
    \end{equation*}
    where 
    \begin{equation*}
    Z^i_{xy}(t):=(V^i_t(x,y),\mu^i_t(x),\mu^i_t(y)),
    \quad
    Z^\tau_{xy}(t):=\tau Z^1_{xy}(t)+(1-\tau)Z^2_{xy}(t).
    \end{equation*}
Define the nonnegative \emph{gap}
    \begin{equation*}
    G_\tau(t):=\frac12\sum_{x,y}\Big(\tau\,\alpha(Z^1_{xy}(t))+(1-\tau)\,\alpha(Z^2_{xy}(t))-\alpha(Z^\tau_{xy}(t))\Big)K(x,y)\,\varpi(x)\ge 0.
    \end{equation*}
Integrating in $t$ and using the constancy of $\Xi$ from Step~1,
    \begin{equation*}
    \int_0^1 G_\tau(t)\di t
    =\int_0^1\!\Big(\tau A'(\mu^1_t,V^1_t)+(1-\tau)A'(\mu^2_t,V^2_t)-A'(\mu^\tau_t,V^\tau_t)\Big)\di t
    =\tau\Xi(1)+(1-\tau)\Xi(0)-\Xi(\tau)=0.
    \end{equation*}
Hence $G_\tau(t)=0$ for a.e.\ $t$. 
Since $G_\tau(t)$ is a finite sum of nonnegative terms, 
    each term must vanish for those $t$:
    \begin{equation}
    \label{eq:edgewise-affine} 
    \alpha(Z^\tau_{xy}(t))
    =\tau\,\alpha(Z^1_{xy}(t))+(1-\tau)\,\alpha(Z^2_{xy}(t))
    \qquad\text{for every }(x,y)\text{ with }K(x,y)>0,\ \text{a.e.\ }t.
    \end{equation}
\eqref{eq:edgewise-affine} holds for all $\tau\in[0,1]$.
\vspace{0.2cm}

\noindent$\bullet$ \textbf{Step 3. Nodewise ratios.}
Define $\kappa_t(x):=\mu^2_t(x)/\mu^1_t(x)$.
This is well defined since $\mu_t^1(x)>0$ by assumption.
Fix an ordered edge $(x,y)$ with $K(x,y)>0$ and 
    take almost every $t$ for which \eqref{eq:edgewise-affine} holds.
If $Z^1_{xy}(t)=Z^2_{xy}(t)$, 
    then $\mu_t^2(x)=\mu_t^1(x)$ and $V_t^2(x,y)=V_t^1(x,y)$, hence $\kappa_t(x)=1$ and
    \begin{equation*}
    V_t^2(x,y)=\kappa_t(x)\,V_t^1(x,y).
    \end{equation*}
If $Z^1_{xy}(t)\neq Z^2_{xy}(t)$, 
    then by the affine segment characterization in Lemma~\ref{lem_alpha_convex} the relation \eqref{eq:edgewise-affine} forces exactly one of the following:
    \begin{equation}
    \label{kappa_dichtomy}
    \begin{aligned}
    &\text{(i) } V_t^1(x,y)=V_t^2(x,y)=0,\\
    &\text{(ii) } (V_t^2(x,y),\mu_t^2(x),\mu_t^2(y))=\kappa_t(x,y)\,(V_t^1(x,y),\mu_t^1(x),\mu_t^1(y)) \text{ for some }\kappa_t(x,y)\in(0,1)\cup(1,\infty).
    \end{aligned}
    \end{equation}
In case (ii) we compare the second components and obtain $\mu_t^2(x)=\kappa_t(x,y)\,\mu_t^1(x)$, 
    hence $\kappa_t(x,y)=\kappa_t(x)$ and therefore
    \begin{equation*}
    V_t^2(x,y)=\kappa_t(x)\,V_t^1(x,y).
    \end{equation*}
In case (i) we also have $V_t^2(x,y)=\kappa_t(x)\,V_t^1(x,y)$ because both sides vanish.

Finally, if $K(x,y)=0$, then by the standing convention $V_t^1(x,y)=V_t^2(x,y)=0$ and again
\begin{equation*}
V_t^2(x,y)=\kappa_t(x)\,V_t^1(x,y).
\end{equation*}
\vspace{0.2cm}

\noindent$\bullet$ \textbf{Step 4. Regularity of $\kappa$ and the nodewise ODE.}
We have the nodewise continuity equation with antisymmetric flux,
    \begin{equation}
    \label{eq_ce_antisym_flux}
    \dot\mu^i_t(x)+\sum_{y}V^i_t(x,y)\,K(x,y)=h_t\,p(x),\qquad i=1,2.
    \end{equation}
Since $\mu^1_t(x)$ is continuous and strictly positive on the compact interval $[0,1]$, 
    it admits a positive lower bound. 
Hence $\kappa_t(x)={\mu^2_t(x)}/{\mu^1_t(x)}$ is absolutely continuous and
    for a.e. $t\in[0,1]$,
    \begin{equation*}
    \dot \kappa_t(x)=\frac{\dot\mu^2_t(x)\,\mu^1_t(x)-\mu^2_t(x)\,\dot\mu^1_t(x)}{\mu^1_t(x)^2}
    = (\dot \mu^2_t(x)-\kappa_t(x)\dot \mu^1_t(x))(\mu^1_t(x))^{-1}.
    \end{equation*}
Using the continuity equation \eqref{eq_ce_antisym_flux}
    and $V^2_t(x,y)=\kappa_t(x)V _t^1(x,y)$,
    we obtain, for a.e.\ $t\in[0,1]$,
    \begin{equation}\label{eq:kappa-ode}
    \begin{aligned}
    \dot \kappa_t(x) = & \ (\dot \mu^2_t(x)-\kappa_t(x)\dot \mu^1_t(x)) (\mu^1_t(x))^{-1}
    \\
    = &\left(\ h_t p(x)- \sum_{y}V^2_t(x,y)\,K(x,y)
    - \kappa_t(x)\Big(h_t p(x)- \sum_{y}V^1_t(x,y)\,K(x,y)\Big)\right) (\mu^1_t(x))^{-1}
    \\
    =& \ (1-\kappa_t(x))\frac{h_t p(x)}{\mu^1_t(x)}.
    \end{aligned}
    \end{equation}
\vspace{0.2cm}

\noindent$\bullet$ \textbf{Step 5. Global solution on $[0,1]$ and $\kappa\equiv 1$.}
Since $\mu^1_t(x)$ has a positive lower bound and $h\in \mathrm L^2([0,1];\bb R)\subset \mathrm L^1([0,1];\bb R)$, the coefficient
    \begin{equation*}
    \frac{h_t\,p(x)}{\mu^1_t(x)}\in \mathrm L^1([0,1];\bb R),
    \end{equation*}
    and \eqref{eq:kappa-ode} is a linear Carathéodory ODE on $[0,1]$ with unique absolutely continuous solution given by the integrating-factor formula
    \begin{equation*}
    \kappa_t(x)=1+\big(\kappa_0(x)-1\big)\exp\!\Big(-\!\int_0^{t}\frac{h_\tau\,p(x)}{\mu^1_\tau(x)}\,\di\tau\Big).
    \end{equation*}
The endpoints of the two minimizers coincide, so $\kappa_0(x)=\mu^2_0(x)/\mu^1_0(x)=1$, and consequently
    \begin{equation*}
    \kappa_t(x)\equiv 1\qquad\text{for all }t\in[0,1].
    \end{equation*}
Therefore $\mu^2_t(x)=\mu^1_t(x)$ for every $x$ and all $t$.
Moreover, $V^2_t$ coincides with $V^1_t$ for a.e. $t$.
\end{proof}

\subsection{Geodesic equation}
\label{sec_geoeq}

Having established $(\al M,\al W_p^{a,b})$ as a geodesic space, 
    we are naturally curious about the geodesic equation.
For that, an all-time strictly positive condition is necessary.

\begin{theorem}
\label{thm:GE-weak}
Let $\mu_t:[0,1]\to \al M_+$ be an absolutely continuous curve.
Let $\nabla\psi_t$ and $h_t$ satisfy the continuity equation for a.e. $t\in[0,1],$
    \begin{equation*}
    \dot \mu_t + \nabla_{\mu_t} \cdot \nabla \psi_t = h_t p.
    \end{equation*}
Then the following two assertions are equivalent:

\medskip
\noindent \emph{(i)} 
$\mu_t$ is a constant-speed geodesic for $\al W^{a,b}_p$.

\medskip
\noindent \emph{(ii)} $(\nabla\psi_t,h_t)$ satisfies the weak geodesic system on $(0,1)$.
For all scalar tests $\zeta\in\mathrm C_{\mathrm c}^1((0,1);\mathbb{R})$, one has
    \begin{equation}
    \tag{GE–$h$}\label{eq:GE-h-weak}
    \int_0^1 \, h_t\,\dot\zeta(t)\di t
    = \,\frac{b^2}{2 a^2}\int_0^1
    \Bigg[
    \sum_{x,y\in \al X} \nabla\psi_t(x,y)^2  \partial_1\theta(\mu_t(x),\mu_t(y))K(x,y) \varpi(x)  p(x)
    \Bigg]\,
    \zeta(t)\di t.
    \end{equation}
Meanwhile, for all scalar tests $\zeta\in \mathrm C_{\mathrm c}^1((0,1);\mathbb{R})$ and $x,y\in \al X$ with $K(x,y)\neq 0$, one has
    \begin{align}
    \tag{GE–$\psi$}\label{eq:GE-psi-weak}
    \int_0^1 \nabla \psi_t(x,y) \dot\zeta(t)\di t
    =
    \frac{1}{2}
    \int_0^1
    & \left(
     \sum_{z}
    \nabla\psi_t(z,y)^2
    K(y,z)
    \partial_1 \theta(\mu_t(y),\mu_t(z))
    \right.
    \\
    & \left. -
    \sum_{z}
    \nabla\psi(z,x)^2
    K(x,z)
    \partial_1 \theta(\mu_t(x),\mu_t(z))\right)
    \zeta(t)
    \di t. \notag
    \end{align}
\end{theorem}

Since the proof is lengthy yet technically standard, we postpone it to Appendix \ref{sec_app_GE_weak} and only sketch it here.
\begin{proof}[Sketch of proof]
(i) $\Rightarrow$ (ii).
For a.e.\ $t$, set $(\nabla\psi_t,h_t):=D_{\mu_t}(\dot\mu_t)$ and
    $V_t(x,y):=\theta(\mu_t(x),\mu_t(y))\,\nabla\psi_t(x,y)$.
Consider first variations that perturb only $h$ (with a compensating perturbation of $\mu$
    in the $p$-direction) and compute the derivative of $\mathfrak{E}^{a,b}_{\rm Quad}$ at $0$;
    optimality and integration by parts in time yield \eqref{eq:GE-h-weak}.
Then use coupled $(\mu,V)$–perturbations that redistribute mass along a single edge,
    together with discrete integration by parts in space, to obtain \eqref{eq:GE-psi-weak}.

(ii) $\Rightarrow$ (i).
Assume $(\mu_t,\nabla\psi_t,h_t)_{t\in[0,1]}$ satisfies the continuity equation and
    \eqref{eq:GE-h-weak}–\eqref{eq:GE-psi-weak}, 
    and define $V_t(x,y):=\theta(\mu_t(x),\mu_t(y))\,\nabla\psi_t(x,y)$.
Any admissible first–order perturbation decomposes into a source part and a transport part
of the type used above.
The weak equations imply that the first variation of $\mathfrak{E}^{a,b}_{\rm Quad}$
vanishes along both families, hence along all feasible directions.
Since $\mathfrak{E}^{a,b}_{\rm Quad}$ is convex on the affine feasible set,
    this stationarity implies global minimality, and the dynamic characterization
    yields that $\mu_t$ is a constant–speed $\al W^{a,b}_p$–geodesic.
\end{proof}

\begin{remark}
The assumption $\mu_t\in\mathcal M_+$ for all $t$ in Theorem \ref{thm:GE-weak} is imposed to ensure the validity of the argument. 
By Theorem \ref{thm_a.e._not_in_boundary} we know only that any constant-speed geodesic between arbitrary $\mu_0,\mu_1\in\mathcal M$ lies in $\mathcal M_+$ for almost every $t\in(0,1)$.
Our present proof does not exclude momentary contact with the boundary, 
    so there remains a gap between the weak result and the desired global non-locality. 
This is a limitation of the current theory, and we plan to remove it in future work. 
Numerical evidence in Remark \ref{rmk_numerical} indicates that the computed geodesics stay in the interior $\al M_+$ for the time interval $(0,1)$.
\end{remark}

\begin{remark}\label{rmk_numerical}
    \begin{figure}[h!]
    \centering
    \includegraphics[width=0.6\textwidth]{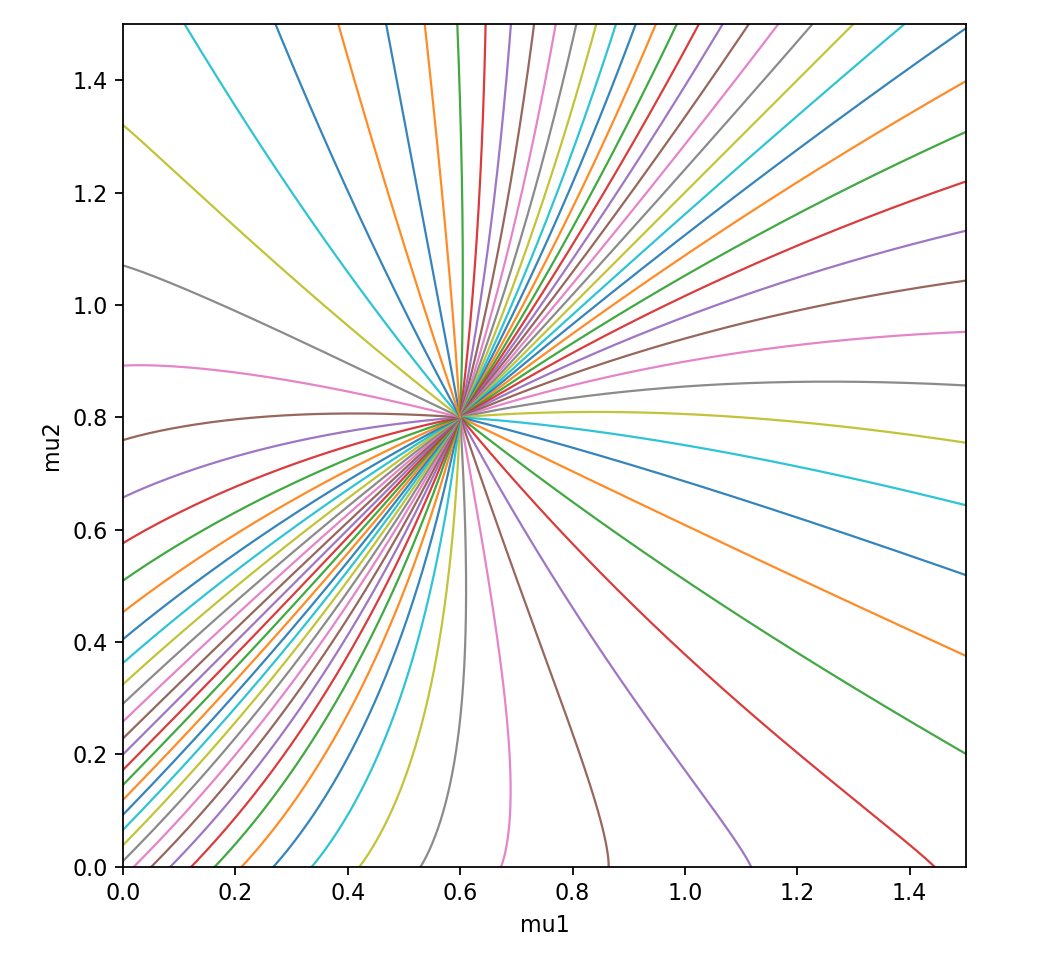}
    \captionsetup{width=\textwidth}
    \caption{Rays starting from $(0.6,0.8)^{\mathrm{T}}$}
    \label{fig-ray060080}
    \end{figure}
    \begin{figure}[h]
    \centering
    \includegraphics[width=0.6\textwidth]{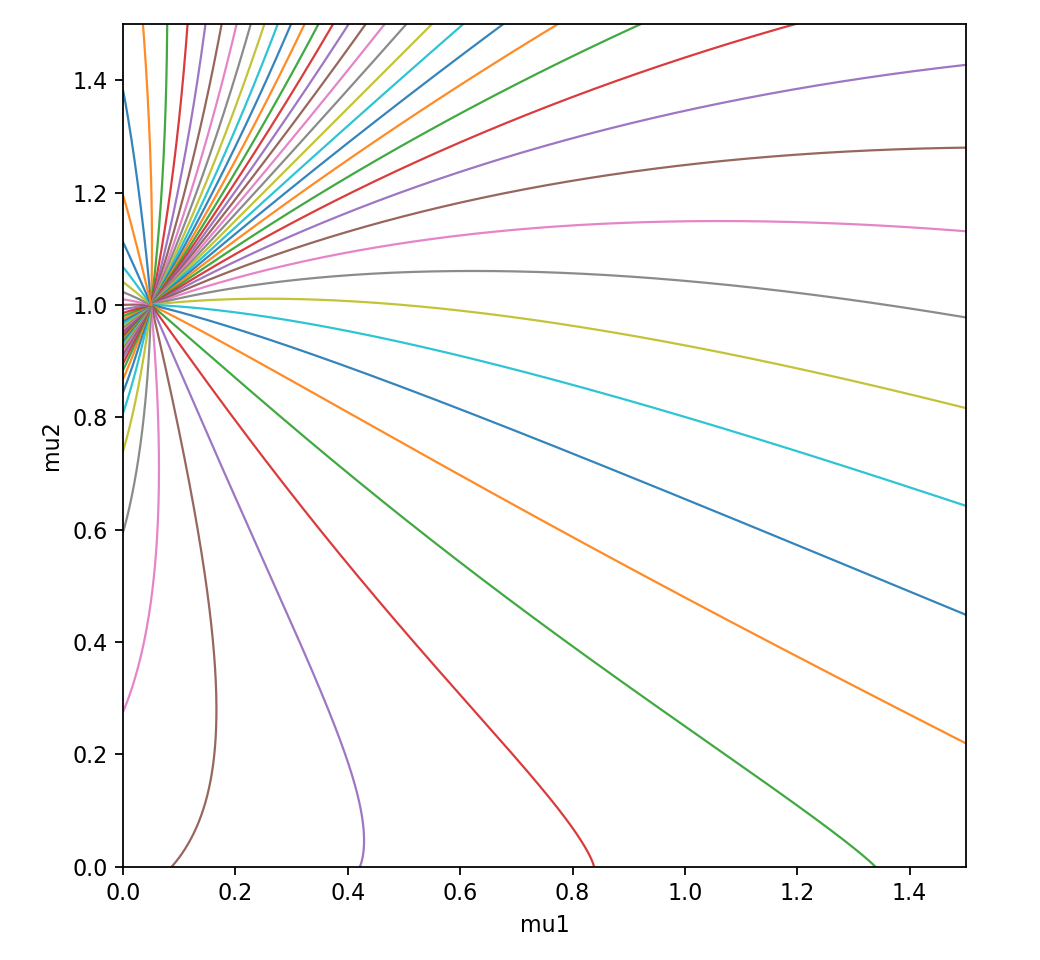}
    \captionsetup{width=\textwidth}
    \caption{Rays starting from $(0.05,1)^{\mathrm{T}}$}
    \label{fig-ray005100}
    \end{figure}
With the weak formulations \eqref{eq:GE-h-weak} and \eqref{eq:GE-psi-weak} available, 
    we pass to the strong ordinary differential system obtained by combining the continuity equation \eqref{eq_ce} with the geodesic equations. 
The resulting initial value problem reads
    \begin{equation}
    \label{eq_strong_geo}
    \left\{
    \begin{aligned}
    &
    \dot\mu_t(x) +\sum_{w}
    \nabla\psi_t(x,w)
    \theta(\mu_t(x),\mu_t(w))
    K(x,w)
    = h_t p(x),
    \\
    & \dot h_t
    = -\frac{b^2}{2 a^2}
    \sum_{z,w} \nabla\psi_t(z,w)^2  \partial_1\theta(\mu_t(z),\mu_t(w))K(z,w) \varpi(z)  p(z),
    \\
    &
    \nabla \dot \psi_t(x,y) 
    =
    \frac{1}{2}
    \sum_{z}
    \nabla\psi_t(z,x)^2
    K(x,z)
    \partial_1 \theta(\mu_t(x),\mu_t(z))
    -
    \frac{1}{2}
    \sum_{z}
    \nabla\psi_t(z,y)^2
    K(y,z)
    \partial_1 \theta(\mu_t(y),\mu_t(z)),
    \\
    &
    \mu_t|_{t=0}=\mu_0,
    \quad
    h_t|_{t=0}=h_0,
    \quad
    \nabla\psi_t|_{t=0}=\nabla\psi_0.
    \end{aligned}
    \right.
    \end{equation}
Based on this strong system we conduct numerical experiments that 
    integrate geodesic rays emanating from an interior state 
    and visualize their trajectories. The experiments proceed as follows: 
    we specify the background data
    \begin{equation*}
    K=
    \begin{pmatrix}
    0.8 & 0.2 \\
    0.4 & 0.6
    \end{pmatrix},
    \qquad
    \varpi=
    \begin{pmatrix}
    2/3 \\ 1/3
    \end{pmatrix},
    \qquad
    p=
    \begin{pmatrix}
    1 \\ 1
    \end{pmatrix},
    \qquad
    a=b=1,
    \end{equation*}
    we generate initial directions on the unit circle by setting
    \begin{equation*}
    v^{(k)}=(\psi_{\mathrm{free}},h)\in\mathbb{R}^2,\qquad
    v^{(k)}=(\cos\theta_k,\sin\theta_k),\qquad
    \theta_k=\frac{2\pi k}{n_{\mathrm{rays}}},\quad k=0,\ldots,n_{\mathrm{rays}}-1,
    \end{equation*}
    we embed the free potential into the full variable by taking 
    $\psi^{(k)}:=(0,\psi_{\mathrm{free}}^{(k)})^{\mathrm T}$ 
    and rescale $(\psi^{(k)},h^{(k)})$ to $(\psi_0^{(k)},h_0^{(k)})$ so that the initial speed constraint
    \begin{equation*}
    a^2\big(h_0^{(k)}\big)^2+b^2\,[A_{\mu_\ast}\psi_0^{(k)},\psi_0^{(k)}]=1
    \end{equation*}
    holds for the chosen interior starting point $\mu_\ast$, 
    and we set $n_{\mathrm{rays}}=72$ to achieve a dense, uniform coverage of directions around $\mu_\ast$.

The time evolution for each direction is obtained by 
    a fourth–order Runge–Kutta method with adaptive step control applied to \eqref{eq_strong_geo}. 
The numerical trajectory is terminated 
    when the first among the following events occurs: 
    (1) the final time $t=T_{\max}=3$ is reached, 
    (2) a component touches the boundary in the sense that $\min_i\mu_t(i)\le \varepsilon_{\mathrm{bd}}=10^{-6}$, 
    or (3) the adaptive time step decreases below $\mathrm d t_{\min}=5\times 10^{-4}$. 
The latter stopping event corresponds to a Runge–Kutta failure caused by stiffness near the boundary; 
    this stiffness is induced by the logarithmic mean $\theta(u,v)$
    and its partial derivative $\partial_1\theta(u,v)$, 
    which behave like $1/u$ or $1/v$ as either argument approaches zero, 
    and consequently amplify the coefficients entering $\nabla\dot\psi_t$ and $\dot h_t$ to the point that the adaptive controller enforces prohibitively small steps and stops.

The experimental outcome is summarized by ray plots from two interior points. 
In the first case (see Figure \ref{fig-ray060080}), rays start from $(0.6,0.8)^{\mathrm T}$.
In the second case (see Figure \ref{fig-ray005100}), rays start from $(0.05,1)^{\mathrm T}$.
The trajectories exhibit visible bending away from the boundary of the state space as they evolve.
\end{remark}

\section{Duality formula with Hamilton--Jacobi subsolutions}
\label{sec_duality}
In this section, we present a duality result for $\al W_p^{a,b}$ in terms of Hamilton--Jacobi subsolutions.
In \cite{erbar2019geometry}, the authors use such a duality principle to prove weak locality of geodesics.
In our setting, we have already shown in Section \ref{sec_geodesic} that weak locality always fails.
Nevertheless, we view the dual formula as a relevant and important result, as it provides a characterization of the metric and highlights the role of the momentum associated with geodesic curves.
The proof follows the same skeleton as \cite[Theorem 3.3]{erbar2019geometry}, with additional arguments to handle the source term.

\begin{definition}
A function $\phi\in \mathrm{H}^1((0,T);\mathbb{R}^{\al X})$ is called a \textit{Hamilton--Jacobi subsolution} if, for almost every $t\in(0,T)$, it holds that
    \begin{equation}
    \label{ineq_Hamliton-Jacobi}
    \langle \dot \phi_t,\mu \rangle_\varpi + \frac{1}{2 b^2}\,\| \nabla \phi_t\|_\mu^2 \le 0,
    \qquad \text{for all } \mu \in \al M.
    \end{equation}
We denote by $\mathsf{HJ}_{\al X}^{T,b}$ the set of all Hamilton--Jacobi subsolutions.
\end{definition}

For a normed vector space $E$ with topological dual $E^\ast$ 
    and a proper convex function $f:E\to\mathbb{R}\cup\{+\infty\}$, 
    the Legendre--Fenchel transform $f^\ast:E^\ast\to\mathbb{R}\cup\{+\infty\}$ is given by
    \begin{equation*}
    f^\ast(z^\ast)
    := \sup_{z\in E}\bigl\{\langle z,z^\ast\rangle - f(z)\bigr\},
    \qquad z^\ast\in E^\ast.
    \end{equation*}
We now recall the classical Fenchel--Rockafellar duality theorem \cite[Corollary 2.8.5]{zalinescu2002convex}, which will be used later in the proofs of the saddle-point statement and the duality result.
\begin{lemma}[Fenchel--Rockafellar duality]
\label{lem_FR}
Let $E$ and $F$ be locally convex topological vector spaces, and let $A:E\to F$ be a continuous linear map.
Let $f:E\to\mathbb{R}\cup\{+\infty\}$ and $g:F\to\mathbb{R}\cup\{+\infty\}$ be proper, convex, and lower semicontinuous.
Assume that there exists $x_{0}\in E$ such that $f(x_{0})<\infty$ and that $g$ is continuous at $Ax_{0}$.
Then
    \begin{equation*}
    \inf_{x\in E}\bigl[f(x)+g(Ax)\bigr]
    =
    \sup_{y\in F^{\ast}}\bigl[-f^{\ast}(-A^{\ast}y)-g^{\ast}(y)\bigr].
    \end{equation*}
\end{lemma}

With the above preparations, we now establish the following saddle point result.
\begin{lemma}[Saddle point]
\label{lem_saddle}
Let $X$ be a locally convex topological vector space.
Let $Y$ be a reflexive Banach space and denote by $\langle \cdot,\cdot\rangle$ the dual pairing between $Y^{\ast}$ and $Y$.
Let $J:X\to\mathbb{R}\cup\{-\infty\}$ be concave and upper semicontinuous, and assume that $J\not\equiv-\infty$.
Let $K:X\to Y^{\ast}$ be a continuous linear map.
Let $\Psi:Y\to\mathbb{R}\cup\{+\infty\}$ be proper, convex, and weakly lower semicontinuous.
Define the saddle function $L:X\times Y\to\mathbb{R}\cup\{\pm\infty\}$ by
    \begin{equation*}
    L(\phi,h)=J(\phi)-\langle K\phi,h\rangle+\Psi(h).
    \end{equation*}
Assume that there exists $\phi_{0}\in X$ such that $J(\phi_{0})>-\infty$ and the convex function $\Psi^{\ast}$ is continuous at $K\phi_{0}$.
Then the minimax identity holds in $\mathbb{R}\cup\{\pm\infty\}$:
    \begin{equation*}
    \sup_{\phi\in X}\inf_{h\in Y} L(\phi,h)
    =
    \inf_{h\in Y}\sup_{\phi\in X} L(\phi,h).
    \end{equation*}
\end{lemma}

\begin{proof}
For every $\phi\in X$ we compute the inner infimum in $h$.
By definition of the Legendre--Fenchel transform,
    \begin{equation*}
    \Psi^{\ast}(\eta)=\sup_{h\in Y}\bigl[\langle \eta,h\rangle-\Psi(h)\bigr]
    \qquad
    \text{for all }\eta\in Y^{\ast}.
    \end{equation*}
Hence, for fixed $\phi$ we have
    \begin{equation*}
    \inf_{h\in Y}\bigl[\Psi(h)-\langle K\phi,h\rangle\bigr]
    =
    -\Psi^{\ast}(K\phi).
    \end{equation*}
Therefore
    \begin{equation}
    \label{eq_left_value}
    \sup_{\phi\in X}\inf_{h\in Y}L(\phi,h)
    =
    \sup_{\phi\in X}\bigl[J(\phi)-\Psi^{\ast}(K\phi)\bigr].
    \end{equation}
Set $f:=-J$.
Then $f:X\to\mathbb{R}\cup\{+\infty\}$ is proper, convex, and lower semicontinuous.
Using \eqref{eq_left_value}, we can rewrite the right-hand side as
    \begin{equation}
    \label{eq_left_convex}
    \sup_{\phi\in X}\inf_{h\in Y}L(\phi,h)
    =
    -\inf_{\phi\in X}\bigl[f(\phi)+\Psi^{\ast}(K\phi)\bigr].
    \end{equation}

We apply Lemma \ref{lem_FR} with $E=X$, $F=Y^{\ast}$, $A=K$, $g=\Psi^{\ast}$,
    along with $f=-J$ already set.
The qualification assumption in Lemma \ref{lem_FR} is satisfied by the hypothesis that $J(\phi_{0})>-\infty$ and $\Psi^{\ast}$ is continuous at $K\phi_{0}$.
We obtain
    \begin{equation}
    \label{eq_FR_applied}
    \inf_{\phi\in X}\bigl[f(\phi)+\Psi^{\ast}(K\phi)\bigr]
    =
    \sup_{h\in (Y^{\ast})^{\ast}}\bigl[-f^{\ast}(-K^{\ast}h)-(\Psi^{\ast})^{\ast}(h)\bigr].
    \end{equation}
Since $Y$ is reflexive, we identify $(Y^{\ast})^{\ast}=Y$.
Since $\Psi$ is convex and weakly lower semicontinuous, it is lower semicontinuous for the norm topology as well,
    and hence $(\Psi^{\ast})^{\ast}=\Psi$.
Combining \eqref{eq_left_convex} and \eqref{eq_FR_applied} yields
    \begin{equation}
    \label{eq_dual_value}
    \sup_{\phi\in X}\inf_{h\in Y}L(\phi,h)
    =
    \inf_{h\in Y}\bigl[f^{\ast}(-K^{\ast}h)+\Psi(h)\bigr].
    \end{equation}

It remains to identify $f^{\ast}(-K^{\ast}h)$ as a supremum over $\phi$.
By definition of the conjugate,
    \begin{equation*}
    f^{\ast}(\xi)=\sup_{\phi\in X}\bigl[\langle \xi,\phi\rangle-f(\phi)\bigr]
    \qquad
    \text{for all }\xi\in X^{\ast}.
    \end{equation*}
Using $f=-J$ and $\langle -K^{\ast}h,\phi\rangle=-\langle K\phi,h\rangle$, we get
    \begin{equation*}
    f^{\ast}(-K^{\ast}h)
    =
    \sup_{\phi\in X}\bigl[J(\phi)-\langle K\phi,h\rangle\bigr]
    =
    \sup_{\phi\in X}L(\phi,h)-\Psi(h).
    \end{equation*}
Plugging this into \eqref{eq_dual_value} gives
    \begin{equation*}
    \sup_{\phi\in X}\inf_{h\in Y}L(\phi,h)
    =
    \inf_{h\in Y}\sup_{\phi\in X}L(\phi,h),
    \end{equation*}
    which is the desired minimax identity.
\end{proof}

We now state the Kantorovich-type dual representation, extending the result and proof of \cite[Theorem 3.3]{erbar2019geometry} by dealing with the source term in our setting.
\begin{theorem}[Duality formula]
\label{thm_duality}
For any $\mu_0,\mu_1\in \al M$,
    \begin{equation}
    \label{eq_duality}
    \frac{1}{2}\,\al W_p^{a,b}(\mu_0,\mu_1)^2
    =
    \sup\left\{
    \langle \phi_1,\mu_1\rangle_\varpi
    -
    \langle \phi_0,\mu_0\rangle_\varpi
    -
    \frac{1}{2a^2}\int_0^1 \langle \phi_t,p\rangle_\varpi^2\,\di t
    \ \middle|\
    \phi\in \mathsf{HJ}_{\al X}^{1,b}
    \right\}.
    \end{equation}
Moreover, the same identity holds if the supremum is taken only over
    $\phi\in \mathrm C^{1}([0,1];\bb R^{\al X})$ satisfying \eqref{ineq_Hamliton-Jacobi}.
\end{theorem}

\begin{proof}
Since \eqref{ineq_Hamliton-Jacobi} is convex in the pair $(\dot\phi_t,\nabla\phi_t)$ and is imposed pointwise in $t$,
    we may regularize $\phi$ in time by convolution with a nonnegative mollifier.
The Jensen inequality then shows that the mollified functions still satisfy \eqref{ineq_Hamliton-Jacobi}.
Moreover, the mollification converges uniformly on $[0,1]$.
Therefore every $\phi\in \mathsf{HJ}_{\al X}^{1,b}$ admits a sequence of $\mathrm C^1$ approximations in $\mathsf{HJ}_{\al X}^{1,b}$
    that converges uniformly to $\phi$.
This yields the last assertion of the theorem.
Accordingly, it suffices to establish the dual representation with $\mathrm C^{1}$ test functions.
We split the proof into four steps.
\vspace{0.2cm}

\noindent$\bullet$ {\textbf{Step 1. Rewrite \eqref{eq_duality} with Fenchel--Rockafellar duality.}}
We begin by rewriting the supremum in \eqref{eq_duality} as a convex-analytic problem on a suitable Banach space.

Recall that $\mathscr G_{\mathbf 1}$ is the quotient of $\bb R^{\al X\times \al X}$ obtained by identifying two functions
whenever they coincide on every pair $(x,y)$ with $K(x,y)>0$.
Endowed with $\langle\cdot,\cdot\rangle_{\mathbf 1}=\langle\cdot,\cdot\rangle_{\varpi}$,
    the space $(\mathscr G_{\mathbf 1},\langle\cdot,\cdot\rangle_{\varpi})$ is a Hilbert space.
We then consider the Banach space
    \begin{equation*}
    \bb E:=\mathrm C^{1}([0,1];\bb R^{\al X})\times \mathrm L^{2}([0,1];\mathscr G_{\mathbf 1}).
    \end{equation*}
Here, we endow $\mathrm C^1([0,1];\bb R^{\al X})$ with the norm
    \begin{equation*}
    \|\phi\|_{\mathrm C^1}
    =
    \|\phi_0\|_{\varpi,1}
    +
    \|\dot\phi\|_{\mathrm C^0([0,1];\|\cdot\|_{\varpi,1})}.
    \end{equation*}
With this choice, the map $\phi\mapsto(\phi_0,\dot\phi)$ is an isometric isomorphism
    from $\mathrm C^1([0,1];\bb R^{\al X})$ onto $\bb R^{\al X}\times \mathrm C^0([0,1];\bb R^{\al X})$.
Consequently, the dual space of $ \mathrm C^1([0,1];\bb R^{\al X}) $ identifies with
    $\bb R^{\al X}\times \mathrm M([0,1];\bb R^{\al X})$ through the Riesz representation theorem,
    where $\mathrm M([0,1];\bb R^{\al X})$ denotes the space of $\bb R^{\al X}$-valued finite measures on $[0,1]$.
We may identify the dual space $\bb E^\ast$ as
    \begin{equation*}
    \bb E^\ast
    =
    \bb R^{\al X}\times \mathrm M([0,1];\bb R^{\al X})\times \mathrm L^{2}([0,1];\mathscr G_{\mathbf 1}).
    \end{equation*}
For $(\phi,\Phi)\in \bb E$ and $(c,\sigma,V)\in \bb E^\ast$, the duality pairing is
    \begin{equation*}
    \langle (\phi,\Phi),(c,\sigma,V)\rangle_{\varpi}
    =
    \langle \phi_0,c\rangle_{\varpi}
    +
    \int_0^1 \langle \dot\phi_t,\di \sigma(t)\rangle_{\varpi}
    +
    \int_0^1 \langle \Phi_t,V_t\rangle_{\varpi}\,\di t.
    \end{equation*}

We introduce two extended-real-valued functionals $f_a,g_b:\bb E\to \bb R\cup\{+\infty\}$ by
    \begin{gather*}
    f_a(\phi,\Phi)=
    \left\{
    \begin{aligned}
    &
    -\langle \phi_1,\mu_1\rangle_\varpi + \langle \phi_0,\mu_0\rangle_\varpi
    +\frac{1}{2a^2}\int_0^1 \langle \phi_t,p\rangle_\varpi^2\,\di t,
    && \Phi=\nabla\phi,
    \\
    &+\infty,
    && \text{otherwise},
    \end{aligned}
    \right.
    \\
    g_b(\phi,\Phi)=
    \left\{
    \begin{aligned}
    &0,
    && (\phi,\Phi)\in \mathcal D^{b},
    \\
    &+\infty,
    && \text{otherwise}.
    \end{aligned}
    \right.
    \end{gather*}
Here $(\phi,\Phi)\in \bb E$ is said to belong to $\mathcal D^{b}$ if, for every continuous curve $t\mapsto \eta_t\in \al M$,
    \begin{equation*}
    \int_0^1
    \langle \dot\phi_t,\eta_t\rangle_\varpi
    +\frac{1}{2b^2}\|\Phi_t\|_{\eta_t}^2
    \,\di t
    \le 0.
    \end{equation*}
For $\phi\in \mathrm C^1([0,1];\bb R^{\al X})$, the condition $(\phi,\nabla\phi)\in\mathcal D^b$
    is equivalent to \eqref{ineq_Hamliton-Jacobi} holding for every $t\in[0,1]$.
The implication follows by testing the defining inequality of $\mathcal D^b$
    with curves $\eta$ supported in a small time interval and then letting the interval shrink to a point.
Conversely, integrating \eqref{ineq_Hamliton-Jacobi} in time yields the defining inequality of $\mathcal D^b$.
Consequently, the right-hand side of \eqref{eq_duality} can be rewritten as
    \begin{equation*}
    \sup_{\phi\in \mathsf{HJ}^{1,b}_{\al X}}
    \left\{
    \langle \phi_1,\mu_1\rangle_\varpi
    -
    \langle \phi_0,\mu_0\rangle_\varpi
    -
    \frac{1}{2a^2}\int_0^1 \langle \phi_t,p\rangle_\varpi^2\,\di t
    \right\}
    =
    \sup_{(\phi,\Phi)\in \bb E}
    \left\{-f_a(\phi,\Phi)-g_b(\phi,\Phi)\right\}.
    \end{equation*}

It is straightforward to verify that both $f_a$ and $g_b$ are convex.
In addition, if we set $\bar\phi_t=t(-1,-1,\dots,-1)$ and $\bar\Phi_t=\nabla\bar\phi_t\equiv \mathbf 0$,
    then $f_a(\bar\phi,\bar\Phi)$ and $g_b(\bar\phi,\bar\Phi)$ are finite,
Besides, $g_b$ is continuous at $(\bar\phi,\bar\Phi)$,
    because $g_b$ is continuous at a point $(\phi,\Phi)\in\mathcal D^b$
    whenever $(\phi,\Phi)$ lies in the interior of $\mathcal D^b$ with respect to the $\bb E$-topology.
For the special choice $\bar\phi_t=-t\mathbf 1$ and $\bar\Phi\equiv \mathbf 0$,
    one checks that the defining inequality of $\mathcal D^b$ is strict at $(\bar\phi,\bar\Phi)$,
    and the left-hand side depends continuously on $(\phi,\Phi)$ in the $\bb E$ norm.
Therefore there exists $\rho>0$ such that the open ball $B_\rho^{\bb E}(\bar\phi,\bar\Phi)\subset \mathcal D^b$,
    which implies that $g_b$ is continuous at $(\bar\phi,\bar\Phi)$.
We may therefore invoke Lemma~\ref{lem_FR} with $E=F=\bb E$ and $A=\mathrm{Id}_{\bb E}$ to obtain
    \begin{equation*}
    \sup_{(\phi,\Phi)\in \bb E}
    \left\{-f_a(\phi,\Phi)-g_b(\phi,\Phi)\right\}
    =
    \inf_{(c,\sigma,V)\in \bb E^\ast}
    \left\{f_a^\ast(c,\sigma,V)+g_b^\ast(-c,-\sigma,-V)\right\}.
    \end{equation*}

To complete the proof of \eqref{eq_duality}, it remains to compute the Legendre--Fenchel conjugates
    $f_a^\ast$ and $g_b^\ast$, and to show that the infimum on the right-hand side coincides with
    $\frac12 \al W_p^{a,b}(\mu_0,\mu_1)^2$.
\vspace{0.2cm}

\noindent$\bullet$ {\textbf{Step 2. Computation of $f_a^\ast$.}}
For the Legendre--Fenchel transform $f_a^\ast$,
    \begin{align*}
    f_a^{\ast}(c,\sigma,V)
    &=
    \sup_{(\phi,\Phi)\in \mathbb{E}}
    \left\{
    \langle(\phi,\Phi),(c,\sigma,V)\rangle_\varpi
    -
    f_a(\phi,\Phi)
    \right\}
    \\
    &=
    \sup_{\phi\in\mathrm{C}^1([0,1];\mathbb{R}^{\al X})}
    \left\{
    \langle \phi_{0}, c\rangle_\varpi
    +
    \int_{0}^{1}\langle \dot\phi_{t}, \di \sigma(t)\rangle_\varpi
    +
    \int_{0}^{1}\langle \nabla\phi_{t}, V_{t}\rangle_\varpi\di t
    \right.
    \\
    &
    \qquad \qquad \qquad \qquad
    \left.
    +
    \langle \phi_{1}, \mu_{1}\rangle_\varpi
    -
    \langle \phi_{0}, \mu_{0}\rangle_\varpi
    -
    \frac{1}{2a^2} \int_0^1 \langle \phi_t, p\rangle_\varpi^2 \di t
    \right\}
    \\
    &=
    \sup_{\phi\in\mathrm{C}^1([0,1];\mathbb{R}^{\al X})}
    \left\{
    \langle \phi_{1}, \mu_{1}\rangle_\varpi
    -
    \langle \phi_{0}, \mu_0-c\rangle_\varpi
    +
    \int_{0}^{1}\langle \dot\phi_{t}, \di \sigma(t)\rangle_\varpi
    +
    \int_{0}^{1}\langle \nabla\phi_{t}, V_{t}\rangle_\varpi\di t
    \right.
    \\
    &
    \qquad \qquad \qquad \qquad
    \left.
    -
    \frac{1}{2a^2} \int_0^1 \langle \phi_t, p\rangle_\varpi^2 \di t
    \right\}
    \\
    &=
    \sup_{\phi\in\mathrm{C}^1([0,1];\mathbb{R}^{\al X})}
    \left\{
    \langle \phi_{1}, \mu_{1}\rangle_\varpi
    -
    \langle \phi_{0}, \mu_0-c \rangle_\varpi
    +
    \int_{0}^{1}\langle \dot\phi_{t}, \di \sigma(t)\rangle_\varpi
    +
    \int_{0}^{1}\langle \nabla\phi_{t}, V_{t}\rangle_\varpi\di t
    \right.
    \\
    &
    \qquad \qquad \qquad \qquad
    \left.
    +
        \inf_{h\in \mathrm{L}^2([0,1];\mathbb{R})}
        \left\{
        \frac{a^2}{2} \int_0^1 h_t^2 \di t
        -
        \int_0^1 \langle \phi_t,h_t p\rangle_\varpi \di t
        \right\}
    \right\}.
    \end{align*}
The last expression is from the rewrite:
    \begin{equation*}
    \frac{a^2}{2} h_t^2
    -
    \langle \phi_t,h_t p\rangle_\varpi
    =
    \frac{a^2}{2} \left(h_t-\frac{1}{a^2}\langle\phi_t,p \rangle_\varpi\right)^2
    -
    \frac{1}{2a^2}\langle \phi_t, p\rangle_\varpi^2.
    \end{equation*}
By the saddle point result in Lemma \ref{lem_saddle}, we can exchange the supremum and infimum operators and get
    \begin{align*}
    f_a^{\ast}(c,\sigma,V)
    &=
    \inf_{h\in \mathrm{L}^2([0,1];\mathbb{R})}
    \Bigg\{
    \sup_{\phi\in\mathrm{C}^1([0,1];\mathbb{R}^{\al X})}
    \bigg\{
        \int_{0}^{1}\langle \dot\phi_{t}, \di \sigma(t)\rangle_\varpi
        +
        \int_{0}^{1}\langle \nabla\phi_{t}, V_{t}\rangle_\varpi\di t
        -
        \int_0^1 \langle \phi_t,h_t p\rangle_\varpi \di t
    \\
    &
    \qquad \qquad \qquad \qquad
        \langle \phi_{1}, \mu_{1}\rangle_\varpi
        -
        \langle \phi_{0}, \mu_0-c \rangle_\varpi
    \bigg\}
    +\frac{a^2}{2} \int_0^1 h_t^2 \di t
    \Bigg\}.
    \end{align*}
Hence, by the homogeneity of the last expression in $\phi$,
we necessarily have $f_a^\ast(c,\sigma,V)=+\infty$ unless there exists a function $h$ such that
$(\sigma,V,h)$ satisfies the continuity equation
    \begin{equation*}
    \partial_t \sigma(t)+\nabla\cdot V_t=-h_t p
    \end{equation*}
    with boundary data $-(\mu_0-c)$ and $-\mu_1$, in the weak sense that
    \begin{equation}
    \label{eq_weak_ce_duality}
    \langle \phi_{1}, -\mu_{1}\rangle_\varpi
    -
    \langle \phi_{0}, -(\mu_{0}-c)\rangle_\varpi
    =
    \int_{0}^{1}\langle \dot\phi_{t}, \di\sigma(t)\rangle_\varpi
    +
    \int_{0}^{1}\langle \nabla\phi_{t}, V_{t}\rangle_\varpi \,\di t
    -
    \int_0^1 \langle \phi_t,h_t p\rangle_\varpi \,\di t
    \end{equation}
    for every $\phi\in \mathrm C^{1}([0,1];\bb R^{\al X})$.
As the distributional derivative $\partial_t \sigma(t)=-h_t p-\nabla\cdot V_t$ belongs to $\mathrm L^2([0,1];\bb R^{\al X})$,
    we have $\sigma\in \mathrm H^1([0,1];\bb R^{\al X})$.
In particular, the $\bb R^{\al X}$-valued measure $\sigma$ is absolutely continuous with respect to Lebesgue measure
    and admits a density, still denoted by $(\sigma_t)_{t\in[0,1]}$, such that $\di\sigma(t)=\sigma_t\di t$.
Moreover, \eqref{eq_weak_ce_duality} enforces the endpoint conditions $\sigma_0=-(\mu_0-c)$ and $\sigma_1=-\mu_1$.
Consequently,
    \begin{equation}
    \label{eq_Fa_Legendre_transform}
    f_a^\ast(c,\sigma,V)
    =
    \left\{
    \begin{aligned}
    &
    \inf\left\{\frac{a^2}{2}\int_0^1 h_t^2\,\di t
    \ \middle|\ 
    h\in \mathcal{CE}'_p(-\sigma,-V;\mu_0-c,\mu_1)\right\},
    && \mathcal{CE}'_p(-\sigma,-V;\mu_0-c,\mu_1)\neq\emptyset,
    \\
    &+\infty,
    && \text{otherwise},
    \end{aligned}
    \right.
    \end{equation}
    where $h\in \mathcal{CE}'_p(-\sigma,-V;\mu_0-c,\mu_1)$
    means that $h$ satisfies $(-\sigma,-V,h)\in \mathcal{CE}_p(\mu_0-c,\mu_1)$
    yet dropping the positivity requirement for $-\sigma$,
    and we identify the measure $\sigma$ with its $\mathrm H^{1}$-representative $(\sigma_t)_t$.
\vspace{0.2cm}

\noindent$\bullet$ {\textbf{Step 3. Computation of $g_b^\ast$.}}
Since we only need to evaluate $g_b^\ast$ at those $(c,\sigma,V)$ for which $f_a^\ast(-c,-\sigma,-V)$ is finite,
we may assume that $\di\sigma(t)=\sigma_t\di t$ with $(\sigma_t)_t\in \mathrm H^{1}([0,1];\bb R^{\al X})$.
We claim that
\begin{equation}
\label{eq_Gb_Legendre_transform}
g_b^\ast(c,\sigma,V)
=
\left\{
\begin{aligned}
& \frac{b^2}{2}\int_0^1 A'(\sigma_t,V_t)\di t,
&& c=0,
\\
& +\infty,
&& \text{otherwise}.
\end{aligned}
\right.
\end{equation}
We further split the verification into four sub-steps.

\noindent$\star$ \textbf{Step 3.1. If $c\neq0$ or $\sigma_t \notin \al M$, then $g_b^\ast(c,\sigma,V)=+\infty$.}
Indeed,
\begin{align*}
g_b^\ast(c,\sigma,V)
&=
\sup_{(\phi,\Phi)\in \bb E}
\left\{
\langle(\phi,\Phi),(c,\sigma,V)\rangle_\varpi
-
g_b(\phi,\Phi)
\right\}
\\
&=
\sup_{(\phi,\Phi)\in \mathcal D^{b}}
\left\{
\langle \phi_0,c\rangle_\varpi
+
\int_0^1
\big(
\langle \dot\phi_t,\sigma_t\rangle_\varpi
+
\langle \Phi_t,V_t\rangle_\varpi
\big)\di t
\right\}.
\end{align*}
Since $(\phi,\Phi)\in \mathcal D^{b}$ implies $(\phi+c',\Phi)\in \mathcal D^{b}$ for every $c'\in \bb R^{\al X}$,
it follows that $g_b^\ast(c,\sigma,V)=+\infty$ unless $c= 0$.

Moreover, the condition $\sigma_t\in \al M$ is necessary.
Indeed, assume that there exist $x_\ast\in\al X$ and a measurable set $I\subset(0,1)$ with $|I|>0$
such that $\int_I \sigma_t(x_\ast)\di t<0$.
For $k\ge 0$, define $\phi^k\in \mathrm C^1([0,1];\bb R^{\al X})$ by prescribing
\begin{equation*}
\dot\phi_t^k(x)=
\left\{
\begin{aligned}
&-k,&& t\in I,x=x_\ast,
\\
&0,&& \text{otherwise},
\end{aligned}
\right.
\qquad
\phi_0^k=0,
\end{equation*}
and set $\Phi^k\equiv \mathbf 0$.
Since $\dot\phi_t^k\le 0$ componentwise for all $t$ and $\Phi^k\equiv \mathbf 0$,
we have $(\phi^k,\Phi^k)\in \mathcal D^b$.
Therefore,
\begin{equation*}
g_b^\ast(0,\sigma,V)
\ge
\sup_{k\ge 0}\int_0^1 \langle \dot\phi_t^k,\sigma_t\rangle_\varpi\di t
=
\sup_{k\ge 0}\left\{-k\varpi(x_\ast)\int_I \sigma_t(x_\ast)\di t\right\}
=
+\infty.
\end{equation*}
Consequently, $g_b^\ast(c,\sigma,V)<\infty$ forces $\sigma_t(x)\ge 0$ for all $x\in\al X$ and $t\in[0,1]$.
\vspace{0.2cm}

\noindent$\star$ \textbf{Step 3.2. When $c=0$ and $\int_0^1 A'(\sigma_t,V_t)\di t<\infty$, $g_b^\ast(c,\sigma,V)\le$ RHS-\eqref{eq_Gb_Legendre_transform}.}
Assume now that $c=0$ and that $\int_0^1 A'(\sigma_t,V_t)\di t<\infty$.
We then obtain one inequality in \eqref{eq_Gb_Legendre_transform}:
\begin{align*}
g_b^\ast(0,\sigma,V)
&=
\sup_{(\phi,\Phi)\in \mathcal D^{b}}
\left\{
\int_0^1
\big(
\langle \dot\phi_t,\sigma_t\rangle_\varpi
+
\langle \Phi_t,V_t\rangle_\varpi
\big)\di t
\right\}
\\
&\le
\sup_{(\phi,\Phi)\in \mathcal D^{b}}
\left\{
\int_0^1
\left(
-\frac{1}{2b^2}\|\Phi_t\|_{\sigma_t}^2
+
\langle \Phi_t,V_t\rangle_\varpi
\right)\di t
\right\}
\le
\frac{b^2}{2}\int_0^1 A'(\sigma_t,V_t)\di t.
\end{align*}
Since $(\sigma_t)_t\in \mathrm H^{1}([0,1];\bb R^{\al X})$ and $\bb R^{\al X}$ is finite dimensional,
    we have $\sigma\in \mathrm C^{0}([0,1];\bb R^{\al X})$.
In particular, the curve $t\mapsto \eta_t:=\sigma_t$ is admissible in the definition of $\mathcal D^b$,
    thus the first inequality follows from the definition of $\mathcal D^{b}$.
The second inequality follows from
    \begin{align*}
    \langle \Phi,V\rangle_\varpi
    &=
    \frac12\sum_{x,y\in \al X}\Phi(x,y)V(x,y)K(x,y)\varpi(x)
    \\
    &\le
    \frac12\sum_{x,y\in \al X}
    \left[
    \frac{1}{2b^2}\Phi(x,y)^2\hat\sigma(x,y)
    +\frac{b^2}{2}\alpha(V(x,y),\sigma(x),\sigma(y))
    \right]K(x,y)\varpi(x)
    \\
    &=
    \frac{1}{2b^2}\|\Phi\|_{\sigma}^2
    +\frac{b^2}{2}A'(\sigma,V).
    \end{align*}
\vspace{0.2cm}

\noindent$\star$ \textbf{Step 3.3. When $c=0$, $g_b^\ast(c,\sigma,V)\ge$ RHS-\eqref{eq_Gb_Legendre_transform}.}
It remains to prove the matching lower bound when $c=0$.
The strategy is to construct, for small parameters $\ep>0$ and $\delta>0$,
    an explicit pair $(\phi^{\delta,\ep},\Phi^{\delta,\ep})\in\mathcal D^b$
    whose value in the dual functional approaches $(b^2/2)\int_0^1 A'(\sigma_t,V_t)\di t$.
$\int_0^1 A'(\sigma_t,V_t)\di t$ can be finite or infinite.
As this step is very lengthy, we leave them in the Appendix \ref{App-B-3.3}.
\vspace{0.2cm}

Combining Step 3.1--3.3, we conclude that \eqref{eq_Gb_Legendre_transform} is verified.
\vspace{0.2cm}

\noindent$\bullet$ \textbf{Step 4. Conclusion.}
With \eqref{eq_Fa_Legendre_transform} and \eqref{eq_Gb_Legendre_transform} now ready,
    we have
    \begin{align*}
    & f_a^\ast(-c,-\sigma,-V)+g^\ast_b(c,\sigma,V)
    \\
    &=
    \left\{
    \begin{aligned}
    & \inf_h \left\{\frac{a^2}{2} \int_0^1 h_t^2 \di t\right\} +\frac{b^2}{2}\int_0^1 A'(\sigma_t,V_t) \di t,
    && c=0 ~\text{and}~ \exists h, \text{ s.t. }(\sigma,V,h)\in\mathcal{CE}_p(\mu_0,\mu_1),
    \\
    & +\infty,
    && \text{otherwise}.
    \end{aligned}
    \right.
    \end{align*}
Thus,
    \begin{equation*}
    \inf_{(c,\sigma,V)\in\mathbb{E}^\ast} \{ f_a^\ast(-c,-\sigma,-V)+ g^\ast_b(c,\sigma,V) \}
    =
    \frac{1}{2}
    \inf_{(\sigma,V,h)\in\mathcal{CE}_p(\mu_0,\mu_1)}
    \mathfrak E_{\rm Quad}^{a,b}(\sigma,V,h)
    =
    \frac{1}{2}
    \al W_p^{a,b}(\mu_0,\mu_1)^2,
    \end{equation*}
    and \eqref{eq_duality} holds.
\end{proof}

\section{Conclusions}
\label{sec_conc}
We studied the geometry and duality of the Benamou--Brenier-type metric $\al W^{a,b}_p$
for nonnegative measures on a finite reversible Markov chain.
We proved that $(\al M,\al W^{a,b}_p)$ is a geodesic metric space and uncovered a strong nonlocal property:
for almost every time, geodesics are supported on the whole chain, independently of the endpoints.
As an application, we compared $\al W^{a,b}_p$ with shift--transport-type distances and obtained an explicit bound with a strictness criterion.
Motivated by the conjecture that full support persists for all times when the endpoints are fully supported,
we established conditional uniqueness of geodesics and derived the weak geodesic equation.
We also proved a Kantorovich-type duality formula in terms of Hamilton--Jacobi subsolutions,
yielding a variational characterization of $\al W^{a,b}_p$ adapted to the constrained source direction.
These results provide basic tools for nonconservative transport on graphs in which reservoir-like effects are intrinsic.
Several directions remain open, including closing the gap between almost-everywhere and all-time nonlocality,
extending the theory to infinite graphs, and exploring potential applications of $\al W^{a,b}_p$ in data-driven tasks such as image recognition.
We leave these questions for future work.

\section*{Acknowledgments}
The work of X. Xue is supported by the Natural Science Foundation of China (grants 12271125), and the work of X. Wang was supported by the Natural Science Foundation of China (grants 123B2003), the China Postdoctoral Science Foundation (grants 2025M774290), and Heilongjiang Province Postdoctoral Funding (grants LBH-Z24167).

\section*{Conflict of interest statement}
The authors declare no conflicts of interest.

\section*{Data availability statement}
The data supporting the findings of this study are available from the corresponding author upon reasonable request.

\section*{Ethical statement}
The authors declare that this manuscript is original, has not been published before, and is not currently being considered for publication elsewhere. The study was conducted by the principles of academic integrity and ethical research practices. All sources and contributions from others have been properly acknowledged and cited. The authors confirm that there is no fabrication, falsification, plagiarism, or inappropriate manipulation of data in the manuscript.

\appendix
\addtocontents{toc}{\protect\setcounter{tocdepth}{1}} %目录不显示附录的次级章节
\section{Proofs in Section \ref{sec_geodesic}}

\subsection{Proof of Lemma \ref{lem_alpha_convex}}
\label{sec_app_e}
The proof is divided into two steps.

\noindent$\bullet$  \textbf{Step 1. Convexity.}
Recalling the integral representation \eqref{eq_logarithmic_mean_reciprocal} of the reciprocal logarithmic mean,
    \begin{equation*}
    \frac{1}{\theta(s,t)}=\int_0^1 \frac{1}{\xi s+(1-\xi)t}\di \xi\qquad \forall~ s,t>0,
    \end{equation*}
    one has for $\theta(s,t)>0$ the integral rewrite for $\alpha$,
    \begin{equation}\label{eq:alpha-int}
    \alpha(v,s,t)=\int_0^1 \frac{v^2}{\xi s+(1-\xi)t}\di \xi.
    \end{equation}
For fixed $\xi\in[0,1]$, the mapping
    \begin{equation*}
    (v,s,t)\longmapsto \frac{v^2}{\xi s+(1-\xi)t}
    \end{equation*}
    is the composition of 
        the perspective $(\eta,y)\mapsto \eta^2/y$ 
        with the affine map $y=\xi s+(1-\xi)t$.
Hence it is convex on $\mathbb{R}\times(0,\infty)^2$, 
    and it extends to a convex, lower semi-continuous function on $\mathbb{R}\times[0,\infty)^2$ 
    via the usual convention $v^2/0=0$ if $v=0$ and $+\infty$ otherwise. 
Integrating \eqref{eq:alpha-int} over $\xi\in[0,1]$ preserves 
    convexity and lower semi-continuity for $\alpha$ on $\mathbb{R}\times[0,\infty)^2$.

\noindent$\bullet$  \textbf{Step 2. Affine cases.}
Fix a nondegenerate affine line lying in $\mathbb{R}\times(0,\infty)^2$, parametrized by
    \begin{equation*}
    \big(\mathbf v(\tau),\mathbf s(\tau),\mathbf t(\tau)\big)
    =(d,u,z)+\tau\,(c,r,w),\qquad \tau\in(0,1),
    \end{equation*}
    where $(d,u,z)\in\mathbb{R}\times(0,\infty)^2$ is a base point and $(c,r,w)\in\mathbb{R}^3\setminus\{(0,0,0)\}$ is its direction vector.
We aim to prove that $\tau \mapsto \alpha(\mathbf v(\tau),\mathbf s(\tau),\mathbf t(\tau))$ is affine on $[0,1]$ if and only if
    \begin{gather} 
    (\romannumeral 1)~ \mathbf v(\tau)\equiv 0, 
    \text{ or }
    (\romannumeral 2)~
    \exists\kappa\in\mathbb{R},
    \text{ with } 
    (d,u,z)=\kappa(c,r,w). 
    \label{eq:alpha-cases} 
    \end{gather}

By \eqref{eq:alpha-int},
    \begin{equation*}
    \begin{aligned}
    \alpha(\mathbf v(\tau),\mathbf s(\tau),\mathbf t(\tau))
    = &
    \int_0^1 \frac{ \mathbf{v}(\tau)^2}{\xi\mathbf s(\tau)+(1-\xi) \mathbf t(\tau)} \di \xi
    \\
    = &
    \int_0^1 \frac{(c\tau + d)^2}{\xi ( r\tau+u)+(1-\xi)(w\tau+z)} \di \xi
    \\
    =&
    \int_0^1 \frac{(c\tau + d)^2}{(\xi r+(1-\xi)w)\tau+\xi u+(1-\xi)z} \di \xi.
    \end{aligned}
    \end{equation*}
We abbreviate, for $\xi\in[0,1]$,
    \begin{equation*}
    Q_\xi=\xi r+(1-\xi)w,\qquad R_\xi=\xi u+(1-\xi)z,\qquad
    \phi_\xi(\tau)=\frac{(c\tau+d)^2}{Q_\xi\tau+R_\xi}.
    \end{equation*}
Note that $Q_\xi\tau+R_\xi= \xi\mathbf s(\tau)+(1-\xi) \mathbf t(\tau)$ is positive since $\mathbf s(\tau) $ and $\mathbf t(\tau)$ are in $(0,\infty)$.
A direct differentiation gives
    \begin{equation*}
    \frac{\di}{\di\tau}\phi_\xi(\tau)
    =\frac{2(c\tau+d)c(Q_\xi\tau+R_\xi)-(c\tau+d)^2Q_\xi}{(Q_\xi\tau+R_\xi)^2}
    =\frac{2(c\tau+d)c}{Q_\xi\tau+R_\xi}
    -\frac{(c\tau+d)^2Q_\xi}{(Q_\xi\tau+R_\xi)^2}.
    \end{equation*}
Differentiation again yields
    \begin{equation}
    \begin{aligned}
    \frac{\di ^2}{\di\tau^2}\phi_\xi(\tau)
    = &
    \frac{2c^2(Q_\xi\tau+R_\xi)-2cQ_\xi(c\tau+d)}{(Q_\xi\tau+R_\xi)^2}
    -\frac{2cQ_\xi(c\tau+d)(Q_\xi\tau+R_\xi)-2Q_\xi^2(c\tau+d)^2}{(Q_\xi\tau+R_\xi)^3}
    \\
    = &
    \frac{2c^2(Q_\xi\tau+R_\xi)^2-4cQ_\xi(c\tau+d)(Q_\xi\tau+R_\xi)+2Q_\xi^2(c\tau+d)^2}{(Q_\xi\tau+R_\xi)^3}
    \\
    = &
    \frac{2(c(Q_\xi\tau+R_\xi)-Q_\xi(c\tau+d))^2}{(Q_\xi\tau+R_\xi)^3}
    =
    \frac{2(R_\xi c-Q_\xi d)^2}{(Q_\xi\tau+R_\xi)^3}\ \ge\ 0,
    \end{aligned}
    \label{eq:phix-second}
    \end{equation}
    hence along the line
    \begin{equation*}
    \frac{\di ^2}{\di\tau^2}\alpha\big(\mathbf v(\tau),\mathbf s(\tau),\mathbf t(\tau)\big)
    =\int_0^1 \frac{\di ^2}{\di\tau^2}\phi_\xi(\tau)\di \xi
    =\int_0^1 \frac{2(R_\xi c-Q_\xi d)^2}{(Q_\xi\tau+R_\xi)^3}\di \xi.
    \end{equation*}
Therefore the restriction of $\alpha$ to the line is affine 
    if and only if for almost every $\xi\in[0,1]$,
    \begin{equation*}
    0= R_\xi c-Q_\xi d
    =
    (\xi u+(1-\xi)z)c
    -
    (\xi r+(1-\xi)w)d
    =
    \xi(uc-rd)
    +(1-\xi)(zc-wd),
    \end{equation*}
    which is equivalent to the pair of scalar relations
    \begin{equation}\label{eq:two-rel}
    uc=rd,\qquad
    zc=wd.
    \end{equation}
There are exactly two possibilities.

First, $c=d=0$, i.e., $\mathbf v(\tau)\equiv 0$ along the line.
Then $\alpha(\mathbf v(\tau),\mathbf s(\tau),\mathbf t(\tau))\equiv 0$. 
This yields the alternative (\romannumeral 1) in \eqref{eq:alpha-cases}.

Second, $c$ and $d$ are not both zero, 
    in which case \eqref{eq:two-rel} implies that $(r,u)$ and $(w,z)$ are both proportional to $(c,d)$.
If $c=0$ while $d\neq 0$, then $r=w=0$, 
    which is contradictory to the nondegenerate condition $(c,r,w)\neq(0,0,0)$.
Then $c\neq 0$ and we can denote $\kappa=d/c$.
Hence the whole line is contained in a ray through the origin: for $\tau\in[0,1]$,
    \begin{equation*}
    (\mathbf v(\tau),\mathbf s(\tau),\mathbf t(\tau))
    =
    \tau(c,r,w)+(d,u,z)
    =( \tau+\kappa)(c,r,w),
    \end{equation*}
    which proves the alternative (\romannumeral 2) in \eqref{eq:alpha-cases}. 

Conversely, if neither alternative holds, 
    then the integral of \eqref{eq:phix-second} is strictly positive on $(0,1)$, 
    so $\alpha$ is strictly convex along the segment. 
This establishes the equivalence cases of the affine condition.

\subsection{Proof of Lemma \ref{lem_W_rewrite}}
\label{sec_app_W_rewrite}
We aim to prove the identity
    \begin{equation*}
    ({\al W}_p^{a,b}(\mu_0,\mu_1))^2=
    \inf_{}
    \left\{
    \mathfrak E^{a,b}_{\rm Quad}((\mu_t,V_t,h_t)_{t\in[0,1]})
    \mid
    (\mu_t,V_t,h_t)_{t\in[0,1]}\in\mathcal{CE}_p(\mu_0,\mu_1;[0,1])
    \right\}.
    \end{equation*}

\noindent$\bullet$  \textbf{Step 1. The ``$\ge$'' part.}
Notice that for any $(\mu_t,\psi_t,h_t)_{t\in[0,1]}$ in $\mathrm{CE}_p(\mu_0,\mu_1;[0,1])$, 
    the triple $(\mu_t,\hat\mu_t * \nabla\psi_t,h_t)_{t\in[0,1]}$ is admissible in $\mathcal{CE}_p(\mu_0,\mu_1;[0,1])$, by the continuity equation in the form of \eqref{eq_ce_discrete_calculus}:
    \begin{equation*}
    \dot\mu_t+\nabla\cdot (\hat\mu_t *\nabla\psi_t)=h_t\,p .
    \end{equation*}
The instantaneous action in the flux formulation coincides with that in the potential formulation. If $\hat\mu_t(x,y)=0$, then
\begin{equation*}
\alpha\big((\hat\mu_t\!*\!\nabla\psi_t)(x,y),\mu_t(x),\mu_t(y)\big)=\alpha\big(0,\mu_t(x),\mu_t(y)\big)=0=\nabla\psi_t(x,y)^2\,\hat\mu_t(x,y).
\end{equation*}
If $\hat\mu_t(x,y)>0$, then $\mu_t(x),\mu_t(y)>0$ and
\begin{equation*}
\alpha\big((\hat\mu_t\!*\!\nabla\psi_t)(x,y),\mu_t(x),\mu_t(y)\big)
=\frac{((\hat\mu_t\!*\!\nabla\psi_t)(x,y))^2}{\hat\mu_t(x,y)}
=\nabla\psi_t(x,y)^2\,\hat\mu_t(x,y).
\end{equation*}
Hence
\begin{equation*}
\begin{aligned}
\frac{1}{2}\sum_{x,y\in\al X}\nabla\psi_t(x,y)^2\,\hat\mu_t(x,y)\,K(x,y)\,\varpi(x)
&=\frac{1}{2}\sum_{x,y\in\al X}\alpha\big((\hat\mu_t\!*\!\nabla\psi_t)(x,y),\mu_t(x),\mu_t(y)\big)\,K(x,y)\,\varpi(x)\\
&=A'(\mu_t,\hat\mu_t\!*\!\nabla\psi_t).
\end{aligned}
\end{equation*}
Therefore the total actions agree:
\begin{equation*}
E^{a,b}_{\rm Quad}\big((\mu_t,\psi_t,h_t)_{t\in[0,1]}\big)
=
\mathfrak E^{a,b}_{\rm Quad}\big((\mu_t,\hat\mu_t\!*\!\nabla\psi_t,h_t)_{t\in[0,1]}\big).
\end{equation*}
Thus
\begin{equation*}
\big({\al W}_p^{a,b}(\mu_0,\mu_1)\big)^2
\ge
\inf\left\{
\mathfrak E^{a,b}_{\rm Quad}\big((\mu_t,V_t,h_t)_{t\in[0,1]}\big)
\ \middle|\ 
(\mu_t,V_t,h_t)_{t\in[0,1]}\in\mathcal{CE}_p(\mu_0,\mu_1;[0,1])
\right\}.
\end{equation*}

\noindent$\bullet$  \textbf{Step 2. The ``$\le$'' part.}
Fix an arbitrary $(\mu_t,V_t,h_t)_{t\in[0,1]}\in \al{CE}_p(\mu_0,\mu_1;[0,1])$.
To overcome the regularity obstruction in recovering $\nabla\psi_t(x,y)$ from $V_t(x,y)/\hat\mu_t(x,y)$, we first build a time–regularized approximation.

\noindent$\star$ \textbf{Step 2.1. Construction of $(\mu_t^\ep,V_t^\ep,h_t^\ep)_{t\in[0,1]}$.}
For $0<\ep<\tfrac12$, define
\begin{equation}\label{eq_W_rewrite_concatenation}
(\tilde\mu_t,\tilde V_t,\tilde h_t):=
\begin{cases}
(\mu_0,0,0), & t\in[-\ep,\ep),\\
\big(\mu_{\frac{t-\ep}{1-2\ep}},\tfrac{1}{1-2\ep}V_{\frac{t-\ep}{1-2\ep}},\tfrac{1}{1-2\ep}h_{\frac{t-\ep}{1-2\ep}}\big), & t\in[\ep,1-\ep),\\
(\mu_1,0,0), & t\in[1-\ep,1+\ep].
\end{cases}
\end{equation}
Set
\begin{equation}\label{eq_W_rewrite_convolution}
(\mu_t^\ep,V_t^\ep,h_t^\ep):=\int_{-\ep}^{\ep} \eta(s)\,(\tilde\mu_{t+s},\tilde V_{t+s},\tilde h_{t+s}) \,\di s,
\end{equation}
where $\eta:\mathbb{R}\to\mathbb{R}_+$ is smooth, supported in $[-\ep,\ep]$, strictly positive on $(-\ep,\ep)$, and satisfies $\int \eta(s)\,\di s=1$.
Then $(\mu_t^\ep,V_t^\ep,h_t^\ep)_{t\in[0,1]}\in \al{CE}_p(\mu_0,\mu_1;[0,1])$ because the continuity equation holds on each subinterval in \eqref{eq_W_rewrite_concatenation} and convolution in time preserves it.

By convexity of $\alpha$ in Lemma \ref{lem_alpha_convex} and Jensen's inequality in the time variable, the convolution yields
\begin{equation}\label{ineq_W_rewrite_approx_1}
\begin{aligned}
\int_0^1 A'(\mu_t^\ep,V_t^\ep)\,\di t
&\le \int_0^1 \int_{-\ep}^{\ep} \eta(s)\,A'(\tilde \mu_{t+s},\tilde V_{t+s}) \,\di s \,\di t\\
&\le \int_{-\ep}^{\ep} \eta(s)\,\di s \int_{-\ep}^{1+\ep}A'(\tilde \mu_{t},\tilde V_{t})\,\di t\\
&= \int_{\ep}^{1-\ep} A'\!\left(\mu_{\frac{t-\ep}{1-2\ep}},\tfrac{1}{1-2\ep}V_{\frac{t-\ep}{1-2\ep}}\right)\,\di t\\
&=\frac{1}{1-2\ep}\int_0^1 A'(\mu_{t},V_{t})\,\di t.
\end{aligned}
\end{equation}
The last identity follows from the change of variables.
Similarly,
\begin{equation}\label{ineq_W_rewrite_approx_2}
\begin{aligned}
\int_0^1 |h_t^\ep|^2 \,\di t
&\le \int_0^1 \int_{-\ep}^{\ep} \eta(s)\,|\tilde h_{t+s}|^2 \,\di s \,\di t\\
&\le \int_{-\ep}^{1+\ep} |\tilde h_{t}|^2 \,\di t
=\frac{1}{1-2\ep}\int_0^1 |h_t|^2 \,\di t.
\end{aligned}
\end{equation}

\noindent$\star$ \textbf{Step 2.2. $V_t^\ep/\hat\mu_t^\ep$ is well-defined.}
Assume without loss of generality that $\int_0^1 A'(\mu_t,V_t)\,\di t<\infty$.
By definition,
\begin{equation*}
A'(\mu_t,V_t)=\frac{1}{2}\sum_{x,y\in \al X}\alpha\big(V_t(x,y),\mu_t(x),\mu_t(y)\big)\,K(x,y)\varpi(x),
\end{equation*}
and $\alpha\big(V_t(x,y),\mu_t(x),\mu_t(y)\big)\ge 0$ with equality $+\infty$ if and only if $V_t(x,y)\ne 0$ while $\hat\mu_t(x,y)=0$.
Hence finiteness of $\int_0^1 A'(\mu_t,V_t)\,\di t$ implies that for all $x,y$ with $K(x,y)\ne 0$ the set
\begin{equation}\label{eq_W_rewrite_zero_1}
\{\,t\in[0,1]\mid \hat \mu_t(x,y)=0 \text{ and } V_t(x,y)\ne 0 \,\}
\text{ is negligible.}
\end{equation}

Suppose for some $x,y$ with $K(x,y)\ne 0$ the set
\begin{equation*}
\{\,t\in[0,1]\mid \hat \mu_t^\ep(x,y)=\theta(\mu_t^\ep(x),\mu_t^\ep(y))=0 \text{ and } V_t^\ep(x,y)\ne 0 \,\}
\end{equation*}
is nonempty and fix $t$ in it.
By the definition of $\theta$ in \eqref{eq_logarithmic_mean} this means
\begin{equation*}
\mu_t^\ep(x)=0 \text{ or } \mu_t^\ep(y)=0 \quad \text{and} \quad V_t^\ep(x,y)\ne 0.
\end{equation*}
We may assume $\mu_t^\ep(x)=0$.
From \eqref{eq_W_rewrite_convolution} and the nonnegativity of $\mu$ we obtain
\begin{equation*}
\tilde\mu_s(x)=0 \text{ for all } s\in[t-\ep,t+\ep],
\end{equation*}
and from $V_t^\ep(x,y)\ne 0$ we obtain the existence of a subset $I_t\subset[t-\ep,t+\ep]$ of positive measure such that
\begin{equation*}
\tilde V_s(x,y)\ne 0 \text{ for all } s\in I_t.
\end{equation*}
In particular,
\begin{equation}\label{eq_W_rewrite_zero_2}
\tilde\mu_s(x)=0 \text{ and } \tilde V_s(x,y)\ne 0 \text{ for all } s\in I_t.
\end{equation}
By \eqref{eq_W_rewrite_concatenation} we have $\tilde V_s(x,y)=0$ on $[-\ep,\ep)\cup[1-\ep,1+\ep]$, hence $I_t\subset[\ep,1-\ep)$.
Define the bijection $\varsigma(s):=(s-\ep)/(1-2\ep)$.
Then $\varsigma(I_t)$ has positive measure and
\begin{equation*}
\varsigma(I_t)\subset \varsigma([\ep,1-\ep))=[0,1).
\end{equation*}
By \eqref{eq_W_rewrite_concatenation} and \eqref{eq_W_rewrite_zero_2}, for all $s$ in a non-negligible set $\varsigma(I_t)\subset[0,1]$,
\begin{equation*}
\mu_s(x)=0 \quad \text{and} \quad V_s(x,y)\ne 0,
\end{equation*}
which contradicts \eqref{eq_W_rewrite_zero_1}.
Therefore for any $(x,y)$ with $K(x,y)>0$,
\begin{equation}\label{eq_W_rewrite_zero_3}
\{\,t\in[0,1]\mid \hat \mu_t^\ep(x,y)=0 \text{ and } V_t^\ep(x,y)\ne 0 \,\}=\emptyset.
\end{equation}

Define the time measurable function $\Psi_t^\ep(x,y)$ by
\begin{equation*}
\Psi_t^\ep(x,y):=
\begin{cases}
\dfrac{V^\ep_t(x,y)}{\hat \mu_t^\ep(x,y)}, & \text{if } K(x,y)>0 \text{ and } \hat \mu_t^\ep(x,y)>0,\\
0, & \text{otherwise}.
\end{cases}
\end{equation*}
By \eqref{eq_W_rewrite_zero_3} and the definition of $\Psi_t^\ep$, for any $(x,y)$ with $K(x,y)>0$,
\begin{equation*}
V_t^\ep(x,y)=\hat \mu_t^\ep(x,y)\,\Psi_t^\ep(x,y).
\end{equation*}
Using reversibility \eqref{eq_markov_reversible}, whenever $K(x,y)>0$ we also have $K(y,x)>0$, hence
\begin{equation}\label{eq_Psi=V_1}
V_t^\ep(x,y)-V_t^\ep(y,x)=\hat\mu_t^\ep(x,y)\,\Psi_t^\ep(x,y)-\hat \mu_t^\ep(y,x)\,\Psi_t^\ep(y,x)
=(\Psi_t^\ep(x,y)-\Psi_t^\ep(y,x)) \hat\mu_t^\ep(x,y).
\end{equation}
Finally,
\begin{equation}\label{eq_Psi=V_2}
\Psi_t^\ep(x,y)^2\,\hat \mu_t^\ep(x,y)=\alpha\big(V_t^\ep(x,y),\mu_t^\ep(x),\mu_t^\ep(y)\big).
\end{equation}

\noindent$\star$ \textbf{Step 2.3. $\psi^\ep_t$ recovered by projection.}
Recall that $\mathscr{G}_{\mu_t^\ep}$ consists of equivalence classes of functions in $\mathbb{R}^{\al X\times \al X}$ that agree at every pair $(x,y)$ with ${\hat\mu_t^\ep}(x,y)K(x,y)>0$.
Equipped with $\langle\cdot,\cdot\rangle_{\mu_t^\ep}$, the space $\mathscr{G}_{\mu_t^\ep}$ is a Hilbert space.
The gradient operator $\nabla:\mathrm L^2(\al X,\varpi)\to \mathscr{G}_{\mu_t^\ep}$ is a well defined linear map induced by the quotient, and its negative adjoint is $\nabla_{\mu_t^\ep}:\mathscr{G}_{\mu_t^\ep}\to \mathrm L^2(\al X,\varpi)$.
We have the orthogonal decomposition
\begin{equation}\label{eq_orthogonal_decomposition_gmu}
\mathscr{G}_{\mu_t^\ep}=\mathrm{Ran}(\nabla)\oplus^{\perp}\mathrm{Ker}(\nabla_{\mu_t^\ep}\cdot).
\end{equation}
Let $\mathscr P_{\mu_t^\ep}$ be the orthogonal projection in $\mathscr{G}_{\mu_t^\ep}$ onto $\mathrm{Ran}(\nabla)$.
There exists a measurable function $\psi^\ep_t:[0,1]\to \mathbb{R}^{\al X}$ such that $\mathscr P_{\mu_t^\ep}\Psi_t^\ep=\nabla \psi_t^\ep$.

Using \eqref{eq_Psi=V_1},
\begin{equation*}
\begin{aligned}
\nabla \cdot V_t^\ep(x)
&=\frac{1}{2}\sum_{y\in \al X}\big(V_t^\ep(x,y)-V_t^\ep(y,x)\big)K(x,y)\\
&=\frac{1}{2}\sum_{y\in \al X}\big(\Psi_t^\ep(x,y)-\Psi_t^\ep(y,x)\big)\hat \mu_t^\ep(x,y)K(x,y)
=\nabla_{\mu_t^\ep}\cdot \Psi_t^\ep .
\end{aligned}
\end{equation*}
By \eqref{eq_orthogonal_decomposition_gmu},
\begin{equation*}
\nabla \cdot V_t^\ep
=\nabla_{\mu_t^\ep}\cdot \Psi_t^\ep
=\nabla_{\mu_t^\ep}\cdot\big(\Psi_t^\ep-\mathscr P_{\mu_t^\ep}\Psi_t^\ep\big)
+\nabla_{\mu_t^\ep}\cdot\big(\mathscr P_{\mu_t^\ep}\Psi_t^\ep\big)
=\nabla_{\mu_t^\ep}\cdot \nabla \psi_t^\ep
=\nabla\cdot\big(\hat \mu_t^\ep * \nabla \psi_t^\ep\big).
\end{equation*}
Hence the potential continuity equation \eqref{eq_ce_discrete_calculus} holds:
\begin{equation*}
\dot \mu^\ep_t+\nabla_{\mu_t^\ep}\cdot \nabla \psi_t^\ep
=\dot \mu^\ep_t+\nabla \cdot V_t^\ep
=h_t^\ep\,p ,
\end{equation*}
so $(\mu^\ep_t,\psi^\ep_t,h^\ep_t)_{t\in[0,1]}\in \mathrm{CE}_p(\mu_0,\mu_1;[0,1])$.

Using \eqref{eq_orthogonal_decomposition_gmu} and \eqref{eq_Psi=V_2},
\begin{equation*}
\begin{aligned}
\langle \nabla\psi_t^\ep,\nabla\psi_t^\ep\rangle_{\mu_t^\ep}
&=\langle \mathscr P_{\mu_t^\ep}\Psi_t^\ep,\mathscr P_{\mu_t^\ep}\Psi_t^\ep\rangle_{\mu_t^\ep}
\le \langle \Psi_t^\ep,\Psi_t^\ep\rangle_{\mu_t^\ep}\\
&=\frac{1}{2}\sum_{x,y\in\al X}\Psi_t^\ep(x,y)^2\,\hat \mu_t^\ep(x,y)\,K(x,y)\,\varpi(x)
\\
&=\frac{1}{2}\sum_{x,y\in\al X}\alpha\big(V_t^\ep(x,y),\mu_t^\ep(x),\mu_t^\ep(y)\big)\,K(x,y)\,\varpi(x)
=A'(\mu_t^\ep,V_t^\ep).
\end{aligned}
\end{equation*}
Combining this with \eqref{ineq_W_rewrite_approx_1} and \eqref{ineq_W_rewrite_approx_2} yields,
\begin{equation*}
\begin{aligned}
E_{\rm Quad}^{a,b}\big((\mu^\ep_t,\psi^\ep_t,h^\ep_t)_{t\in[0,1]}\big)
&=a^2\int_0^1 |h_t^\ep|^2\,\di t
+b^2\int_0^1 \langle \nabla\psi_t^\ep,\nabla\psi_t^\ep\rangle_{\mu_t^\ep}\,\di t\\
&\le a^2\int_0^1 |h_t^\ep|^2\,\di t
+b^2\int_0^1 A'(\mu_t^\ep,V_t^\ep)\,\di t\\
&\le \frac{1}{1-2\ep}\left(a^2\int_0^1 |h_t|^2\,\di t
+b^2\int_0^1 A'(\mu_t,V_t)\,\di t\right)\\
&=\frac{1}{1-2\ep}\,\mathfrak E_{\rm Quad}^{a,b}\big((\mu_t,V_t,h_t)_{t\in[0,1]}\big).
\end{aligned}
\end{equation*}
By the arbitrariness of $\ep$ and $(\mu_t,V_t,h_t)_{t\in[0,1]}$,
\begin{equation*}
\big({\al W}_p^{a,b}(\mu_0,\mu_1)\big)^2
\le
\inf\left\{
\mathfrak E^{a,b}_{\rm Quad}\big((\mu_t,V_t,h_t)_{t\in[0,1]}\big)
\ \middle|\ 
(\mu_t,V_t,h_t)_{t\in[0,1]}\in\mathcal{CE}_p(\mu_0,\mu_1;[0,1])
\right\}.
\end{equation*}

\subsection{Proof of Theorem \ref{thm_minimizer}.}
\label{sec_app_minimizer}

The proof follows the approach in \cite[Theorem 3.2]{erbar2012ricci}.

\noindent$\bullet$ \textbf{Step 1. Minimizing sequence and uniform bounds.}
For $n=1,2,...$,
    fix a minimizing sequence $(\mu^n_t,V^n_t,h^n_t)_{t\in[0,1]}$ such that
    \begin{equation*}
    \lim_{n\to\infty} \mathfrak E^{a,b}_{\rm Quad}((\mu_t^n,V_t^n,h_t^n)_{t\in[0,1]})
    =
    {\al W}_p^{a,b}(\mu_0,\mu_1)^2.
    \end{equation*}
Without loss of generality, assume $V_t^n(x,y)=0$ whenever $K(x,y)=0$.
For some $C>0$ and all $n$,
    \begin{equation}
    \label{eq_L2_boundedness_minimizing_action_parts}
    \int_0^1 | h^n_t |^2 \,\mathrm dt\le C,
    \qquad
    \int_0^1 \frac 1 2 \sum_{x,y\in\mathcal X}\alpha\!\big(V^n_t(x,y),\mu^n_t(x),\mu^n_t(y)\big) K(x,y)\varpi(x)\,\mathrm dt\le C .
    \end{equation}
From the continuity equation, there is a uniform upper bound of mass:
\begin{equation*}
\max_{t\in[0,1]}[\mu^n_t,\varpi]
\le [\mu_0,\varpi]+\int_0^1 |h_t^n|\,\mathrm dt
\le [\mu_0,\varpi]+\sqrt{C} .
\end{equation*}
In particular, for all $t,n$ and $x$,
\begin{equation*}
0\le \mu^n_t(x)\le \frac{[\mu_0,\varpi]+\sqrt{C}}{\min\!\varpi} .
\end{equation*}
By the monotonicity of the logarithmic mean,
\begin{equation*}
0\le \hat\mu^n_t(x,y)\le C'
\quad\text{for all }x,y,t,n,
\qquad
C':=\max\Bigl\{\theta(s,r): s,r\in \Bigl[0,\frac{[\mu_0,\varpi]+\sqrt{C}}{\min\!\varpi}\Bigr]\Bigr\}.
\end{equation*}

\noindent$\bullet$ \textbf{Step 2. Passing to weak$^*$ limits of time–measures.}
Define signed Borel measures on $[0,1]$ by
\begin{equation*}
\mathrm d\eta^n(t):=h^n_t\,\mathrm dt,
\qquad
\mathrm d\rho^n_{x,y}(t):=V^n_t(x,y)\,\mathrm dt
\quad\text{for }x,y\in\al X .
\end{equation*}
Let $\lambda(\cdot)$ denote Lebesgue measure on $[0,1]$.
For every Borel set $S\subset[0,1]$, the total variation of $\eta^n$ on $S$ satisfies
    \begin{equation}\label{ineq_eta_leb}
    \|\eta^n\|(S)=\int_S |h^n_t|\,\mathrm dt
    \le \sqrt{\lambda(S)}\left(\int_0^1 |h^n_t|^2\,\mathrm dt\right)^{1/2}
    \le \sqrt{C\,\lambda(S)}.
    \end{equation}
By the upper bound \eqref{eq_L2_boundedness_minimizing_action_parts}, 
    we know that $\alpha(V^n_t(x,y),\mu^n_t(x),\mu^n_t(y))$ is finite for almost every $(x,y)$ with $K(x,y)>0$.
Hence, by the definition of $\alpha$ in \eqref{def_alpha}, we obtain for almost every $t$,
    \begin{equation*}
    |V^n_t(x,y)|
    = \sqrt{\alpha(V^n_t(x,y),\mu^n_t(x),\mu^n_t(y))}\,\sqrt{\hat\mu^n_t(x,y)}
    \le \sqrt{C'}\,\sqrt{\alpha(V^n_t(x,y),\mu^n_t(x),\mu^n_t(y))},
    \end{equation*}
    and this inequality also holds when $K(x,y)=0$, due to our assumption that $V^n_t(x,y)$ vanishes in this case.
By Cauchy--Schwarz inequality in $t$,
    \begin{equation}
    \label{ineq_rho_leb}
    \begin{aligned}
    \sum_{x,y\in\al X} K(x,y)\varpi(x)\,\|\rho^n_{x,y}\|(S)
    &=\sum_{x,y} K(x,y)\varpi(x)\int_S |V^n_t(x,y)|\,\mathrm dt \\
    &\le \sqrt{C'}\int_S \Bigl(\sum_{x,y} K(x,y)\varpi(x)\,\alpha\big(V^n_t(x,y),\mu^n_t(x),\mu^n_t(y)\big)\Bigr)^{\!1/2}\,\mathrm dt\\
    &\le \sqrt{2C'\,\lambda(S)}\left(\int_0^1 \frac 1 2 \sum_{x,y} \alpha\big(V^n_t(x,y),\mu^n_t(x),\mu^n_t(y)\big)K(x,y)\varpi(x)\,\mathrm dt\right)^{\!1/2}\\
    &\le \sqrt{2CC'\,\lambda(S)} .
    \end{aligned}
    \end{equation}
In particular, taking $S=[0,1]$ shows that 
    each total variation $\|\eta^n\|([0,1])$ and $\|\rho^n_{x,y}\|([0,1])$ is uniformly bounded in $n$.
By Banach--Alaoglu on the dual of $\mathrm C([0,1];\mathbb{R})$, 
    there exists a subsequence, not relabeled, and finite signed measures $\eta$ and $\rho_{x,y}$ such that
    \begin{equation*}
    \eta^n \stackrel{*}{\rightharpoonup}\eta,
    \qquad
    \rho^n_{x,y}\stackrel{*}{\rightharpoonup}\rho_{x,y}
    \quad\text{in the weak$^*$ sense of measures on }[0,1] .
    \end{equation*}
\eqref{ineq_eta_leb} and \eqref{ineq_rho_leb} also shows that $(\eta^n)_n$ and $(\rho^n_{x,y})_n$ are uniformly absolutely continuous with respect to $\lambda$.
Hence $\eta$ and each $\rho_{x,y}$ are absolutely continuous with respect to $\lambda$.
There exist $h:[0,1]\to \mathbb{R}$ and $V:[0,1]\to\mathbb R^{\al X\times\al X}$ with
\begin{equation*}
\mathrm d\eta(t)=h_t\,\mathrm dt,
\qquad
\mathrm d\rho_{x,y}(t)=V_t(x,y)\,\mathrm dt .
\end{equation*}

\noindent$\bullet$ \textbf{Step 3. Defining the limit $\mu$ via the continuity identity.}
For each $x\in\al X$ and $s\in[0,1]$ the continuity equation in \eqref{new_ce} gives
\begin{equation}\label{eq_new_ce_integral}
\begin{aligned}
\mu_s^n(x)-\mu_0(x)
= &\ \frac{1}{2}\int_0^s \sum_{y\in \al X}\!\big(V_t^n(y,x)-V_t^n(x,y)\big)K(x,y)\,\mathrm dt \;+\; p(x)\int_0^s h_t^n\,\mathrm dt
\\
= &\ \frac{1}{2}\sum_{y\in \al X} K(x,y)\left( \int_0^s \di \rho^n_{x,y}
-\int_0^s \di \rho^n_{y,x}
\right) + p(x)\int_0^s  \di \eta^n.
\end{aligned}
\end{equation}
By weak$^*$ convergence of $\rho^n_{x,y}$ and $\eta^n$
and a standard application of the Portmanteau Theorem \cite[Theorem 2.1]{billingsley2013convergence}, the right-hand side of \eqref{eq_new_ce_integral} converges for every fixed $s$.
Define $\mu_s(x)$ by the limit
\begin{equation*}
\mu_s(x):=\mu_0(x)+\frac{1}{2}\int_0^s \sum_{y\in \al X}\!\big(V_t(y,x)-V_t(x,y)\big)K(x,y)\,\mathrm dt \;+\; p(x)\int_0^s h_t\,\mathrm dt .
\end{equation*}
Then $\mu_s^n(x)\to \mu_s(x)$ for all $s$ and $x$.
Since $\mu^n_t(x)$ are uniformly bounded in $t,n$, dominated convergence yields $\mu^n(\cdot,x)\to \mu(\cdot,x)$ in $\mathrm L^1([0,1];\mathbb{R})$ for each $x$, hence the measures $\mu^n_t(x)\,\mathrm dt$ converge weak$^*$ to $\mu_t(x)\,\mathrm dt$.
By construction $(\mu_t,V_t,h_t)_{t\in[0,1]}$ satisfies the continuity equation in \eqref{new_ce} and belongs to $\mathcal{CE}_p(\mu_0,\mu_1;[0,1])$ with endpoints $\mu_0$ and $\mu_1$.
\vspace{0.2cm}

\noindent$\bullet$ \textbf{Step 4. Lower semicontinuity of the action and optimality.}
By Lemma \ref{lem_alpha_convex}, $\alpha(v,s,r)$ is convex and lower semicontinuous.
Viewing $A'(\mu_t,V_t)$ as a finite sum of $\alpha$, 
    we invoke the general lower semicontinuity theorem for convex integral functionals on measures \cite[Theorem 3.4.3]{buttazzo1989semicontinuity} to conclude
\begin{equation*}
\int_0^1 A'(\mu_t,V_t)\,\mathrm dt
\le \liminf_{n\to\infty}\int_0^1 A'(\mu^n_t,V^n_t)\,\mathrm dt .
\end{equation*}
For the source part, the same convex-integral lower semicontinuity yields
\begin{equation*}
\int_0^1 |h_t|^2\,\mathrm dt\le \liminf_{n\to\infty}\int_0^1 |h_t^n|^2\,\mathrm dt .
\end{equation*}
Combining these estimates with \eqref{eq_L2_boundedness_minimizing_action_parts} gives
\begin{equation*}
a^2\int_0^1 | h_t |^2 \,\mathrm dt
+
b^2 \int_0^1 A'(\mu_t,V_t )\,\mathrm dt
\le
\liminf_{n\to\infty}\left\{
a^2\int_0^1 | h_t^n |^2 \,\mathrm dt
+
b^2 \int_0^1 A'(\mu_t^n,V^n_t )\,\mathrm dt
\right\}.
\end{equation*}
Since the right-hand side equals $\big(\al W_p^{a,b}(\mu_0,\mu_1)\big)^2$ by minimality of the sequence, the left-hand side attains the infimum.
Hence $(\mu_t,V_t,h_t)_{t\in[0,1]}$ is a minimizer for \eqref{eq_Wpab_with_V}.

\subsection{Proof of Lemma \ref{lem_asymmetric}}
\label{sec_app_f}

First we have for a measurable $\Psi\in\mathbb{R}^{\al X\times \al X}$,
    \begin{equation*}
    \nabla\cdot \check{\Psi}(x)= \frac{1}{2}\sum_{y\in \al X} (\Psi(x,y)-\Psi(y,x))K(x,y)=\nabla\cdot \Psi(x).
    \end{equation*}
For any $\nu\in \al M$ and $x,y\in \al X$,
    \begin{enumerate}
    \item if $\theta(\nu(x),\nu(y))=0$ and $\Psi(x,y)\neq0$ or $\Psi(y,x)\neq0$, then
        \begin{equation*}
        \begin{aligned}
        & \alpha(\check\Psi(x,y),\nu(x),\nu(y)) K(x,y)\varpi(x)+\alpha(\check\Psi(y,x),\nu(y),\nu(x)) K(y,x)\varpi(y)
        \\
        &\qquad\le\alpha(\Psi(x,y),\nu(x),\nu(y)) K(x,y)\varpi(x)+\alpha(\Psi(y,x),\nu(y),\nu(x)) K(y,x)\varpi(y)= \infty;
        \end{aligned}
        \end{equation*}
    \item if $\theta(\nu(x),\nu(y))=0$ and $\Psi(x,y)=\Psi(y,x)=0$, then
        \begin{equation*}
        \begin{aligned}
        & \alpha(\check\Psi(x,y),\nu(x),\nu(y)) K(x,y)\varpi(x)+\alpha(\check\Psi(y,x),\nu(y),\nu(x)) K(y,x)\varpi(y)
        \\
        &\qquad=\alpha(\Psi(x,y),\nu(x),\nu(y)) K(x,y)\varpi(x)+\alpha(\Psi(y,x),\nu(y),\nu(x)) K(y,x)\varpi(y)= 0;
        \end{aligned}
        \end{equation*}
    \item if $\theta(\nu(x),\nu(y))\neq0$, then
        \begin{equation*}
        \begin{aligned}
        & \alpha(\check\Psi(x,y),\nu(x),\nu(y)) K(x,y)\varpi(x)+\alpha(\check\Psi(y,x),\nu(y),\nu(x)) K(y,x)\varpi(y)
        \\
        &\qquad= 
        \frac{\check\Psi(x,y)^2}{\hat{\nu}(x,y)}K(x,y)\varpi(x)+\frac{\check\Psi(y,x)^2}{\hat{\nu}(y,x)}K(y,x)\varpi(y)
        \\
        &\qquad=
        \frac{\check\Psi(x,y)^2+\check\Psi(y,x)^2}{\hat{\nu}(x,y)}K(x,y)\varpi(x)
        \\
        &\qquad\le
        \frac{\Psi(x,y)^2+\Psi(y,x)^2}{\hat{\nu}(x,y)}K(x,y)\varpi(x)
        \\
        &\qquad=\alpha(\Psi(x,y),\nu(x),\nu(y)) K(x,y)\varpi(x)+\alpha(\Psi(y,x),\nu(y),\nu(x)) K(y,x)\varpi(y).
        \end{aligned}
        \end{equation*}
    \end{enumerate}
Thus
    \begin{equation*}
    \begin{aligned}
    A'(\nu,\check\Psi) = \frac 1 2 \sum_{x,y\in \al X}\alpha(\check\Psi(x,y),\nu(x),\nu(y)) K(x,y)\varpi(x)
    \le A'(\nu,\Psi).
    \end{aligned}
    \end{equation*}

\subsection{Proof of Theorem \ref{thm:GE-weak}}
\label{sec_app_GE_weak}

We split our proof into two parts.
\vspace{0.2cm}

\noindent\textbf{$\ast$ Part A. Necessity: geodesic (i) $\Rightarrow$ weak geodesic system (ii).}
Let $\mu_t:[0,1]\to \al M_+$ be a constant–speed $\al W^{a,b}_p$–geodesic and, for a.e.\ $t$, let
    \begin{equation*}
    (\nabla\psi_t,h_t):=D_{\mu_t}(\dot\mu_t),
    \end{equation*}
    where the operator $D_{\mu_t}$ is given in Lemma \ref{lem_ce_solver_combined}.
Define the corresponding flux as
    \begin{equation*}
    V_t(x,y)=\theta\big(\mu_t(x),\mu_t(y)\big)\,\nabla\psi_t(x,y),
    \qquad\text{for a.e.\ }t.
    \end{equation*}
Then the action reads
    \begin{equation*}
    \begin{aligned}
    E^{a,b}_{\rm Quad}((\mu_t,\psi_t,h_t)_{t\in[0,1]})
    =&
    a^2\!\int_0^1 h_t^2\di t
    +
    \frac{b^2}{2}\!
    \int_0^1\sum_{x,y\in \al X}\!
    \theta\big(\mu_t(x),\mu_t(y)\big)\,
    \nabla\psi_t(x,y)^2\,K(x,y)\,\varpi(x)\di t
    \\
    =&
    a^2\!\int_0^1 h_t^2\di t
    +
    \frac{b^2}{2}\!
    \int_0^1\sum_{x,y\in \al X}\!
    \frac{V_t(x,y)^2}{\theta\big(\mu_t(x),\mu_t(y)\big)}\,
    K(x,y)\,\varpi(x)\di t
    \\
    = & \
    \mathfrak{E}^{a,b}_{\rm Quad}((\mu_t,V_t,h_t)_{t\in[0,1]}).
    \end{aligned}
    \end{equation*}
{{Next, we verify weak geodesic systems \eqref{eq:GE-h-weak} and \eqref{eq:GE-psi-weak} step by step. }}

\medskip
\noindent$\bullet$ {\textbf{Step A.1. Source variation $\Rightarrow$ \eqref{eq:GE-h-weak}.}}
Fix $\zeta\in \mathrm C_{\mathrm c}^1((0,1);\mathbb{R})$ and set
    \begin{equation*}
    f_t:=-\dot\zeta(t),\qquad F_t:=\int_0^t f_s\di s=-\zeta(t),
    \end{equation*}
    so that $\zeta(0)=\zeta(1)=0$.
For $\varepsilon\in\mathbb{R}$ small, define the source–only perturbation
    \begin{equation*}
    \mu_t^\varepsilon:=\mu_t+\varepsilon F_t\,p,\qquad h_t^\varepsilon:=h_t+\varepsilon f_t.
    \end{equation*}
Then $(\mu^\varepsilon_t,V_t,h^\varepsilon_t)_{t\in[0,1]}$ is 
    an admissible triplet in $\mathcal{CE}_p(\mu_0,\mu_1;[0,1])$ 
    for $\ep$ small enough to keep $\mu^\ep$ in $\al M_+$.
By optimality of $(\mu_t,V_t,h_t)_{t\in[0,1]}$,
    \begin{equation*}
    \begin{aligned}
    0=&
    \left. \frac{\di}{\di\varepsilon} \right|_{\varepsilon=0}
    \mathfrak{E}^{a,b}_{\rm Quad} ((\mu^\varepsilon_t,V_t,h^\varepsilon_t)_{t\in[0,1]})
    \\
    =&
    \left.\frac{\di}{\di\ep}\right|_{\ep=0}
    \left(
    a^2 \int_0^1  (h_t+\ep f_t)^2 \di t
    +
    \frac{b^2}{2} \int_0^1 \sum_{x,y}
    \frac{V_t(x,y)^2 K(x,y)\varpi(x)}
    {\theta\big(\mu_t(x)+\ep F_t p(x),\mu_t(y)+\ep F_t p(y)\big)}
    \di t
    \right)
    \\
    =&
    2a^2 \int_0^1 h_t f_t \di t
    \\
    &-\frac{b^2}{2} \int_0^1 \sum_{x,y}
    \frac{V_t(x,y)^2 K(x,y)\varpi(x)}
    {\theta\big(\mu_t(x),\mu_t(y)\big)^2}
    \big( \partial_1 \theta(\mu_t(x),\mu_t(y)) p(x) 
    + \partial_2 \theta(\mu_t(x),\mu_t(y)) p(y) \big)
     F_t \di t.
    \end{aligned}
    \end{equation*}
Using reversibility $K(x,y)\varpi(x)=K(y,x)\varpi(y)$, 
    the symmetry $V_t(x,y)^2=V_t(y,x)^2$, 
    and $\partial_1\theta(s,t)=\partial_2\theta(t,s)$, 
    pairing $(x,y)$ with $(y,x)$ in the weighted edge sum $\sum_{x,y}(\cdot)\,K(x,y)\varpi(x)$ shows that 
    the combined contribution $\partial_1\theta(\mu_t(x),\mu_t(y))\,p(x)+\partial_2\theta(\mu_t(x),\mu_t(y))\,p(y)$ equals $2\,\partial_1\theta(\mu_t(x),\mu_t(y))\,p(x)$.
Thus
    \begin{equation*}
    \begin{aligned}
    0
    =&
    2a^2 \int_0^1 h_t f_t \di t
    -b^2 \int_0^1 \sum_{x,y}
    \frac{V_t(x,y)^2 K(x,y)\varpi(x)}
    {\theta\big(\mu_t(x),\mu_t(y)\big)^2}
    \partial_1 \theta(\mu_t(x),\mu_t(y)) p(x) 
    F_t \di t
    \\
    =&
    2a^2 \int_0^1 h_t f_t \di t
    -b^2 \int_0^1 \sum_{x,y}
    \nabla\psi_t(x,y)^2
    \partial_1\theta\big(\mu_t(x),\mu_t(y)\big)
    K(x,y) \varpi(x) p(x)
    F_t \di t.    
    \end{aligned}
    \end{equation*}
Using $f=-\dot\zeta$ and $F=-\zeta$, we obtain
    \begin{equation*}
    \int_0^1 h_t\,\dot\zeta(t)\di t
    =
    \frac{b^2}{2a^2}\int_0^1\!
    \Bigg(
        \sum_{x,y}
        \nabla\psi_t(x,y)^2\,
        \partial_1\theta\big(\mu_t(x),\mu_t(y)\big)
        \,K(x,y)\,\varpi(x)\,p(x)
    \Bigg)\,\zeta(t)\di t,
    \end{equation*}
    which is exactly \eqref{eq:GE-h-weak}.

\medskip
\noindent$\bullet$ {\textbf{Step A.2. Transport variation $\Rightarrow$ \eqref{eq:GE-psi-weak}.}}
Fix an ordered pair $(x,y)$ with $K(x,y)>0$ and $\zeta\in \mathrm C_{\mathrm c}^1((0,1);\mathbb{R})$. 
Define the mean–zero node vector for $u\in \al X$
    \begin{equation*}
    \xi^{x,y}(u):=\frac{\mathbbm{1}_{\{u=y\}}}{\varpi(y)}-\frac{\mathbbm{1}_{\{u=x\}}}{\varpi(x)}.
    \end{equation*}
Here $\mathbbm{1}_S$ stands for the characteristic function of set $S$. Then
    \begin{equation*}
    \langle \xi^{x,y},\mathbf{1} \rangle_\varpi=0, \ \ 
    \langle \phi,\xi^{x,y} \rangle_\varpi=\phi(y)-\phi(x)=\nabla \phi(x,y).
    \end{equation*}
There exists a time-independent edge field 
    $U^{x,y}\in\mathbb{R}^{\al X\times \al X}$ 
    with $\nabla\!\cdot U^{x,y}=\xi^{x,y}$.
For example, one can fix any $\nu\in \al M_+$ and, by Lemma \ref{lem_ce_solver_combined}, 
    pick $\eta^{x,y}\in\mathbb{R}^{\al X}$ such that 
    $\nabla\!\cdot(\hat\nu\ast\nabla\!\eta^{x,y})
    =\nabla_{\!\nu}\!\cdot\nabla\!\eta^{x,y}=\xi^{x,y}$, 
    then define $U^{x,y}:=\hat\nu\ast\nabla\!\eta^{x,y}$.
Set
    \begin{equation*}
    \Omega_t:=\zeta(t)\,\xi^{x,y},\quad W_t:=\dot\zeta(t)\,U^{x,y},\quad\text{and we have }\quad  \nabla\!\cdot W_t=\dot\zeta(t)\,\xi^{x,y}=\dot\Omega_t,\quad \Omega_0=\Omega_1=0.
    \end{equation*}
For $\delta\in\mathbb{R}$ small, define the coupled transport perturbation
    \begin{equation*}
    \mu_t^\delta:=\mu_t-\delta\,\Omega_t,\qquad 
    V_t^\delta:=V_t+\delta\,W_t,
    \end{equation*}
    so that $\dot\mu_t^\delta=-\nabla\!\cdot V_t^\delta+h_t p$,
    $\mu_t^\delta\in \al M_+$
    and the endpoints are fixed; 
    hence $(\mu^\delta_t,V^\delta_t,h_t)_{t\in[0,1]}$ is admissible 
    in $\mathcal{CE}_p(\mu_0,\mu_1;[0,1])$.
Optimality of $(\mu_t,V_t,h_t)_{t\in[0,1]}$ yields
    \begin{align*}
    0=&
    \left.\frac{\di}{\di \delta}\right|_{\delta=0}
    \mathfrak{E}^{a,b}_{\rm Quad}((\mu^\delta_t,V^\delta_t,h_t)_{t\in[0,1]})
    \\
    =&
    \left.\frac{\di}{\di\delta}\right|_{\delta=0}
    \left(
    a^2 \int_0^1  h_t^2 \di t
    +
    \frac{b^2}{2} \int_0^1 \sum_{u,v}
    \frac{(V_t(u,v)+\delta W_t(u,v))^2 K(u,v)\varpi(u)}
    {\theta\big(\mu_t(u)-\delta \Omega_t(u),\mu_t(v)-\delta \Omega_t(v)\big)}
    \di t
    \right)
    \\
    =&
    b^2\!\int_0^1\!\sum_{u,v} 
    \frac{V_t(u,v)W_t(u,v)\,K(u,v)\varpi(u)}
    {\theta\big(\mu_t(u),\mu_t(v)\big)}\di t
    \\
    &+
    b^2\!\int_0^1\!
    \sum_{u,v} \frac{V_t(u,v)^2\,K(u,v)\varpi(u)}
    {\theta\big(\mu_t(u),\mu_t(v)\big)^2}
    \partial_1\theta(\mu_t(u),\mu_t(v))\,\Omega_t(u)
    \di t
    \\
    =&
    b^2\!\int_0^1\!\sum_{u,v} 
    \nabla\psi_t(u,v)\,W_t(u,v)\,K(u,v)\varpi(u)\di t
    \\
    &+
    b^2\!\int_0^1\!\sum_{u,v}\!
    \partial_1\theta(\mu_t(u),\mu_t(v))\,\Omega_t(u)\,
    \nabla\psi_t(u,v)^2 K(u,v)\varpi(u)\di t.
    \end{align*}
Apply the discrete integration by parts formula \eqref{eq_int_by_parts}
    and $\langle \phi,\xi^{x,y} \rangle_\varpi=\nabla \phi(x,y)$:
    \begin{equation*}
    \begin{aligned}
    \int_0^1\!\sum_{u,v}\nabla\psi_t(u,v)\,W_t(u,v)\,K(u,v)\varpi(u)\di t
    = &
    2\!\int_0^1 \langle \nabla\psi_t, W_t\rangle_\varpi\di t
    = 
    -2\!\int_0^1 \langle \psi_t, \nabla \!\cdot\! W_t\rangle_\varpi\di t
    \\
    = &
    -2\!\int_0^1 \langle \psi_t, \xi^{x,y}\rangle_\varpi \dot\zeta(t)\di t
    = 
    -2\!\int_0^1  \nabla \psi_t(x,y)\dot\zeta(t) \di t.
    \end{aligned}
    \end{equation*}
Therefore,
    \begin{align*}
    \int_0^1  \nabla \psi_t(x,y)\dot\zeta(t) \di t
    = &
    \,\frac{1}{2}
    \!\int_0^1\!\sum_{u,v}\!
        \partial_1\theta(\mu_t(u),\mu_t(v))\,\Omega_t(u)\,
        \nabla\psi_t(u,v)^2 K(u,v)\varpi(u)
    \di t
    \\
    = &
    \,\frac{1}{2}
    \!\int_0^1
        \zeta(t)
        \!\sum_{u,v}\!
        \partial_1\theta(\mu_t(u),\mu_t(v))
        \nabla\psi_t(u,v)^2 K(u,v)\xi^{x,y}(u)\varpi(u)
    \di t
    \\
    = &
    \,\frac{1}{2}
    \!\int_0^1
        \zeta(t)
        \left \langle
        \!\sum_{v}\!
        \partial_1\theta(\mu_t(\cdot),\mu_t(v))
        \nabla\psi_t(\cdot,v)^2 K(\cdot,v)
        ,
        \xi^{x,y}
        \right \rangle_{\varpi}
    \di t
    \\
    = &
    \frac{1}{2}
    \int_0^1
    \left(
    \sum_{z}
    \nabla\psi_t(z,y)^2
    K(y,z)
    \partial_1 \theta(\mu_t(y),\mu_t(z))
    \right.
    \\
    &\qquad \qquad-
    \left.
    \sum_{z}
    \nabla\psi_t(z,x)^2
    K(x,z)
    \partial_1 \theta(\mu_t(x),\mu_t(z))
    \right)
    \zeta(t)
    \di t,
    \end{align*}
    which is precisely \eqref{eq:GE-psi-weak} for the chosen $(x,y)$. 
Since $(x,y)$ and $\zeta$ are arbitrary, 
    \eqref{eq:GE-psi-weak} holds for all edges with $K(x,y)>0$ and all $\zeta\in \mathrm C_{\mathrm c}^1((0,1);\mathbb{R})$.

\medskip
\noindent
\noindent\textbf{$\ast$ Part B. Sufficiency: weak geodesic system (ii) $\Rightarrow$ geodesic (i).}
Assume $\mu_t:[0,1]\to \al M_+$ is absolutely continuous, 
    and $(\nabla\psi_t,h_t):=D_{\mu_t}(\dot\mu_t)$ satisfies the weak geodesic system \eqref{eq:GE-h-weak}–\eqref{eq:GE-psi-weak}. 
Define for a.e.\ $t$ the flux
    \begin{equation*}
    V_t(x,y):=\theta\big(\mu_t(x),\mu_t(y)\big)\,\nabla\psi_t(x,y).
    \end{equation*}
We claim that $(\mu_t,V_t,h_t)_{t\in[0,1]}$ is a global minimizer of the action $\mathfrak{E}^{a,b}_{\rm Quad}$ among all feasible triples,
    hence by Corollary \ref{cor_geodesic} $\mu$ is a constant–speed $\al W^{a,b}_p$–geodesic.
{{Next, we verify our claim in the following five steps. }}

\medskip
\noindent$\bullet$ {\textbf{Step B.1. Feasible perturbations form a linear space and split as source $\oplus$ transport.}}
Consider the ambient Banach space
    \begin{equation*}
    \mathbb{X}
    := \mathrm W^{1,1}\big([0,1];\mathbb{R}^{\al X}\big) \times \mathrm L^1\big([0,1];\mathbb{R}^{\al X\times \al X}\big) \times \mathrm L^1([0,1];\mathbb{R}),
    \end{equation*}
    with the product norm. 
$\mathrm W^{1,1}([0,1];\mathbb{R}^{\al X})$ is the Sobolev space of functions whose weak derivative belongs to $\mathrm L^1$. 
It coincides with $\mathrm{A\!C}([0,1];\mathbb{R}^{\al X})$, and the weak derivative agrees almost everywhere with the classical derivative.
Define the continuous linear operator
    \begin{equation*}
    \mathcal{L}: \mathbb{X}\to \mathrm L^1\big([0,1];\mathbb{R}^{\al X}\big),\qquad
    \mathcal{L}(\mu_t,V_t,h_t)_{t\in[0,1]}:=\dot\mu+\nabla\!\cdot V - h\,p,
    \end{equation*}
    and the endpoint evaluation maps $E_0(\mu):=\mu(0)$, $E_1(\mu):=\mu(1)$.
Let
    \begin{equation*}
    \mathcal{A}:=\Bigl\{(\mu_t,V_t,h_t)_{t\in[0,1]}\in\mathbb{X}:\ \mathcal{L}(\mu_t,V_t,h_t)_{t\in[0,1]}=0,\ \ E_0(\mu)=\mu(0),\ E_1(\mu)=\mu(1)\Bigr\}
    \end{equation*}
    be the affine subspace of the equality constraints, and let
    \begin{equation*}
    \mathcal{K}:=\Bigl\{(\mu_t,V_t,h_t)_{t\in[0,1]}\in\mathbb{X}:\ \mu_t(x)\ge 0\ \text{for a.e.\ }t\in(0,1),\ \forall~ x\in \al X\Bigr\}
    \end{equation*}
    be the closed convex cone encoding nonnegativity.
Then the feasible set can be characterized as
    \begin{equation*}
    \mathcal{CE}_p\big(\mu(0),\mu(1);[0,1]\big)\ =\ \mathcal{A}\ \cap\ \mathcal{K}.
    \end{equation*}
Since $\mathcal{A}$ and $\mathcal{K}$ are closed and convex in $\mathbb{X}$, 
    their intersection $\mathcal{CE}_p\big(\mu(0),\mu(1);[0,1]\big)$ is also a closed convex subset of $\mathbb{X}$.

Denote the \emph{relative interior} with respect to $\mathcal{A}$ of the feasible set as
    \begin{equation*}
    \operatorname{ri}_{\mathcal{A}}\big(\mathcal{A}\cap\mathcal{K}\big)
    :=
    \Bigl\{ 
        z\in \mathcal{A}\cap\mathcal{K} 
        \mid 
        \exists~ r>0,B_{\mathcal{A}}(z,r)\subset \mathcal{A}\cap\mathcal{K}
    \Bigr\},
    \end{equation*}
    where $B_{\mathcal{A}}(z,r)$ denotes the open ball of radius $r$ in the induced norm on the affine subspace $\mathcal{A}$.
And denote the tangent space of the affine constraints as
    \begin{equation*}
    \mathcal{T}:=
    \left\{
        (\delta\mu_t,\delta V_t,\delta h_t)_{t\in[0,1]}
        \left | \
        \begin{aligned}
            &\delta\mu_t\in \mathrm{W}^{1,1}_0([0,1];\mathbb{R}^{\al X}), \
            \delta V_t\in \mathrm{L^1([0,1];\mathbb{R}^{\al X\times \al X})}, \
            \delta h_t \in \mathrm{L^1([0,1];\mathbb{R})}
           \vspace{8pt} \\
            &\dot{\delta\mu}_t=-\nabla\!\cdot\delta V_t+\delta h_t p\ \text{ a.e.}, \
            \delta\mu(0)=\delta\mu(1)=\boldsymbol{0}\in \mathbb{R}^{\al X}
        \end{aligned}
        \right.
    \right\}.
    \end{equation*}
Here $\mathrm W^{1,1}_0([0,1];\mathbb{R}^{\al X})$ is the subspace of $\mathrm W^{1,1}([0,1];\mathbb{R}^{\al X})$ consisting of functions taking zero at $t=0$ and $t=1$.

Since $\mu_t\in \al M_+$ for all $t\in[0,1]$, 
    there exists $\alpha>0$ with $\inf_{t,x}\mu_t(x)\ge\alpha$, hence $(\mu_t,V_t,h_t)_{t\in[0,1]}\in\operatorname{ri}_{\mathcal A}(\mathcal{A}\cap\mathcal{K} )$.
Whence the feasible directions are exactly the linearized constraint directions $\mathcal T$.
Equivalently, $(\delta\mu_t,\delta V_t,\delta h_t)_{t\in[0,1]}\in\mathcal T$ iff
    \begin{equation*}
    (\mu_t,V_t,h_t)_{t\in[0,1]}+s(\delta\mu_t,\delta V_t,\delta h_t)_{t\in[0,1]}\in\mathcal{A}\cap\mathcal{K},
    \quad
    \text{for all sufficiently small } s>0,
    \end{equation*}
    and for every $(\nu_t,W_t,g_t)_{t\in[0,1]}\in\mathcal{A}\cap\mathcal{K} $ one has $(\nu_t,W_t,g_t)_{t\in[0,1]}-(\mu_t,V_t,h_t)_{t\in[0,1]}\in\mathcal T$.

All $(\delta\mu_t,\delta V_t,\delta h_t)_{t\in[0,1]}\in \mathcal{T}$ admits the following decomposition,
    \begin{equation*}
    (\delta\mu_t,\delta V_t,\delta h_t)_{t\in[0,1]}=
    \left(\int_0^t\delta h_t \di t \ p,0,\delta h\right)
    +
    \left(\delta\mu- \int_0^t\delta h_t \di t \ p,\delta V,0\right),
    \end{equation*}
    whence $\mathcal{T}$ is generated by two basic families:
    \begin{itemize}
        \item[(S)] (\emph{source directions}) 
            given $f\in \mathrm L^1([0,1];\mathbb{R})$ with $\int_0^1 f_t \di t=0$, set $F(t):=\int_0^t f_s \di s$ and
                \begin{equation*}
                \delta\mu_t=F(t)\,p,\qquad \delta h_t=f_t,\qquad \delta V_t=0.
                \end{equation*}
        \item[(T)] (\emph{transport directions}) 
            given $\Omega_t\in \mathrm{W}^{1,1}_0([0,1];\mathbb{R}^{\al X})$ 
                with $\langle\Omega_t,\mathbf{1}\rangle_\varpi=0$ and $\Omega_0=\Omega_1=0$, 
                pick $W_t$ such that
                $\nabla\!\cdot W_t=\dot\Omega_t$ and set
                    \begin{equation*}
                    \delta\mu_t=-\,\Omega_t,\qquad \delta V_t=W_t,\qquad \delta h_t=0.
                    \end{equation*}
    \end{itemize}

\medskip
\noindent$\bullet$ {\textbf{Step B.2. First variation vanishes along (S) by \eqref{eq:GE-h-weak}.}}
Consider a source variation $(\mu^\varepsilon_t,V^\varepsilon_t,h^\varepsilon_t)_{t\in[0,1]}$ with
    \begin{equation*}
    \mu^\varepsilon_t=\mu_t+\varepsilon F(t) p,\qquad V^\varepsilon_t=V_t,\qquad h^\varepsilon_t=h_t+\varepsilon f_t,
    \end{equation*}
    where $f=-\dot\zeta$ and $F=-\zeta$ for some $\zeta\in \mathrm C_{\mathrm c}^1((0,1);\mathbb{R})$.
Differentiating at $\varepsilon=0$ and using $V_t=\hat\mu_t\ast\nabla\psi_t$ gives
    \begin{align*}
    &\left.\frac{\di}{\di\varepsilon}\right|_{\ep=0}\mathfrak{E}^{a,b}_{\rm Quad}((\mu^\varepsilon_t,V^\varepsilon_t,h^\varepsilon_t)_{t\in[0,1]})
    \\
    &\qquad\qquad=2a^2\!\int_0^1 h_t f_t\di t
    +b^2\!\int_0^1\!\sum_{x,y}\partial_1\theta\big(\mu_t(x),\mu_t(y)\big)\,F_t p(x)\,\nabla\psi_t(x,y)^2\,K(x,y)\,\varpi(x)\di t.
    \end{align*}
With $f=-\dot\zeta$ and $F=\zeta$, 
    the weak equation \eqref{eq:GE-h-weak} gives that the right-hand side is $0$.
Since $\mathrm C_{\mathrm c}^1((0,1);\mathbb{R})$ is dense in $\mathrm W^{1,1}_0([0,1];\mathbb{R})$, 
    the first variation vanishes along every (S)-direction.

\medskip
\noindent$\bullet$ {\textbf{Step B.3. First variation vanishes along (T) by \eqref{eq:GE-psi-weak}.}}
Consider a transport variation $(\mu^\delta_t,V^\delta_t,h^\delta_t)_{t\in[0,1]}$ with
    \begin{equation*}
    \mu^\delta_t=\mu_t-\delta\,\Omega_t,\qquad V^\delta_t=V_t+\delta W_t,\qquad h_t^\delta=h_t,
    \end{equation*}
    where $\Omega_t$ and $W_t$ are as in (T). 
Differentiating at $\delta=0$ yields
    \begin{equation}
    \label{eq_thm_geodesic_equation_1}
    \begin{aligned}
    \frac{1}{b^2}
    \left.\frac{\di}{\di\delta}\right|_{\delta=0}
    \mathfrak{E}^{a,b}_{\rm Quad}((\mu^\delta_t,V^\delta_t,h^\delta_t)_{t\in[0,1]})
    =&
    \!\int_0^1\!\sum_{x,y}
    \nabla\psi_t(x,y)\,W_t(x,y)\,K(x,y)\,\varpi(x)
    \di t
    \\
    &+
    \!\int_0^1\!\sum_{x,y}
    \nabla\psi_t(x,y)^2\,
    \partial_1\theta\big(\mu_t(x),\mu_t(y)\big)\,
    K(x,y)\,
    \Omega_t(x)\,\varpi(x)\di t.
    \end{aligned}
    \end{equation}
By the discrete integration by parts in space 
    and $\nabla\!\cdot W_t=\dot\Omega_t$,
    \begin{align*}
    \int_0^1\!\sum_{x,y}\nabla\psi_t(x,y)\,W_t(x,y)\,K(x,y)\,\varpi(x)\di t
    &=2\!\int_0^1 \langle \nabla \psi_t,W_t\rangle_\varpi\di t\\
    &=-2\!\int_0^1 \langle \psi_t,\nabla\!\cdot W_t\rangle_\varpi\di t
    = - 2\!\int_0^1 \langle \psi_t,\dot\Omega_t\rangle_\varpi\di t.
    \end{align*}
Since $\langle\dot\Omega_t,\mathbf{1}\rangle_\varpi=\langle \Omega_t,\mathbf{1}\rangle_\varpi=0$, take $u\in \al X$
    \begin{equation*}
    - 2\!\int_0^1 \langle \psi_t,\dot\Omega_t\rangle_\varpi\di t
    =-2
    \sum_x
    \int_0^1 (\psi(x)-\psi(u))\dot\Omega_t(x) \varpi (x)\di t
    =2
    \sum_x
    \int_0^1 \nabla\psi(x,u)\dot\Omega_t(x) \varpi (x)\di t.
    \end{equation*}
Apply \eqref{eq:GE-psi-weak} with a sequence $\zeta_n\in \mathrm C_{\mathrm c}^1((0,1);\mathbb{R})$ such that $\zeta_n$ converges to $\big(t\mapsto \Omega_t(x)\varpi(x)\big)$ in $\mathrm{W}^{1,1}_0([0,1];\mathbb{R})$. 
By density, passing to the limit shows that the following identity holds for every admissible $\Omega$:
    \begin{align*}
    2\sum_x
    \int_0^1 \nabla\psi(x,u)\dot\Omega_t(x) \varpi (x)\di t
    =&
    \sum_x\int_0^1
    \left(
     \sum_{z}
    \nabla\psi(z,u)^2
    K(u,z)
    \partial_1 \theta(\mu_t(u),\mu_t(z))
    \right.
    \\
    & \left. -
    \sum_{z}
    \nabla\psi(z,x)^2
    K(x,z)
    \partial_1 \theta(\mu_t(x),\mu_t(z))\right)
    \Omega_t(x) \varpi (x)
    \di t
    \\
    =&
    \int_0^1
    \sum_{z}
    \nabla\psi(z,u)^2
    K(u,z)
    \partial_1 \theta(\mu_t(u),\mu_t(z))
    \sum_x \Omega_t(x) \varpi (x) \di t
    \\
    &-
    \int_0^1
    \sum_{x,z}
    \nabla\psi(z,x)^2
    K(x,z)
    \partial_1 \theta(\mu_t(x),\mu_t(z))
    \Omega_t(x) \varpi (x)
    \di t
    \\
    =&-
    \int_0^1
    \sum_{x,z}
    \nabla\psi(z,x)^2
    K(x,z)
    \partial_1 \theta(\mu_t(x),\mu_t(z))
    \Omega_t(x) \varpi (x)
    \di t.
    \end{align*}
Hence the two terms in the first variation in \eqref{eq_thm_geodesic_equation_1} cancel and
$\left.\frac{\di}{\di\delta}\right|_{\delta=0}\mathfrak{E}^{a,b}_{\rm Quad}((\mu^\delta_t,V^\delta_t,h^\delta_t)_{t\in[0,1]})=0$.
Since the first variation vanishes along all directions in $\mathcal{T}$, 
    $(\mu_t,V_t,h_t)_{t\in[0,1]}$ is a stationary point of $\mathfrak{E}^{a,b}_{\rm Quad}$ on $\mathcal{CE}_p(\mu_0,\mu_1;[0,1])$.

\medskip
\noindent$\bullet$ {\textbf{Step B.4. Convexity $\Rightarrow$ stationarity implies minimality.}}
The admissible set $\mathcal{CE}_p(\mu_0,\mu_1;[0,1])$ is convex and contained in an affine set $\al A$.
The integrand $a^2 h^2$ is strictly convex in $h$, 
    and by Lemma \ref{lem_alpha_convex}, the transport integrand $\alpha$
    is jointly convex on $\mathbb{R}\times(0,\infty)^2$.
Therefore $\mathfrak{E}^{a,b}_{\rm Quad}$ is 
    convex on $\mathcal{CE}_p(\mu_0,\mu_1;[0,1])$.
A stationary point of a convex functional on an affine set is a global minimizer; 
    thus $(\mu_t,V_t,h_t)_{t\in[0,1]}$ minimizes $\mathfrak{E}^{a,b}_{\rm Quad}$ among all feasible triples.

\medskip
\noindent$\bullet$ {\textbf{Step B.5. From minimality to constant speed geodesic.}}
By Corollary \ref{cor_geodesic}, action minimizers induce constant speed with respect to the metric $\al W_p^{a,b}$.
Hence $\mu$ is a constant–speed $\al W^{a,b}_p$–geodesic.

\medskip
Combining the steps proves the sufficiency.

 \section{Proofs of Step 3.3 in Theorem \ref{thm_duality} }\label{App-B-3.3}
Assume that $c=0$. 
We claim
    \begin{equation}
    \label{claim_step3.3}
    g_b^\ast(0,\sigma,V)=\sup_{(\phi,\Phi)\in \mathcal D^{b}}
    \left\{
    \int_0^1
    \big(
    \langle \dot\phi_t,\sigma_t\rangle_\varpi
    +
    \langle \Phi_t,V_t\rangle_\varpi
    \big)\di t
    \right\}\ge \frac{b^2}{2}\int_0^1 A'(\sigma_t,V_t)\di t=:\text{RHS}-\eqref{eq_Gb_Legendre_transform}. 
    \end{equation}
When $\int_0^1 A'(\sigma_t,V_t)\di t=\infty$, the claim becomes
    $ g_b^\ast(0,\sigma,V)=+\infty $.
We split the proof into four parts and deal with the case $\int_0^1 A'(\sigma_t,V_t)\di t<\infty$ in the first three steps.
\newline

 \noindent$\circ$ \textbf{Part 1. Construction of $(\phi^{\delta,\ep},\Phi^{\delta,\ep})\in \al D^{b}$.}
We now construct near-optimal test pairs in $\mathcal D^b$ to obtain the matching lower bound.
For $0<\ep<\tfrac12$, define
    \begin{equation*}
    (\tilde\sigma_t^\ep,\tilde V_t^\ep):=
    \begin{cases}
    (0,0), & t\in[-\ep,0),\\
    \big(\sigma_t,V_t\big), & t\in[0,1),\\
    (0,0), & t\in[1,1+\ep].
    \end{cases}
    \end{equation*}
Then we set
    \begin{equation}
    \label{ineq_HJ_duality_convolution}
    (\sigma_t^\ep,V_t^\ep):=\int_{-\ep}^{\ep} \zeta^\ep(s)\,(\tilde\sigma_{t+s}^\ep,\tilde V_{t+s}^\ep) \di s,
    \end{equation}
    where $\zeta^\ep:\mathbb{R}\to\mathbb{R}_+$ is smooth, supported in $[-\ep,\ep]$, strictly positive on $(-\ep,\ep)$, and satisfies $\int \zeta^\ep(s)\,\di s=1$.
The convolution yields the convergences
    \begin{equation*}
    \sigma^\ep\to \sigma \quad \text{uniformly on }[0,1],
    \qquad
    V^\ep\to V \quad \text{in }\mathrm L^2([0,1];\mathscr G_{\mathbf 1}),
    \end{equation*}
    as $\ep\to 0$.
In particular, after passing to a subsequence, $V_t^\ep\to V_t$ holds for a.e.\ $t\in(0,1)$.
Fix $\delta>0$ and set 
    \begin{equation*}
    \sigma_t^{\delta,\ep}:=\sigma_t^\ep+\delta\mathbf 1.
    \end{equation*}
The role of the time convolution is to meet the regularity requirements of $\bb E$,
and the role of adding $\delta\mathbf 1$ is to ensure strict positivity so that $\theta$ is differentiable and $\hat\sigma$ is bounded away from zero.
    
By convexity of the action and the monotonicity of the mean $\theta$, one has
    \begin{equation}
    \label{ineq_HJ_duality_2}
    \int_0^1 A'(\sigma_t^{\delta,\ep},V_t^\ep)\di t
    \le
    \int_0^1 A'(\sigma_t^\ep,V_t^\ep)\di t
    \le
    \int_0^1 A'(\sigma_t,V_t)\di t.
    \end{equation}
Since $A'$ is convex and lower semicontinuous by Lemma \ref{lem_alpha_convex}, 
    the same holds for the time integral:
    \begin{equation*}
    \int_0^1 A'(\sigma_t,V_t)\di t \le \varliminf_{\delta\to 0}\varliminf_{\ep\to 0}\int_0^1 A'(\sigma_t^{\delta,\ep},V_t^\ep)\di t.
    \end{equation*}
See \cite[Theorem 3.4.3]{buttazzo1989semicontinuity} for a general lower semicontinuity theorem for convex integral functionals.
Combining with \eqref{ineq_HJ_duality_2},
    \begin{equation*}
    \int_0^1 A'(\sigma_t,V_t)\di t
    =
    \lim_{\delta\to 0}\lim_{\ep\to 0}\int_0^1 A'(\sigma_t^{\delta,\ep},V_t^\ep)\di t.
    \end{equation*}

For each $(\delta,\ep)$ we define the near-optimal pair $(\phi^{\delta,\ep},\Phi^{\delta,\ep})$ by
    \begin{equation*}
    \Phi_t^{\delta,\ep}(x,y)
    :=
    b^2\frac{V_t^\ep(x,y)}{\hat\sigma_t^{\delta,\ep}(x,y)},
    \qquad
    \dot\phi_t^{\delta,\ep}(x)
    =
    -\frac{1}{2b^2}\sum_{y\in \al X}
    \partial_1\theta(\sigma_t^{\delta,\ep}(x),\sigma_t^{\delta,\ep}(y))
    \Phi_t^{\delta,\ep}(x,y)^2
    K(x,y),
    \end{equation*}
    and choosing $\phi_0^{\delta,\ep}=0$.
Since $\sigma_t^{\delta,\ep}(x)\ge \delta>0$,
    we have $\hat\sigma_t^{\delta,\ep}(x,y)>0$
    and $\theta$ differentiable at $(\sigma_t^{\delta,\ep}(x),\sigma_t^{\delta,\ep}(y))$,
    hence $\Phi_t^{\delta,\ep}$ and $\dot\phi_t^{\delta,\ep}(x)$ is well-defined.
Moreover, for fixed $\delta>0$ we have $\sigma_t^{\delta,\ep}\to \sigma_t^\delta:=\sigma_t+\delta\mathbf 1$ and $V_t^\ep\to V_t$
    for a.e.\ $t$ along a subsequence, 
    and hence we define
    \begin{equation}
    \label{eq_Phi_delta_eps_for_infty}
    \dot\phi_t^{\delta}(x)
    =
    -\frac{1}{2b^2}\sum_{y\in \al X}
    \partial_1\theta(\sigma_t^\delta(x),\sigma_t^\delta(y))
    \Phi_t^{\delta}(x,y)^2
    K(x,y),
    \qquad
    \Phi_t^{\delta}(x,y)=b^2\frac{V_t(x,y)}{\hat\sigma_t^\delta(x,y)}.
    \end{equation}
We shall use the inequality
    \begin{equation}
    \label{ineq_HJ_Euler}
    \partial_1\theta(s,t)u+\partial_2\theta(s,t)v\ge \theta(u,v),
    \qquad
    u,v\ge 0,
    \quad
    s,t>0,
    \end{equation}
    which follows from the concavity of $\theta$ and its $1$-homogeneity.
Equality holds in \eqref{ineq_HJ_Euler} when $(u,v)=(s,t)$.

We claim that $(\phi^{\delta,\ep},\Phi^{\delta,\ep})\in \al D^{b}$.
For any $\mu\in \al M$ we obtain
    \begin{equation*}
    \label{ineq_HJ_duality_1}
    \begin{aligned}
    \langle \dot\phi_t^{\delta,\ep},\mu\rangle_\varpi
    &=
    -\frac{1}{2b^2}\sum_{x,y\in \al X}
    \partial_1\theta(\sigma_t^{\delta,\ep}(x),\sigma_t^{\delta,\ep}(y))
    \Phi_t^{\delta,\ep}(x,y)^2
    K(x,y)\mu(x)\varpi(x)
    \\
    &=
    -\frac{1}{4b^2}\sum_{x,y\in \al X}
    \Big(
    \partial_1\theta(\sigma_t^{\delta,\ep}(x),\sigma_t^{\delta,\ep}(y))\mu(x)
    +
    \partial_2\theta(\sigma_t^{\delta,\ep}(x),\sigma_t^{\delta,\ep}(y))\mu(y)
    \Big)
    \Phi_t^{\delta,\ep}(x,y)^2
    K(x,y)\varpi(x)
    \\
    &\le
    -\frac{1}{4b^2}\sum_{x,y\in \al X}
    \theta(\mu(x),\mu(y))
    \Phi_t^{\delta,\ep}(x,y)^2
    K(x,y)\varpi(x)
    =
    -\frac{1}{2b^2}\|\Phi_t^{\delta,\ep}\|_\mu^2,
    \end{aligned}
    \end{equation*}
    which yields the desired membership $(\phi^{\delta,\ep},\Phi^{\delta,\ep})\in \al D^{b}$.
    \vspace{0.2cm}

\noindent$\circ$ \textbf{Part 2. Contribution of $\langle \dot\phi,\sigma\rangle_\varpi$.}
Assuming $\int_0^1 A'(\sigma_t,V_t) \di t<\infty$,
    we next identify the contribution of the $\langle \dot\phi,\sigma\rangle_\varpi$ term.
The strategy is:
    for each fixed $\delta>0$ we first send $\ep\to0$,
    using that $\sigma^{\delta,\ep}$ stays uniformly away from zero and that $V^\ep\to V$ in $\mathrm L^2$;
    we then send $\delta\downarrow0$;
    the resulting limit reproduces the factor $-(b^2/2)\int_0^1 A'(\sigma_t,V_t)\di t$,
    up to a nonnegative remainder term that vanishes after integration in time.
The goal is to prove
    \begin{equation}
    \label{eq_dowble_swap_int}
    \lim_{\delta\to 0}\lim_{\ep\to 0}\int_0^1 \langle \dot\phi_t^{\delta,\ep},\sigma_t\rangle_\varpi\di t
    =
    -\frac{b^2}{2}\int_0^1 A'(\sigma_t,V_t) \di t
    .
    \end{equation}
    
Fix $\delta>0$.
We first prove that
    \begin{equation}
    \label{eq_swap_eps_int}
    \lim_{\ep\to 0}\int_0^1 \langle \dot\phi_t^{\delta,\ep},\sigma_t\rangle_\varpi \di t
    =
    \int_0^1 \lim_{\ep\to 0}\langle \dot\phi_t^{\delta,\ep},\sigma_t\rangle_\varpi \di t
    =
    \int_0^1 \langle \dot\phi_t^{\delta},\sigma_t\rangle_\varpi \di t.
    \end{equation}
For a.e.\ $t\in(0,1)$ we write
\begin{equation*}
\langle \dot\phi_t^{\delta,\ep},\sigma_t\rangle_\varpi
=
-\frac{b^2}{2}\sum_{x,y\in \al X}
c_{t,x,y}^{\delta,\ep}\,
V_t^\ep(x,y)^2\,
K(x,y)\varpi(x),
\end{equation*}
where
\begin{equation*}
c_{t,x,y}^{\delta,\ep}
:=
\frac{\partial_1\theta(\sigma_t^{\delta,\ep}(x),\sigma_t^{\delta,\ep}(y))\,\sigma_t(x)}
{\hat\sigma_t^{\delta,\ep}(x,y)^2}.
\end{equation*}
Let $\max\sigma:=\max_{t\in[0,1],x\in\al X}\sigma_t(x)$.
Since $\sigma_t^{\delta,\ep}(x)\in[\delta,\max\sigma+\delta]$ and $\hat\sigma_t^{\delta,\ep}(x,y)\ge \delta$,
and since $(s,t)\mapsto \partial_1\theta(s,t)$ is continuous on $(0,\infty)^2$,
there exists a constant $C_\delta>0$ such that
\begin{equation*}
|c_{t,x,y}^{\delta,\ep}|\le C_\delta
\qquad
\text{for all }x,y\in \al X,\ \text{for all }\ep\in(0,\tfrac12),\ \text{for a.e.\ }t\in(0,1).
\end{equation*}
Moreover, by the convolution construction we have $\sigma^\ep\to\sigma$ uniformly on $[0,1]$.
Hence $\sigma^{\delta,\ep}\to\sigma^\delta$ uniformly on $[0,1]$.
Therefore, for every fixed pair $(x,y)$,
\begin{equation*}
c_{t,x,y}^{\delta,\ep}\to c_{t,x,y}^{\delta}
\qquad
\text{for a.e.\ }t\in(0,1),
\end{equation*}
where
\begin{equation*}
c_{t,x,y}^{\delta}
:=
\frac{\partial_1\theta(\sigma_t^{\delta}(x),\sigma_t^{\delta}(y))\,\sigma_t(x)}
{\hat\sigma_t^{\delta}(x,y)^2}.
\end{equation*}
Next, since $V^\ep\to V$ in $\mathrm L^2([0,1];\mathscr G_{\mathbf 1})$,
the finite-dimensionality of $\mathscr G_{\mathbf 1}$ implies that for each $(x,y)$ with $K(x,y)>0$,
\begin{equation*}
V^\ep(\cdot,x,y)\to V(\cdot,x,y)
\quad
\text{in }\mathrm L^2,
\end{equation*}
and hence
\begin{equation*}
V^\ep(\cdot,x,y)^2\to V(\cdot,x,y)^2
\quad
\text{in }\mathrm L^1.
\end{equation*}
We claim that for each fixed $(x,y)$,
\begin{equation*}
c_{t,x,y}^{\delta,\ep}\,V_t^\ep(x,y)^2
\to
c_{t,x,y}^{\delta}\,V_t(x,y)^2
\quad
\text{in }\mathrm L^1.
\end{equation*}
Indeed, we decompose
\begin{equation*}
c_{t,x,y}^{\delta,\ep}\,V_t^\ep(x,y)^2
-
c_{t,x,y}^{\delta}\,V_t(x,y)^2
=
c_{t,x,y}^{\delta,\ep}\bigl(V_t^\ep(x,y)^2-V_t(x,y)^2\bigr)
+
\bigl(c_{t,x,y}^{\delta,\ep}-c_{t,x,y}^{\delta}\bigr)V_t(x,y)^2.
\end{equation*}
The first term converges to zero in $\mathrm L^1$ since
$|c_{t,x,y}^{\delta,\ep}|\le C_\delta$ and
$V^\ep(\cdot,x,y)^2\to V(\cdot,x,y)^2$ in $\mathrm L^1$.
For the second term, we use $c_{t,x,y}^{\delta,\ep}\to c_{t,x,y}^{\delta}$ a.e.\ and the bound
\begin{equation*}
\bigl| \bigl(c_{t,x,y}^{\delta,\ep}-c_{t,x,y}^{\delta}\bigr)V_t(x,y)^2 \bigr|
\le
2C_\delta V_t(x,y)^2,
\end{equation*}
where $V(\cdot,x,y)^2\in \mathrm L^1$ since $V(\cdot,x,y)\in \mathrm L^2$.
Thus dominated convergence yields $\mathrm L^1$ convergence of the second term as well.
Since $\al X$ is finite, summing over $(x,y)$ preserves $\mathrm L^1$ convergence.
Consequently,
\begin{equation*}
\langle \dot\phi^{\delta,\ep},\sigma\rangle_\varpi
\to
\langle \dot\phi^{\delta},\sigma\rangle_\varpi
\quad
\text{in }\mathrm L^1,
\end{equation*}
and therefore
\begin{equation*}
\lim_{\ep\to 0}\int_0^1 \langle \dot\phi_t^{\delta,\ep},\sigma_t\rangle_\varpi \di t
=
\int_0^1 \langle \dot\phi_t^{\delta},\sigma_t\rangle_\varpi \di t.
\end{equation*}
This proves \eqref{eq_swap_eps_int}.

We now prove the identity
    \begin{equation}
    \label{eq_target_identity_delta}
    \lim_{\delta\to 0}\int_0^1 \langle \dot\phi_t^{\delta},\sigma_t\rangle_\varpi \di t
    =
    -\frac{b^2}{2}\int_0^1 A'(\sigma_t,V_t)\di t.
    \end{equation}
We begin with a symmetrized identity.
For any $\mu\in \al M$ and a.e.\ $t$, reversibility of $\varpi$ yields
    \begin{align}
    \label{eq_symmetrized_pairing}
    \langle \dot\phi_t^{\delta},\mu\rangle_\varpi
    =
    -\frac{1}{4b^2}\sum_{x,y\in \al X}
    \Big(
    \partial_1\theta(\sigma_t^\delta(x),\sigma_t^\delta(y))\mu(x)
    +
    \partial_2\theta(\sigma_t^\delta(x),\sigma_t^\delta(y))\mu(y)
    \Big)
    \Phi_t^{\delta}(x,y)^2
    K(x,y)\varpi(x).
    \end{align}
We apply \eqref{eq_symmetrized_pairing} with $\mu=\sigma_t$.
Since $\sigma_t=\sigma_t^\delta-\delta\mathbf 1$,
    we use \eqref{ineq_HJ_Euler} in the identity case
    to obtain, for each pair $(x,y)$,
    \begin{align*}
    &\partial_1\theta(\sigma_t^\delta(x),\sigma_t^\delta(y))\sigma_t(x)
    +
    \partial_2\theta(\sigma_t^\delta(x),\sigma_t^\delta(y))\sigma_t(y)
    \\
    &\qquad
    =
    \hat\sigma_t^\delta(x,y)
    -
    \delta\Big(
    \partial_1\theta(\sigma_t^\delta(x),\sigma_t^\delta(y))
    +
    \partial_2\theta(\sigma_t^\delta(x),\sigma_t^\delta(y))
    \Big).
    \end{align*}
Substituting this into \eqref{eq_symmetrized_pairing} gives the decomposition
\begin{equation}
\label{eq_decomposition_pairing}
\langle \dot\phi_t^{\delta},\sigma_t\rangle_\varpi
=
\langle \dot\phi_t^{\delta},\sigma_t^\delta\rangle_\varpi
+
R_t^\delta,
\end{equation}
where the remainder term is
\begin{equation}
\label{eq_remainder_R}
R_t^\delta
=
\frac{\delta}{4b^2}\sum_{x,y\in \al X}
\Big(
\partial_1\theta(\sigma_t^\delta(x),\sigma_t^\delta(y))
+
\partial_2\theta(\sigma_t^\delta(x),\sigma_t^\delta(y))
\Big)
\Phi_t^{\delta}(x,y)^2
K(x,y)\varpi(x).
\end{equation}
Choosing $\mu=\sigma_t^\delta$ in \eqref{eq_symmetrized_pairing} yields the identity
\begin{equation}
\label{eq_identity_sigma_delta}
\langle \dot\phi_t^{\delta},\sigma_t^\delta\rangle_\varpi
=
-\frac{1}{4b^2}\sum_{x,y\in \al X}\hat\sigma_t^\delta(x,y)\Phi_t^\delta(x,y)^2K(x,y)\varpi(x)
=
-\frac{b^2}{2}A'(\sigma_t^\delta,V_t).
\end{equation}

We now justify that the remainder term in \eqref{eq_remainder_R} vanishes as $\delta\to 0$.
We work at times $t$ such that $A'(\sigma_t,V_t)<\infty$.
This holds for almost every $t\in (0,1)$ since $A'(\sigma_t,V_t)$ is integrable in time.
Fix such a time $t$.
For each ordered pair $(x,y)$ with $K(x,y)>0$, the finiteness of $A'(\sigma_t,V_t)$ implies that
    $V_t(x,y)=0$ whenever $\hat\sigma_t(x,y)=\theta(\sigma_t(x),\sigma_t(y))=0$.
Since $\theta(u,v)=0$ if and only if $u=0$ or $v=0$, this means that
    if $V_t(x,y)\ne 0$ then necessarily $\sigma_t(x)>0$ and $\sigma_t(y)>0$.
We rewrite the summand in \eqref{eq_remainder_R} as
    \begin{equation*}
    \frac{b^2}{4}
    \frac{\delta\bigl(\partial_1\theta(\sigma_t^\delta(x),\sigma_t^\delta(y))+\partial_2\theta(\sigma_t^\delta(x),\sigma_t^\delta(y))\bigr)}
    {\hat\sigma_t^\delta(x,y)^2}
    V_t(x,y)^2
    K(x,y)\varpi(x).
    \end{equation*}
If $V_t(x,y)=0$, then this term is identically zero for all $\delta>0$.
Assume now that $V_t(x,y)\ne 0$.
Then $\sigma_t(x)>0$ and $\sigma_t(y)>0,$ and therefore
    $\hat\sigma_t^\delta(x,y)\to \hat\sigma_t(x,y)>0$ as $\delta\to 0$.
Moreover, the map $(s,t)\mapsto \partial_1\theta(s,t)+\partial_2\theta(s,t)$ is continuous on $(0,\infty)^2$,
    so the quantity $\partial_1\theta(\sigma_t^\delta(x),\sigma_t^\delta(y))+\partial_2\theta(\sigma_t^\delta(x),\sigma_t^\delta(y))$
    remains bounded as $\delta\to 0$.
Consequently,
    \begin{equation*}
    \frac{\delta\bigl(\partial_1\theta(\sigma_t^\delta(x),\sigma_t^\delta(y))
        +\partial_2\theta(\sigma_t^\delta(x),\sigma_t^\delta(y))\bigr)}
    {\hat\sigma_t^\delta(x,y)^2}
    \to 0
    \qquad
    \text{as }\delta\to 0.
    \end{equation*}
Therefore the summand converges to zero as $\delta\to 0$ for every pair $(x,y)$.
Since $\al X$ is finite, we conclude that
\begin{equation}
\label{eq_R_pointwise_zero}
R_t^\delta\to 0
\qquad
\text{for almost every }t\in (0,1).
\end{equation}

We keep the domination estimate.
Using the Euler identity
$\partial_1\theta(s,t)s+\partial_2\theta(s,t)t=\theta(s,t)$ on $(0,\infty)^2$
and the bounds $\delta\le \sigma_t^\delta(x)$ and $\delta\le \sigma_t^\delta(y)$, we obtain
\begin{equation*}
\delta\bigl(\partial_1\theta(\sigma_t^\delta(x),\sigma_t^\delta(y))+\partial_2\theta(\sigma_t^\delta(x),\sigma_t^\delta(y))\bigr)
\le
\hat\sigma_t^\delta(x,y).
\end{equation*}
Substituting this into \eqref{eq_remainder_R} yields
\begin{equation*}
0\le R_t^\delta
\le
\frac{b^2}{4}\sum_{x,y\in \al X}\frac{V_t(x,y)^2}{\hat\sigma_t^\delta(x,y)}K(x,y)\varpi(x)
=
\frac{b^2}{2}A'(\sigma_t^\delta,V_t).
\end{equation*}
By the monotonicity of $\theta$, we have $\hat\sigma_t^\delta(x,y)\ge \hat\sigma_t(x,y)$, and hence
$A'(\sigma_t^\delta,V_t)\le A'(\sigma_t,V_t)$.
Therefore $R_t^\delta$ is dominated by the integrable function $(b^2/2)A'(\sigma_t,V_t)$.
Combining this domination with \eqref{eq_R_pointwise_zero}, the dominated convergence theorem gives
\begin{equation}
\label{eq_R_to_zero}
\lim_{\delta\to 0}\int_0^1 R_t^\delta \di t=0.
\end{equation}

Finally, combining \eqref{eq_decomposition_pairing}, \eqref{eq_identity_sigma_delta}, and \eqref{eq_R_to_zero} yields
\begin{align*}
\lim_{\delta\to 0}\int_0^1 \langle \dot\phi_t^{\delta},\sigma_t\rangle_\varpi \di t
&=
\lim_{\delta\to 0}\int_0^1 \langle \dot\phi_t^{\delta},\sigma_t^\delta\rangle_\varpi \di t
+
\lim_{\delta\to 0}\int_0^1 R_t^\delta\di t
\\
&=
-\frac{b^2}{2}\lim_{\delta\to 0}\int_0^1 A'(\sigma_t^\delta,V_t)\di t.
\end{align*}
Since $A'(\sigma_t^\delta,V_t)\uparrow A'(\sigma_t,V_t)$ as $\delta\to 0$ and
$A'(\sigma_t^\delta,V_t)\le A'(\sigma_t,V_t)$, we conclude that
\begin{equation*}
\lim_{\delta\to 0}\int_0^1 A'(\sigma_t^\delta,V_t)\di t
=
\int_0^1 A'(\sigma_t,V_t)\di t,
\end{equation*}
and \eqref{eq_target_identity_delta} follows.
Together with \eqref{eq_swap_eps_int}, this proves \eqref{eq_dowble_swap_int}.
\vspace{0.2cm}

\noindent$\circ$ \textbf{Part 3. Contribution of $\langle \Phi,V\rangle_\varpi$.}
Assuming $\int_0^1 A'(\sigma_t,V_t) \di t<\infty$,
    we finally identify the contribution of the $\langle \Phi,V\rangle_\varpi$ term.
The strategy is:
    for fixed $\delta>0$ we pass to the limit $\ep\to0$ in the pairing,
    using the strong convergence $\Phi^{\delta,\ep}\to\Phi^\delta$ in $\mathrm L^2([0,1];\mathscr G_{\mathbf 1})$;
    the resulting expression is explicit and equals $b^2A'(\sigma_t^\delta,V_t)$ pointwise in time;
    sending $\delta\downarrow0$ and applying monotone convergence then yields \eqref{eq_V_Phi_limit}.
We claim
    \begin{equation}
    \label{eq_V_Phi_limit}
    \lim_{\delta\to 0}\lim_{\ep\to 0}\int_0^1 \langle V_t,\Phi_t^{\delta,\ep}\rangle_\varpi \di t
    =
    b^2\int_0^1 A'(\sigma_t,V_t)\di t.
    \end{equation}
We first fix $\delta>0$ and pass to the limit $\ep\to 0$.
Since $\sigma_t^{\delta,\ep}\ge \delta\mathbf 1$,
    we have $\hat\sigma_t^{\delta,\ep}(x,y)\ge \delta$,
    and therefore $\hat\sigma_t^{\delta,\ep}(x,y)^{-1}\le \delta^{-1}$.
Moreover, the time convolution yields
    $\sigma_t^{\delta,\ep}\to \sigma_t^\delta$ uniformly in $t$,
    and $V^\ep\to V$ in $\mathrm L^2([0,1];\mathscr G_{\mathbf 1})$.
Because $\al X$ is finite and $\theta$ is continuous on $(0,\infty)^2$,
    we also have $\hat\sigma^{\delta,\ep}\to \hat\sigma^\delta$ uniformly on $[0,1]\times \al X\times \al X$.
Consequently,
    \begin{equation*}
    \Phi^{\delta,\ep}\to \Phi^\delta
    \qquad
    \text{in }\mathrm L^2([0,1];\mathscr G_{\mathbf 1}).
    \end{equation*}
Using Cauchy--Schwarz in time and in the $\langle\cdot,\cdot\rangle_\varpi$ inner product, we obtain
    \begin{equation*}
    \left|
    \int_0^1 \langle V_t,\Phi_t^{\delta,\ep}-\Phi_t^\delta\rangle_\varpi \di t
    \right|
    \le
    \left(\int_0^1 \|V_t\|_\varpi^2\di t\right)^{1/2}
    \left(\int_0^1 \|\Phi_t^{\delta,\ep}-\Phi_t^\delta\|_\varpi^2\di t\right)^{1/2},
    \end{equation*}
    and the right-hand side converges to $0$ as $\ep\to 0$.
Therefore,
    \begin{equation}
    \label{eq_eps_to_0_VPhi}
    \lim_{\ep\to 0}\int_0^1 \langle V_t,\Phi_t^{\delta,\ep}\rangle_\varpi \di t
    =
    \int_0^1 \langle V_t,\Phi_t^\delta\rangle_\varpi \di t.
    \end{equation}

We now compute $\langle V_t,\Phi_t^\delta\rangle_\varpi$.
By definition of the pairing,
    \begin{align*}
    \langle V_t,\Phi_t^\delta\rangle_\varpi
    &=
    \frac12\sum_{x,y\in \al X} V_t(x,y)\Phi_t^\delta(x,y)K(x,y)\varpi(x)
    \\
    &=
    \frac{b^2}{2}\sum_{x,y\in \al X}\frac{V_t(x,y)^2}{\hat\sigma_t^\delta(x,y)}K(x,y)\varpi(x)
    =
    b^2 A'(\sigma_t^\delta,V_t).
    \end{align*}
Combining this with \eqref{eq_eps_to_0_VPhi} yields
    \begin{equation}
    \label{eq_fixed_delta_identity}
    \lim_{\ep\to 0}\int_0^1 \langle V_t,\Phi_t^{\delta,\ep}\rangle_\varpi \di t
    =
    b^2\int_0^1 A'(\sigma_t^\delta,V_t)\di t.
    \end{equation}

It remains to pass to the limit $\delta\to 0$.
For a.e. $t$, the monotonicity of $\theta$ implies $\hat\sigma_t^\delta(x,y)\ge \hat\sigma_t(x,y)$,
    hence $A'(\sigma_t^\delta,V_t)\le A'(\sigma_t,V_t)$.
Moreover, for each fixed $(x,y)$,
    the quantity $1/\hat\sigma_t^\delta(x,y)$ increases as $\delta\downarrow 0$ and converges to $1/\hat\sigma_t(x,y)$.
On pairs with $\hat\sigma_t(x,y)=0$, finiteness of $A'(\sigma_t,V_t)$ implies $V_t(x,y)=0$,
    so the corresponding summand is identically $0$ for all $\delta>0$.
Therefore, for a.e. $t$ we have
    \begin{equation*}
    A'(\sigma_t^\delta,V_t)\uparrow A'(\sigma_t,V_t)
    \qquad
    \text{as }\delta\downarrow 0.
    \end{equation*}
Since $A'(\sigma_t^\delta,V_t)$ is nonnegative, the monotone convergence theorem gives
    \begin{equation}
    \label{eq_delta_MCT_Aprime}
    \lim_{\delta\to 0}\int_0^1 A'(\sigma_t^\delta,V_t)\di t
    =
    \int_0^1 A'(\sigma_t,V_t)\di t.
    \end{equation}
Finally, \eqref{eq_fixed_delta_identity} and \eqref{eq_delta_MCT_Aprime} imply \eqref{eq_V_Phi_limit}.
\vspace{0.2cm}

We combine $c=0$, \eqref{eq_dowble_swap_int} and \eqref{eq_V_Phi_limit} to derive \eqref{claim_step3.3} when $ \int_0^1 A'(\sigma_t, V_t) \di t <+\infty $.

\vspace{0.2cm}
\noindent$\circ$ \textbf{Part 4. If $\int_0^1 A'(\sigma_t, V_t) \di t =+\infty$, then $g^\ast_b(0,\sigma,V)=+\infty$.}
We now treat the remaining case $\int_0^1 A'(\sigma_t,V_t)\di t=+\infty$.
Here we do not seek an exact identification of $g_b^\ast(0,\sigma,V)$,
    but only show that it is infinite.
The strategy is:
    evaluate the support function of $\mathcal D^b$ on the same family of test pairs     $(\phi^{\delta,\ep},\Phi^{\delta,\ep})$,
    first letting $\ep\to0$ for fixed $\delta>0$ and then letting $\delta\downarrow0$,
    which yields a lower bound by $(b^2/2)\int_0^1 A'(\sigma_t^\delta,V_t)\di t$;
    since this integral diverges to $+\infty$ as $\delta\downarrow0$ by monotone convergence,
    we conclude that $g_b^\ast(0,\sigma,V)=+\infty$.

Assume that
\begin{equation}
\label{eq_assumption_Aprime_infty}
\int_0^1 A'(\sigma_t,V_t)\di t=+\infty.
\end{equation}
We prove that $g_b^\ast(0,\sigma,V)=+\infty$.

By the definition of the Legendre--Fenchel transform,
    \begin{equation}
    \label{eq_gb_star_sup_test_pair}
    g_b^\ast(0,\sigma,V)
    =
    \sup_{(\phi,\Phi)\in\mathcal D^b}
    \int_0^1\bigl(\langle \dot\phi_t,\sigma_t\rangle_\varpi+\langle \Phi_t,V_t\rangle_\varpi\bigr)\di t
    \ge
    \int_0^1\bigl(\langle \dot\phi_t^{\delta,\ep},\sigma_t\rangle_\varpi+\langle \Phi_t^{\delta,\ep},V_t\rangle_\varpi\bigr)\di t.
    \end{equation}
We next let $\ep\to 0$ for fixed $\delta$.
As the earlier convergence results \eqref{eq_swap_eps_int} and \eqref{eq_eps_to_0_VPhi} do not rely on $\int_0^1 A'(\sigma_t,V_t)\di t<+\infty$, 
    the convolution approximation yields
    \begin{equation*}
    \label{eq_eps_to_0_pairings}
    \lim_{\ep\to 0}\int_0^1 \langle \Phi_t^{\delta,\ep},V_t\rangle_\varpi\di t
    =
    \int_0^1 \langle \Phi_t^\delta,V_t\rangle_\varpi\di t,
    \qquad
    \lim_{\ep\to 0}\int_0^1 \langle \dot\phi_t^{\delta,\ep},\sigma_t\rangle_\varpi\di t
    =
    \int_0^1 \langle \dot\phi_t^\delta,\sigma_t\rangle_\varpi\di t.
    \end{equation*}
Taking $\liminf_{\ep\to 0}$ in \eqref{eq_gb_star_sup_test_pair} and using the above formulae, we obtain
    \begin{equation}
    \label{eq_gb_star_ge_limit_eps}
    g_b^\ast(0,\sigma,V)
    \ge
    \int_0^1\bigl(\langle \dot\phi_t^\delta,\sigma_t\rangle_\varpi+\langle \Phi_t^\delta,V_t\rangle_\varpi\bigr)\di t.
    \end{equation}
We now evaluate the right-hand side in \eqref{eq_gb_star_ge_limit_eps}.
By the definition of $\Phi_t^\delta$,
    \begin{equation*}
    \label{eq_V_Phi_delta_identity}
    \langle \Phi_t^\delta,V_t\rangle_\varpi=b^2 A'(\sigma_t^\delta,V_t).
    \end{equation*}
Moreover, the decomposition \eqref{eq_decomposition_pairing} established earlier yields
    \begin{equation*}
    \label{eq_phi_dot_sigma_decomposition_delta}
    \langle \dot\phi_t^\delta,\sigma_t\rangle_\varpi
    =
    -\frac{b^2}{2}A'(\sigma_t^\delta,V_t)+R_t^\delta,
    \qquad
    R_t^\delta\ge 0.
    \end{equation*}
Combining these equations gives
    \begin{equation}
    \label{eq_gb_star_lower_bound_delta}
    g_b^\ast(0,\sigma,V)
    \ge
    \int_0^1\left(\frac{b^2}{2}A'(\sigma_t^\delta,V_t)+R_t^\delta\right)\di t
    \ge
    \frac{b^2}{2}\int_0^1 A'(\sigma_t^\delta,V_t)\di t.
    \end{equation}
Finally, we let $\delta\downarrow 0$.
For a.e. $t\in(0,1)$, the monotonicity of $\theta$ implies $\hat\sigma_t^\delta(x,y)\downarrow \hat\sigma_t(x,y)$ for each pair $(x,y)$,
and hence
\begin{equation*}
\label{eq_Aprime_monotone_delta}
A'(\sigma_t^\delta,V_t)\uparrow A'(\sigma_t,V_t)
\qquad
\text{as }\delta\downarrow 0.
\end{equation*}
Since $A'(\sigma_t^\delta,V_t)\ge 0$, the monotone convergence theorem and \eqref{eq_assumption_Aprime_infty} yield
    \begin{equation}
    \label{eq_Aprime_integral_to_infty}
    \lim_{\delta\downarrow 0}\int_0^1 A'(\sigma_t^\delta,V_t)\di t
    =
    \int_0^1 A'(\sigma_t,V_t)\di t
    =
    +\infty.
    \end{equation}
Combining \eqref{eq_gb_star_lower_bound_delta} with \eqref{eq_Aprime_integral_to_infty} shows that $g_b^\ast(0,\sigma,V)$ is unbounded above,
    and therefore $g_b^\ast(0,\sigma,V)=+\infty$.

\bibliographystyle{plain}
\bibliography{ref}

\end{document}